\documentclass[11pt]{amsart}
\usepackage[utf8]{inputenc}
\newcommand{\VAN}[3]{#2}

\usepackage{csquotes}
\usepackage[dvipsnames]{xcolor}
\usepackage{xparse}
\usepackage{tikz-cd}
\usepackage[font=small,format=plain,labelfont=bf,up,textfont=it,up]{caption}
\usepackage{adjustbox}
\usepackage[margin=1.5in]{geometry}

\usepackage{fancyhdr}
\usepackage[shortlabels]{enumitem}
\usepackage{amsmath}
\usepackage{mathtools} 
\usepackage{relsize}
\usepackage[bbgreekl]{mathbbol}
\usepackage{amsfonts}
\usepackage{amssymb} 
\usepackage{amsthm}
\usepackage{mathrsfs}
\usepackage{calligra}

\DeclareSymbolFontAlphabet{\mathbb}{AMSb}
\DeclareSymbolFontAlphabet{\mathbbl}{bbold}
\newcommand{\prism}{{\mathlarger{\mathbbl{\Delta}}}}
\usepackage[linktocpage=true, hyperindex,pagebackref]{hyperref}

\mathtoolsset{showonlyrefs}
\usepackage[OT2,T1]{fontenc}

\begingroup
    \makeatletter
    \@for\theoremstyle:=definition,remark,plain\do{%
        \expandafter\g@addto@macro\csname th@\theoremstyle\endcsname{%
            \addtolength\thm@preskip\parskip
            }%
        }
\endgroup

\newtheorem{Thm}[subsubsection]{Theorem}
\newtheorem{Lem}[subsubsection]{Lemma}
\newtheorem{Prop}[subsubsection]{Proposition}
\newtheorem{Cor}[subsubsection]{Corollary}

\newtheorem{Conj}[subsubsection]{Conjecture}
\newtheorem*{Thm*}{Theorem}
\newtheorem{mainThm}{Theorem}
\newtheorem{Claim}[subsubsection]{Claim}

\newtheorem{mainCor}[mainThm]{Corollary}

\theoremstyle{definition}
\newtheorem{Def}[subsubsection]{Definition}

\newtheorem{Rem}[subsubsection]{Remark}

\numberwithin{equation}{subsection}

\newcommand{\spec}{\operatorname{Spec}}
\newcommand{\spf}{\operatorname{Spf}}
\newcommand{\spa}{\operatorname{Spa}}
\newcommand{\spd}{\operatorname{Spd}}
\newcommand{\gal}{\operatorname{Gal}}
\newcommand{\ovfp}{\overline{\mathbb{F}}_p}

\newcommand{\qpbr}{\breve{\mathbb{Q}}_{p}}
\newcommand{\ebreve}{\breve{E}}
\newcommand{\oebreve}{\mathcal{O}_{\breve{E}}}
\newcommand{\zpbr}{\breve{\mathbb{Z}}_{p}}

\newcommand{\afp}{{\mathbb{A}_f^p}}

\newcommand{\CO}{\mathcal{O}}
\newcommand{\calg}{\mathcal{G}}

\newcommand{\qp}{{\mathbb{Q}_p}}
\newcommand{\zp}{{\mathbb{Z}_{p}}}
\newcommand{\fp}{{\mathbb{F}_{p}}}
\newcommand{\fpbar}{{\overline{\mathbb{F}}_{p}}}

\newcommand{\qpbar}{\overline{\mathbb{Q}}_p}

\newcommand{\perf}{\operatorname{Perf}}

\newcommand{\gxg}{{(\mathsf{G}, \mathsf{X}, \mathcal{G})}}
\newcommand{\gxgp}{{(\mathsf{G}', \mathsf{X}', \mathcal{G}')}}

\newcommand{\calgcirc}{\mathcal{G}^{\circ}}
\newcommand{\calgdeltacirc}{\mathcal{G}_{\delta}^{\circ}}
\newcommand{\calgdelta}{\mathcal{G}_{\delta}}

\newcommand{\shtgcircmu}{\mathrm{Sht}_{\calgcirc,\mu}}
\newcommand{\shtgdeltamu}{\mathrm{Sht}_{\calgdelta,\mu, \oebreve}}
\newcommand{\shtgdeltacircmu}{\mathrm{Sht}_{\calgdeltacirc,\mu, \oebreve}}

\newcommand{\shtgmubr}{\mathrm{Sht}_{\mathcal{G},\mu, \oebreve}}

\newcommand{\mintgdeltacircmu}{\mathcal{M}^{\mathrm{int}}_{\mathcal{G}_{\delta}^{\circ},b,\mu}}
\newcommand{\shtgmuone}{\mathrm{Sht}_{\mathcal{G},\mu, \delta=1}}
\newcommand{\shtgmukappa}{\mathrm{Sht}_{\mathcal{G},\mu}^{\kappa=-\mu^\natural}}
\newcommand{\shtgmubrkappa}{\mathrm{Sht}_{\mathcal{G},\mu,\oebreve}^{\kappa=-\mu^\natural}}

\newcommand{\shtg}{\mathrm{Sht}_{\mathcal{G}}}
\newcommand{\shtgmu}{\mathrm{Sht}_{\mathcal{G},\mu}}

\newcommand{\shthmuone}{\mathrm{Sht}_{\mathcal{H},\mu, \delta=1}}
\newcommand{\shtgmub}{\mathrm{Sht}_{\mathcal{G},\mu}^{[b]}}

\newcommand{\shtgmup}{\mathrm{Sht}_{\mathcal{G}^\prime,\mu'}}

\newcommand{\mintgmu}{\mathcal{M}^{\mathrm{int}}_{\mathcal{G},b,\mu}}

\newcommand{\mintxgmu}{\mathcal{M}^{\mathrm{int}}_{\mathcal{G},b_x,\mu}}
\newcommand{\mintxgcircmu}{\mathcal{M}^{\mathrm{int}}_{\calgcirc,b_x,\mu}}

\newcommand{\admu}{\operatorname{Adm}(\mu^{-1})}
\newcommand{\bun}{\mathrm{Bun}}
\newcommand{\bung}{\bun_{G}}
\newcommand{\bungmu}{\bun_{G,\mu^{-1}}}

\newcommand{\gafp}{\mathsf{G}(\afp)}

\newcommand{\gx}{{(\mathsf{G}, \mathsf{X})}}

\newcommand{\gv}{\mathsf{G}_{V}}

\newcommand{\gvx}{(\mathsf{G}_{V},\mathsf{H}_{V})}

\newcommand{\gxp}{(\mathsf{G}', \mathsf{X}')}
\newcommand{\gp}{\mathsf{G}'}

\newcommand{\scrs}{\mathscr{S}}

\newcommand{\scrsg}{\mathscr{S}_K\gx}

\newcommand{\locmodgmu}{\mathbb{M}_{\mathcal{G},\mu}}
\newcommand{\locmodgcircmu}{\mathbb{M}_{\calgcirc,\mu}}
\newcommand{\locmodgcircmuv}{\mathbb{M}_{\calgcirc,\mu}^{\mathrm{v}}}
\newcommand{\locmodgmuv}{\mathbb{M}_{\mathcal{G},\mu}^{\mathrm{v}}}

\newcommand{\bdtimes}{\buildrel{\boldsymbol{.}}\over\times}
\newcommand{\frob}{\mathrm{Frob}}

\newcommand{\bgmu}{B(G,\mu^{-1})}
\newcommand{\Hom}{\operatorname{Hom}}
\newcommand{\Aut}{\operatorname{Aut}}

\newcommand{\conj}[1]{\operatorname{Int} #1}

\newcommand{\g}{\mathsf{G}}
\newcommand{\mintgcircbmu}[1]{\mathcal{M}^\mathrm{int}_{\calgcirc,b,\mu,#1}}
\newcommand{\mintgcircbxmu}[1]{\mathcal{M}^{\mathrm{int}}_{\calgcirc,b_x,\mu,#1}}
\newcommand{\mintgbmu}[1]{\mathcal{M}^\mathrm{int}_{\mathcal{G},b,\mu,#1}}
\newcommand{\mintgbxmu}[1]{\mathcal{M}^\mathrm{int}_{\mathcal{G},b_x,\mu,#1}}
\newcommand{\mintgbmuone}[1]{\mathcal{M}^\mathrm{int}_{\mathcal{G},b,\mu,\delta=1,#1}}
\newcommand{\mintgbxmuone}[1]{\mathcal{M}^\mathrm{int}_{\mathcal{G},b_x,\mu,\delta=1,#1}}

\newcommand{\oeel}[1]{{W_{\mathcal{O}_E,#1}(\ell)}}

\newcommand{\stacks}[1]{\cite[\href{https://stacks.math.columbia.edu/tag/#1}{Tag~#1}]{stacks-project}}

\newcommand\restr[2]{{
  \left.\kern-\nulldelimiterspace 
  #1 
  \vphantom{\big\vert} 
  \right\rvert_{#2} 
  }}

\author{Patrick Daniels}
\address{Department of Mathematics and Statistics, Skidmore College, 815 N Broadway, Saratoga Springs, NY, 12866, USA}
\email{pdaniels@skidmore.edu}

\author{Pol van Hoften} 
\address{School of Mathematical Sciences, Zhejiang University, 866 Yuhangtang Rd, Hangzhou, 310058, P. R. China}
\email{pvhoften@zju.edu.cn}

\author{Dongryul Kim}
\address{Department of Mathematics, Stanford University, 450 Jane Stanford Way
(Building 380), Stanford, California, USA}
\email{dkim04@stanford.edu}

\author{Mingjia Zhang}
\address{Department of Mathematics, Princeton university, Fine Hall, Washington Road,
Princeton, NJ, 08544-1000, USA}
\email{mz9413@princeton.edu}
\thanks{MZ is funded by the Oswald Veblen Fund through Princeton University.}

\iftrue
\newcommand{\commentDaniel}[1]{\textcolor{Blue}{Daniel: #1}}
\newcommand{\commentPatrick}[1]{\textcolor{violet}{Patrick: #1}}
\newcommand{\commentPol}[1]{\textcolor{red}{Pol: #1}}
\newcommand{\commentMingjia}[1]{\textcolor{ForestGreen}{Mingjia: #1}}
\else
\newcommand{\commentDaniel}[1]{}
\newcommand{\commentPatrick}[1]{}
\newcommand{\commentPol}[1]{}
\newcommand{\commentMingjia}[1]{}
\fi

\title{On a conjecture of Pappas and Rapoport}

\begin{document}
\sloppy 
\begin{abstract}
We prove a conjecture of Pappas and Rapoport about the existence of ``canonical'' integral models of Shimura varieties of Hodge type with quasi-parahoric level structure at a prime $p$. For these integral models, we moreover show uniformization of isogeny classes by integral local Shimura varieties and prove a conjecture of Kisin and Pappas on local model diagrams. 
\end{abstract}
\maketitle 
\setcounter{tocdepth}{1}
\tableofcontents

\section{Introduction}

\subsection{Background} Fix a prime $p$. Pappas and Rapoport have recently determined conditions which uniquely characterize $p$-adic integral models of Shimura varieties with parahoric level at $p$ \cite{PappasRapoportShtukas}. Integral models satisfying these conditions are called \textit{canonical integral models}, and Pappas and Rapoport have conjectured the existence of such models in general. Moreover, they prove the conjecture for Shimura varieties of Hodge type, under the assumption that the level subgroup $K_p$ at $p$ is a stabilizer parahoric (see Definition~\ref{Def:StabilizerParahoric}). In this article, we prove the existence of canonical integral models for Hodge-type Shimura varieties with arbitrary parahoric level at $p$ and, more generally, with quasi-parahoric level at $p$.  

When the level subgroup at $p$ is hyperspecial, a collection of smooth integral models for a given Shimura variety can be uniquely characterized by an extension property, similar to the N\'eron mapping property, see \cite{MilnePoints}, \cite{MoonenModels}. In this case Kisin has constructed smooth integral models satisfying the extension property for Shimura varieties of abelian type \cite{KisinModels}. In this article, we are most interested in the case where the level subgroup at $p$ is (more generally) parahoric in the sense of \cite{BTII}. In such cases, even the most accessible Shimura varieties (for example, the Siegel modular varieties) have integral models with complicated singularities, see e.g., \cite{RapoportReduction}, and such models are not so easily characterized.

The key innovation of Pappas and Rapoport in \cite{PappasRapoportShtukas}, building on earlier work of Pappas (see \cite{PappasIntegralModels}), was that integral models of Shimura varieties can be characterized by the existence of a universal $p$-adic shtuka (in the sense of \cite{ScholzeWeinsteinBerkeley}) which satisfies certain compatibilities. In this article we work in reverse, in a sense. We take as a starting point the notion that a shtuka should exist over some integral model of the given Shimura variety at (quasi-)parahoric level, and that such a shtuka should be compatible with transition morphisms between varying levels. Following these ideas, we first define a v-sheaf supporting a universal shtuka, which we then show is the v-sheaf associated to an integral model of the given Shimura variety at parahoric level. We explain our results and methods in more detail below. 

\subsection{Main Results} Let $\gx$ be a Shimura datum with reflex field $\mathsf{E}$. Let $p$ be a prime number, let $v$ be a prime of $\mathsf{E}$ above $p$ and let $E$ be the $v$-adic completion of $\mathsf{E}$ with ring of integers $\mathcal{O}_E$ and residue field $k_E$. We write $G=\mathsf{G} \otimes \qp$ and let $K_p \subset G(\qp)$ be a parahoric subgroup. For $K^p \subset \gafp$ a neat compact open subgroup, we write $K=K^pK_p$. We denote by $\mathbf{Sh}_{K}\gx/\spec(E)$ the base change to $E$ of the canonical model of the Shimura variety at level $K$ over $\spec(\mathsf{E})$. 

We will consider systems $\{\scrs_{K}\gx\}_{K^p}$ of normal schemes $\scrs_{K}\gx$, flat, of finite type, and separated over $\mathcal{O}_E$, with generic fibers $\mathbf{Sh}_{K}\gx$; here $K^p$ runs over all neat compact open subgroups of $\gafp$. Pappas and Rapoport give axioms for such systems, see \cite[Conjecture 4.2.2]{PappasRapoportShtukas}, and show that systems satisfying their axioms are unique if they exist, see \cite[Theorem 4.2.4]{PappasRapoportShtukas}. They conjecture that systems satisfying their axioms always exist, see \cite[Conjecture 4.2.2]{PappasRapoportShtukas}. 

The conjecture of Pappas and Rapoport is known when $\mathsf{G}$ is a torus, see \cite{Daniels}, and, under the assumption that $p>2$, when $K_p$ is hyperspecial and $\gx$ is of abelian type, see \cite{ImaiKatoYoucis}. It is known moreover when $\gx$ is of Hodge type and $K_p$ is a stabilizer parahoric, see \cite[Theorem 4.5.2]{PappasRapoportShtukas}. Our main theorem extends this result to all parahoric subgroups.
\begin{mainThm}[Theorem \ref{Thm:Main}] \label{Thm:IntroMain}
    If $\gx$ is of Hodge type, then there exists a system $\{\scrs_{K}\gx\}_{K^p}$ satisfying \cite[Conjecture 4.2.2]{PappasRapoportShtukas}. 
\end{mainThm}
In fact, in Theorem~\ref{Thm:Main}, we prove an extension of \cite[Conjecture 4.2.2]{PappasRapoportShtukas} to quasi-parahoric subgroups. This generalization was motivated by considerations arising from our construction of Igusa stacks of Hodge type, \cite{IgusaStacks}. We also note that Bruhat--Tits stabilizer quasi-parahorics naturally arise in moduli-theoretic integral models, while parahorics are more natural from the representation-theoretic point of view.

Theorem \ref{Thm:IntroMain} is used in work of one of us (PD) and Youcis, see \cite{DanielsYoucis}, to prove \cite[Conjecture 4.2.2]{PappasRapoportShtukas} for almost all (and all if $p \ge 3$) Shimura varieties of abelian type. Without Theorem \ref{Thm:IntroMain}, the results of \textit{loc. cit.} would have strong restrictions for Shimura varieties of type $D^{\mathbb{H}}$ in the sense of \cite[Appendix B]{Milne}. 

We remark that recent work of Takaya \cite{Takaya} also proves \cite[Conjecture 4.2.2]{PappasRapoportShtukas} under the more restrictive assumption that $K_p$ is contained in a hyperspecial subgroup $K_p'$ of $G(\qp)$, assuming the conjecture holds for $K_p'$. Such a $K_p$ is necessarily a stabilizer parahoric\footnote{See \cite[Remark 4.2.14.(b)]{KisinPappas}.}, so in the Hodge-type case the result of Takaya follows from the work of Pappas and Rapoport. The results of Takaya therefore do not intersect with ours. We mention also that the methods of Takaya, while similar in spirit to ours, crucially require smoothness and so they do not apply in our situation. 

\subsubsection{} As an application of Theorem \ref{Thm:Main}, we prove a conjecture of Kisin and Pappas on the existence of local model diagrams for Shimura varieties of Hodge type \cite[\S 4.3.10]{KisinPappas}. Let $K_p = \mathcal{G}(\zp)$ for some quasi-parahoric $\zp$-model $\mathcal{G}$ for $G$. Recall \cite[\S 4.9]{PappasRapoportShtukas} that a \textit{local model diagram} for $\scrs_K\gx$ is a diagram of $\mathcal{O}_E$-schemes
\begin{equation}\label{Eq:LocalModelDiagram}
    \begin{tikzcd} 
        & \widetilde{\scrs}_K\gx \arrow[dl, "\pi"'] \arrow[dr, "q"] & \\
        \scrs_K\gx & & \mathbb{M}_{\mathcal{G},\mu},
    \end{tikzcd}
\end{equation}
where $\pi$ is a $\mathcal{G}$-torsor and $q$ is a smooth, $\mathcal{G}$-equivariant morphism. Here $\mathbb{M}_{\mathcal{G},\mu}$ is the local model associated to $\mathcal{G}$ and $\mu$, see e.g., \cite{AGLR}. If a diagram as in \eqref{Eq:LocalModelDiagram} exists, then the singularities of $\scrs_K\gx$ are (\'etale-locally) modeled by those of the (often simpler) scheme $\mathbb{M}_{\mathcal{G},\mu}$. The existence of a diagram \eqref{Eq:LocalModelDiagram} is shown in \cite{KisinPappasZhou} under some assumptions on $p$, $G$, and $\mathcal{G}$, see Section \ref{subsub:Assumptions}. These assumptions are satisfied in many cases of interest, see Remarks \ref{Rem:AssumptionsI} and \ref{Rem:AssumptionsII} below. However, we emphasize that they need to assume that $\mathcal{G}$ is a stabilizer quasi-parahoric. 

In \cite[Section 4.9.1]{PappasRapoportShtukas}, Pappas and Rapoport construct an analogous diagram at the level of v-sheaves for any integral model $\scrs_K\gx$ which admits a $\mathcal{G}$-shtuka. A diagram as in \eqref{Eq:LocalModelDiagram} which recovers the Pappas-Rapoport v-sheaf diagram is called a \textit{scheme-theoretic local model diagram}, \cite[Definition 4.9.1]{PappasRapoportShtukas}. Pappas and Rapoport conjecture the existence of scheme-theoretic local model diagrams in general, see \cite[Conjecture 4.9.2]{PappasRapoportShtukas}. The following theorem proves their conjecture in many cases, and part (2) verifies the conjecture of Kisin--Pappas in many cases.

\begin{mainThm}[Theorem \ref{Thm:WhyDidWeCheckThis?}, Theorem \ref{Thm:SchemeTheoreticLocalModelExists}] \label{Thm:IntroLocalModels}
Let $\gxg$ be as above and assume that it satisfies the assumptions of Kisin--Pappas--Zhou, see Section \ref{subsub:Assumptions}; let $\calgcirc$ be the identity component of $\calg$ and set $K_p^{\circ}=\calgcirc(\zp)$. The following hold:
\begin{enumerate}
    \item The local model diagram of \cite{KisinPappasZhou} for $\scrs_{K}\gx$ is a scheme-theoretic local model diagram.

    \item The integral model $\scrs_{K^{\circ}}\gx$ admits a scheme-theoretic local model diagram. 
\end{enumerate}
\end{mainThm}
\subsubsection{} Other recent advances in the theory of integral models of Shimura varieties of Hodge type that require one to restrict to stabilizer parahorics are Rapoport--Zink uniformization of isogeny classes and the existence of CM lifts, see \cite[Theorem 1.1]{ZhouIsogeny}, \cite[Theorem I,II]{vH}, \cite[Corollary 1.4, Corollary 6.3]{GleasonLimXu}. We show that the proof of \cite[Corollary 6.3]{GleasonLimXu} can be
combined with Theorem \ref{Thm:IntroMain} to prove uniformization in full generality, see Corollary \ref{Cor:Uniformization}. We expect that Corollary \ref{Cor:Uniformization} can be used to prove the existence of CM lifts of isogeny classes when $G$ is quasi-split.

\subsection{A sketch of the proof of Theorem \ref{Thm:IntroMain}} From now on, we change our notation and let $\mathcal{G}^{\circ}$ be a parahoric group scheme which is the relative identity component of a stabilizer Bruhat--Tits group scheme $\mathcal{G}$, see Section \ref{Sec:BruhatTits}. Fixing $K^p \subset \gafp$, we will write $K_p^{\circ}=\calgcirc(\zp)$, $K_p=\mathcal{G}(\zp)$, and $K^{\circ}=K^p K_p^{\circ}$.

As usual, to construct integral models of Shimura varieties of level $K$ and $K^{\circ}$, we choose a Hodge embedding $\gx \to \gvx$, where $\gv=\operatorname{GSp}(V, \psi)$ for a symplectic space $(V,\psi)$ over $\mathbb{Q}$. We then choose a lattice $\Lambda \subset V \otimes \qp$ such that $\mathcal{G}(\zpbr)$ is the stabilizer of $\Lambda \otimes_{\zp} \zpbr$, and define $\scrs_{K}\gx$ as the normalization of the Zariski closure of $\mathbf{Sh}_K\gx$ in an integral model of the Shimura variety for $\gvx$ at level $K_{\Lambda}$. The arguments in \cite{PappasRapoportShtukas} show that $\scrs_{K}\gx$ satisfies (a generalization to quasi-parahoric subgroups of) the axioms of \cite[Conjecture 4.2.2]{PappasRapoportShtukas}. We define $\scrs_{K^{\circ}}\gx$ to be the normalization of $\scrs_{K}\gx$ in $\mathbf{Sh}_{K^{\circ}}\gx$. 

\subsubsection{} \label{SubSub:Sketch} The most important part of the axioms of \cite[Conjecture 4.2.2]{PappasRapoportShtukas} is the existence of a $\mathcal{G}$-shtuka on $\scrs_{K}\gx$, encoded as a morphism of v-stacks (here the notation $(-)^{\diamondsuit/}$ denotes a variant of the v-sheaf associated to a $\zp$-scheme, see Section \ref{Sub:AdicSpaces} below, and $\mu$ is the $G(\qpbar)$-conjugacy class of cocharacters of $G$ coming from the Hodge cocharacter and the place $v$)
\begin{align}
    \scrs_{K}\gx^{\diamondsuit/} \to \shtgmu.
\end{align}
Yet it is not a priori clear that there is a $\calgcirc$-shtuka on $\scrs_{K^{\circ}}\gx$, or in other words, that there is a dotted arrow making the following diagram commutative
\begin{equation} \label{Eq:DottedArrow}
    \begin{tikzcd}
        \scrs_{K^{\circ}}\gx^{\diamondsuit/} \arrow{d} \arrow[r, densely dotted] & \shtgcircmu \arrow{d} \\
         \scrs_{K}\gx^{\diamondsuit/} \arrow{r} & \shtgmu.
    \end{tikzcd}
\end{equation}
In fact, considerations from our companion paper \cite{IgusaStacks} lead us to believe that such a diagram exists and is cartesian. A computation of \cite[Section 4.3]{KisinPappas} suggests that the left vertical map in the diagram should be finite \'etale. 

The rough strategy of the proof now goes as follows: We show that $\scrs_{K}\gx^{\diamondsuit/} \to \shtgmu$ and $\shtgcircmu \to \shtgmu$ factor through an open and closed substack $\shtgmuone \subset \shtgmu$. We then show that $\shtgcircmu \to \shtgmuone$ is an \'etale torsor under a finite abelian group $\Lambda$, see Theorem~\ref{Thm:IntroShtukas} and Corollary~\ref{Cor:IntroCor} below. By pulling back this cover along the map $\scrs_{K}\gx^{\diamondsuit/}\to \shtgmuone$, we get an \'etale $\Lambda$-torsor $\mathscr{Y} \to \scrs_K\gx^{\diamondsuit/}$. We show, using a result of one of us (DK) \cite{KimVectorBundles}, that $\mathscr{Y}$ must be isomorphic to $\scrs_{K^{\circ}}\gx^{\diamondsuit/}$ over $\scrs_K\gx^{\diamondsuit/}$, see Proposition \ref{Prop:FiniteEtaleCover}. This then implies that $\scrs_{K^{\circ}}\gx^{\diamondsuit/} \to \scrs_K\gx^{\diamondsuit/}$ is an \'etale $\Lambda$-torsor and establishes the existence of the dotted arrow in \eqref{Eq:DottedArrow}. The resulting diagram is Cartesian, and the rest of the proof of Theorem \ref{Thm:IntroMain} is now routine.

\subsubsection{Moduli stacks of quasi-parahoric shtukas} The stack of $\mathcal{G}$-shtukas $\shtgmu$ is not as well behaved as the stack of $\calgcirc$-shtukas $\shtgcircmu$. For example the image of
\begin{align}
    \shtgcircmu \to \bun_G
\end{align}
is given by the open substack corresponding to the $\mu^{-1}$-admissible elements $\bgmu \subset B(G) = |\bun_G|$, see Lemma \ref{Lem:NewtonOverHodge}; the analogous statement generally fails for $\shtgmu$.

Our first order of business is to show that this failure can be rectified by restricting to the preimage $\shtgmukappa$ of $-\mu^\natural$ under the Kottwitz map (see \cite[Theorem III.2.7]{FarguesScholze})
\begin{align}
    |\shtgmu| \to |\bun_G| \to \pi_1(G)_{\Gamma_p},
\end{align}
see Proposition \ref{Prop:NewtonOverHodge}. Following ideas of \cite[Section 4]{PappasRapoportRZSpaces}, we show the following (see Section \ref{Sec:Decomposition} for the notation). 
\begin{mainThm}[Theorem \ref{Thm:QuasiParahoricShtukas}]\label{Thm:IntroShtukas}
There is a finite decomposition
\[
    \coprod_{\delta \in \Pi_\mathcal{G}}
    [\shtgdeltacircmu/\pi_0(\mathcal{G}_\delta)^\phi] \simeq
    \shtgmubrkappa.
\]
\end{mainThm}
For $\delta=1 \in \Pi_{\mathcal{G}}$, we have $\calgdeltacirc=\calgcirc$. In particular, this establishes the following corollary, which clarifies the relationship between the stack of $\mathcal{G}^\circ$-shtukas with one leg bounded by $\mu$ with that of $\mathcal{G}$-shtukas. We see that the image of $\shtgcircmu \to \shtgmu$ defines an open and closed substack $\shtgmuone \subset \shtgmu$.
\begin{mainCor}[Corollary \ref{Cor:QuasiParahoricShtukas}] \label{Cor:IntroCor}
    The morphism $\shtgcircmu \to \shtgmuone$ is a torsor for the abelian group $\pi_0(\mathcal{G})^\phi$.
\end{mainCor}

As explained above, Corollary \ref{Cor:IntroCor} allows us to prove that $\scrs_{K^\circ}\gx^{\diamondsuit /}$ is the fiber product of $\scrs_K\gx^{\diamondsuit /}$ with $\shtgcircmu$ over $\shtgmu$.

\subsubsection{} 
Let us close the introduction with some comments on the proof of Theorem \ref{Thm:IntroLocalModels}. Both parts of Theorem \ref{Thm:IntroLocalModels} contain two separate assertions: That there exists a diagram as in \eqref{Eq:LocalModelDiagram} for $\scrs_{K}\gx$ (and $\scrs_{K^{\circ}}\gx$), and that the diagram recovers the one of \cite[Section 4.9.1]{PappasRapoportShtukas} at the level of v-sheaves. Our strategy is to verify both assertions for $\scrs_K\gx$, and then deduce the two simultaneously for $\scrs_{K^\circ}\gx$. 

Under the assumptions in Theorem \ref{Thm:IntroLocalModels}, the existence of a diagram \eqref{Eq:LocalModelDiagram} is proved for $\scrs_K\gx$ in \cite{KisinPappasZhou}. Pappas and Rapoport point out that the construction of \textit{loc. cit.} provides a scheme-theoretic local model diagram in \cite[Section 4.9.2]{PappasRapoportShtukas}. We verify this statement in Appendix \ref{Appendix:A}. 

Given a scheme-theoretic local model diagram for $\scrs_K\gx$, we obtain in particular a $\mathcal{G}$-torsor on $\scrs_{K^\circ}\gx$ by pullback, and we have to show this torsor admits a reduction of structure group to $\mathcal{G}^\circ$. Such a reduction exists at the level of v-sheaves by functoriality of the construction, so the crux of the argument is to show that this arises from a reduction at the level of schemes. This is done in Proposition \ref{Prop:KisinPappasConjecture}.

\subsection{Outline of the paper} In Section \ref{Sec:Preliminaries} we recall preliminaries on perfectoid geometry and Bruhat--Tits theory, and we prove a key technical result, Proposition \ref{Prop:FiniteEtaleCover}. In Section \ref{Sec:Shtukas} we study the moduli stack of $\mathcal{G}$-shtukas for a quasi-parahoric group $\mathcal{G}$ and its relationship to the moduli stack of $\mathcal{G}^\circ$-shtukas for the parahoric group scheme $\mathcal{G}^\circ$ associated with $\mathcal{G}$. This culminates in the proof of Theorem \ref{Thm:IntroShtukas}. Finally, in Section \ref{Sec:ConjecturalModels}, we recall the conjecture of Pappas and Rapoport, and prove our main result, Theorem \ref{Thm:IntroMain}. We close by proving Theorem \ref{Thm:IntroLocalModels}, and proving Rapoport--Zink uniformization, see Theorem \ref{Thm:Uniformization}. In Appendix \ref{Appendix:A}, we verify that the local model diagrams of \cite{KisinPappas, KisinPappasZhou} give scheme-theoretic local model diagrams in the sense of \cite[Conjecture 4.9.2]{PappasRapoportShtukas}, for stabilizer Bruhat--Tits group schemes. 

\subsection{Acknowledgments} We would like to thank Ian Gleason and Alex Youcis for helpful discussions about the proof of Proposition \ref{Prop:FiniteEtaleCover}, Zhiyu Zhang for his comments, and the anonyomous referee for their comments and corrections. 

\subsection{Declarations} The authors declare that they have no conflicts of interest pertaining to this manuscript. Data sharing is not applicable as no datasets were generated or analyzed during the writing of this article. 

\section{Preliminaries} \label{Sec:Preliminaries}

\subsection{Recollections from \texorpdfstring{\cite{FarguesScholze}}{Fargues--Scholze} }
We begin by establishing notation and recalling some definitions from the theory of v-sheaves. For a more comprehensive background, we refer the reader to \cite{ScholzeWeinsteinBerkeley}, \cite{FarguesScholze}, and \cite[Section 2.1]{PappasRapoportShtukas}. Throughout this section, we let $k$ be a perfect field of characteristic $p$, and write $\perf_{k}$ for the category of \textit{affinoid} perfectoid spaces over $k$. If $k=\fp$ we will write $\perf=\perf_{\fp}$.\footnote{In the literature, $\perf_{k}$ usually denotes the category of \emph{all} perfectoid spaces over $k$.}

For any perfectoid space $S$ over $\fp$, we write $S \bdtimes \spa(\zp)$ for the analytic adic space defined in \cite[Proposition 11.2.1]{ScholzeWeinsteinBerkeley}. In particular, when $S = \spa(R,R^+)$ is affinoid perfectoid, $S \bdtimes \spa(\zp)$ is given by 
\begin{align}
    S \bdtimes \spa(\zp) = \spa(W(R^{+})) \setminus \{[\varpi]=0\},
\end{align}
where $[\varpi]$ denotes the Teichm\"uller lift to $W(R^+)$ of a fixed pseudouniformizer $\varpi$ in $R^+$, and where $W(R^+)$ denotes the $p$-typical Witt vectors of $R^+$. The Frobenius for $W(R^+)$ restricts to a Frobenius operator $\frob_S$ on $S \bdtimes \spa(\zp)$. By \cite[Proposition 11.3.1]{ScholzeWeinsteinBerkeley}, any untilt $S^\sharp$ of $S$ determines a closed Cartier divisor $S^\sharp \hookrightarrow S \bdtimes \spa(\zp)$. 

For $S$ in $\perf$, define $Y_S = S \bdtimes \spa(\zp) \setminus \{p = 0\}$. If $S = \spa(R,R^+)$ we write also $Y(R,R^+)$ for $Y_S$. For any $S = \spa(R,R^+)$ in $\perf$, one defines a function (here $\lvert X \rvert$ denotes the underlying topological space of an adic spaces or v-sheaf)
\begin{align}
    \kappa: \lvert S \bdtimes \spa(\zp) \rvert \to [0,\infty)
\end{align}
by $\kappa(x) = (\log \lvert [\varpi](\tilde{x}) \rvert)/(\log(\lvert p(\tilde{x}) \rvert)$, where $\tilde{x}$ denotes the maximal generalization of $x \in \lvert S \bdtimes \spa(\zp) \rvert$, see \cite[Proposition II.1.16]{FarguesScholze} for details. For any interval $I = [a,b] \subset [0,\infty)$ with rational endpoints, denote by $\mathcal{Y}_I(S)$ the open subset of $S \bdtimes \spa(\zp)$ given by 
\begin{align}
    \mathcal{Y}_I(S) = \{\lvert p \rvert^b \le \lvert [\varpi] \rvert \le \lvert p\rvert^a\} \subset \kappa^{-1}(I).
\end{align} 
One extends this definition to open intervals in the obvious way. In particular,
we have $\mathcal{Y}_{[0,\infty)}(S) = S \bdtimes \spa(\zp)$ and $\mathcal{Y}_{(0,\infty)}(S) = \mathcal{Y}_S$. 
\subsubsection{}
By \cite[Proposition II.1.16]{FarguesScholze}, for any $S$ in $\perf$ the action of $\frob_S$ on $Y_S$ is free and totally discontinuous, hence we may take the quotient
\begin{align}
    \mathcal{X}_S = \mathcal{Y}_S / \frob_S^\mathbb{Z},
\end{align}
called the \textit{relative adic Fargues--Fontaine curve over} $S$, which is an analytic adic space.

Let $G$ be a reductive group over $\qp$. Following \cite{FarguesScholze}, we denote by $\bun_G(S)$ the groupoid of $G$-torsors on $\mathcal{X}_S$. By \cite[Theorem~III.0.2]{FarguesScholze}, the presheaf of groupoids $\bun_G$ on $\perf$ sending $S$ to $\bun_G(S)$ is a small v-stack.

\subsubsection{} \label{Sec:BGMU} For a choice of algebraic closure $\ovfp$ of $\fp$ we set $\zpbr=W(\ovfp)$ and $\qpbr=W(\ovfp)[1/p]$. Let $\sigma$ be the automorphism of $\qpbr$ induced by the absolute Frobenius on $\ovfp$. Let $B(G)$ be the set of $\sigma$-conjugacy classes in $G(\qpbr)$, equipped with the topology coming from the \emph{opposite} of the partial order defined in \cite[Section 2.3]{RapoportRichartz}. The formation of $B(G)$ is invariant under extensions of algebraically closed fields $\fpbar\hookrightarrow F$. Indeed, for such an extension, the natural map $G(\qpbr)\hookrightarrow G(W(F)[1/p])$ induces a bijection on $\sigma$-conjugacy classes. 

By \cite[Theorem 1]{Viehmann}, there is a homeomorphism
\begin{align}
    \lvert \bun_G \rvert \xrightarrow{\sim} B(G). 
\end{align}
If $\mu$ is a $G(\qpbar)$-conjugacy class of minuscule cocharacters, we let $\bgmu \subset B(G)$ be the (open) subset of $\mu^{-1}$-admissible elements, as defined in \cite[Section 1.1.5]{KMPS}; this defines an open substack $$\bungmu \subset \bung$$ via \cite[Proposition 12.9]{EtCohDiam}. Explicitly, for $S$ in $\perf$, $\bungmu(S)$ consists of maps $S \to \bung$ for which the induced map on topological spaces factors through $\bgmu \subset B(G)\cong \lvert \bun_G \rvert$. 

\subsubsection{} For $b \in G(\qpbr)$, we let $[b] \in B(G)$ denote the $\sigma$-conjugacy class of $b$. We recall the following result. 

\begin{Thm}{\cite[Theorem~III.0.2]{FarguesScholze}}\label{Thm:GeometryOfBunG}
The subfunctor      
\[
        \bun_{G}^{[b]} = \bun_G \times_{\lvert \bun_G \rvert} \lbrace [b] \rbrace
        \subseteq \bun_G
\]
      is locally closed. Moreover its base change to $\spd(\ovfp)$ is isomorphic to $[\spd(\ovfp) / \tilde{G}_b]$, where $\tilde{G}_b = \Aut(\mathcal{E}_b)$ and $\mathcal{E}_b \in \bun_G(\spd(\ovfp))$ corresponds  to $b$ (see \cite[Theorem 5.3]{AnschutzIsocrystal}).
\end{Thm}

For any v-stack $\mathscr{Y}$ on $\perf$ equipped with a morphism $\mathscr{Y} \to \bun_G$, we write
\begin{equation} \label{Eq:NewtonStratumNotation}
    \mathscr{Y}^{[b]} := \mathscr{Y} \times_{\bun_G} \bun_G^{[b]}.
\end{equation}

\subsubsection{}\label{Sub:AdicSpaces} If $X$ is a pre-adic space over $\spa(\zp)$ in the sense of \cite[Section 3.4]{ScholzeWeinsteinBerkeley}, we let $X^\diamondsuit$ denote the set-valued functor on $\perf$ given by
\begin{align}
    X^\diamondsuit(S) = \{(S^\sharp, f)\} / \text{isom.}
\end{align}
for any $S$ in $\perf$, where $S^\sharp$ is an untilt of $S$ and $f:S^\sharp \to X$ is a morphism of pre-adic spaces. This determines a v-sheaf on $\perf$ by \cite[Lemma 18.1.1]{ScholzeWeinsteinBerkeley}. For a Huber pair $(A,A^+)$ we write $\spd(A,A^+)$ in place of $\spa(A,A^+)^\diamondsuit$, and we abbreviate it as $\spd(A)$ when $A^+$ is equal to the subring $A^\circ$ of power bounded elements. In particular, $\spd(\zp)$ parametrizes isomorphism classes of untilts, see \cite[Definition 10.1.3]{ScholzeWeinsteinBerkeley}.

For a formal scheme $\mathfrak{X}$ over $\spf(\zp)$, we write $\mathfrak{X}^\mathrm{ad}$ for the pre-adic space associated to $\mathfrak{X}$ as in \cite[Proposition 2.2.1]{ScholzeWeinsteinModuli}. We then write $\mathfrak{X}^\diamondsuit$ as shorthand for $(\mathfrak{X}^\mathrm{ad})^\diamondsuit$.

For a $\zp$-scheme $X$, we can attach to it two different v-sheaves, following \cite[Section 2.2]{AGLR}. If $X = \spec(A)$ is affine, we define v-sheaves $X^\diamond$ and $X^\diamondsuit$ whose points on an affinoid perfectoid space $S = \spa(R,R^+)$ are
\begin{equation}
    X^\diamond(S) = \{ (\spa(R^\sharp, R^{\sharp+}), f: A \to R^{\sharp+}) \}/ \text{isom.},
\end{equation}
and respectively
\begin{equation}
    X^\diamondsuit(S) = \{ (\spa(R^\sharp, R^{\sharp+}), f: A \to R^\sharp)\} / \text{isom.},
\end{equation}
where $\spa(R^{\sharp}, R^{\sharp+})$ denotes an untilt of $\spa(R,R^+)$, and in each case $f$ denotes a $\zp$-algebra homomorphism.\footnote{Note that in \cite{PappasRapoportShtukas}, the notation $(-)^\blacklozenge$ is used in place of $(-)^\diamond$.} 

Both $(-)^\diamond$ and $(-)^\diamondsuit$ are compatible with localisations and glue to define functors from the category of schemes over $\spec(\zp)$ to the category of v-sheaves over $\spd(\zp)$. Following \cite{AGLR}, we refer to these as the ``small diamond'' and ``big diamond'' functors, respectively. There is a natural transformation
\begin{equation}
    j_X: X^\diamond \to X^\diamondsuit,
\end{equation}
which is a monomorphism if $X$ is separated over $\zp$, an open immersion if $X$ is separated and of finite type over $\zp$, and is an isomorphism if $X$ is proper over $\zp$. 

The two diamond functors can also be obtained by passing first (suitably) from schemes to their attached adic spaces. Indeed, if $X$ is a $\zp$-scheme, then $X^\diamond \cong (\widehat{X})^\diamondsuit$, where $\widehat{X}$ denotes the formal scheme over $\spf(\zp)$ given by the $p$-adic completion of $X$. If $X$ is additionally locally of finite type over $\spec(\zp)$, then we denote by $X^\mathrm{ad}$ the fiber product
\begin{equation}\label{Eq:AdicSpace}
    X^\mathrm{ad} = X \times_{\spec(\zp)} \spa(\zp)
\end{equation}
in the sense of \cite[Proposition 3.8]{Huber}, and one can check that $X^\diamondsuit \cong (X^\mathrm{ad})^\diamondsuit$. 

Following \cite[Definition 2.1.9]{PappasRapoportShtukas}, for a scheme $X$ which is separated and of finite type over $\zp$, we will also consider the v-sheaf $X^{\diamondsuit/}$, defined by gluing $X^{\diamond}$ to $X_{\qp}^{\diamondsuit}$\footnote{The operations of applying $(-)^\diamondsuit$ and base change to $\qp$ commute, so this notation is unambiguous. We also remark that $(-)^\diamond$ does not commute with base change to $\qp$.} along the open immersion $(X^{\diamond})_{\qp} \to X_{\qp}^{\diamondsuit}$, that is,
\begin{align}
    X^{\diamondsuit /} = X^\diamond \sqcup_{(X^\diamond)_\qp} X_\qp^\diamondsuit.
\end{align}

All constructions above extend to schemes over the ring of integers $\mathcal{O}_E$ in a finite extension $E$ of $\qp$ or $\qpbr$. Below we will use these constructions without comment.

\subsubsection{} We recall the following construction from \cite{ScholzeWeinsteinBerkeley}\footnote{The terminology first appeared in \cite{GleasonSpecialization}.}.

\begin{Def} \label{Def:ProductOfPoints}
  A \textit{product of geometric points} is the adic spectrum of a perfectoid Huber pair of the form 
  \begin{align*}
      \big((\prod_{i\in I}
  C_i^+)[\varpi^{-1}], \prod_{i\in I} C_i^+\big),
  \end{align*}
  where $I$ is a set, and for each $i \in I$, 
  \begin{itemize}
      \item $C_i$ is an algebraically closed perfectoid field of characteristic $p$, and
      \item $C_i^+$ is an open, bounded valuation subring of $C_i$ with pseudouniformizer $\varpi_i$.
  \end{itemize}
  Here we give $\prod_i
  C_i^+$ the $\varpi$-adic topology, where $\varpi = (\varpi_i)$. 
\end{Def}

\def\RkOne{\coprod_i s_i}
\def\ProdPts{S}
\def\ProdRkOne{s}
We introduce the following definition.
\begin{Def}  \label{Def:ProperStar}
Let $f \colon \mathcal{F} \to \mathcal{G}$ be a map of presheaves of groupoids on $\perf$.

\begin{enumerate}[(1)] 
  \item  Given a morphism $T \to S$ in $\perf$ and a
  2-commutative diagram of solid arrows
  \begin{equation} \begin{tikzcd}
    T \arrow[d] \arrow{r} & \mathcal{F}
    \arrow{d} \\ S \arrow[r]
    \arrow[dashed]{ru} & \mathcal{G}, \label{Eq:LiftingDiagram}
  \end{tikzcd} \end{equation}
  we say that $f$ has \textit{uniquely existing lifts along $T \to S$} if the map
  \[
    \lambda_f \colon \mathcal{F}(S) \to \mathcal{F}(T)
    \times_{\mathcal{G}(T)} \mathcal{G}(S).
  \]
  induced by the diagram \eqref{Eq:LiftingDiagram} is an equivalence of groupoids. 
  \item We say $f$ is \textit{proper*} if $f$ has uniquely existing lifts along every morphism of the form
  \begin{align*}
    \RkOne := \coprod_i \spa(C_i, \mathcal{O}_{C_i}) \to \ProdPts := \spa((\prod_i C_i^+)[\varpi^{-1}], \prod_i C_i^+)
    \end{align*}
  where $S$ is a product of geometric points  in characteristic $p$.
\end{enumerate}
\end{Def}

\begin{Lem} \label{Lem:LiftsTwoOutOfThree}
  Let $f \colon \mathcal{F} \to \mathcal{G}$ and $g \colon \mathcal{G} \to
  \mathcal{H}$ be maps between presheaves of groupoids on $\perf$. Assume $g$ is
  proper*. Then $f$ is proper* if and only if $g \circ f$ is proper*.
\end{Lem}

\begin{proof}
  We note that $\lambda_{g \circ f}$ can be identified with the composition
  \[
    \mathcal{F}(\ProdPts) \xrightarrow{\lambda_f}
    \mathcal{F}({\textstyle\RkOne}) \times_{\mathcal{G}(\RkOne)}
    \mathcal{G}(\ProdPts) \xrightarrow{\mathrm{id} \times \lambda_g}
    \mathcal{F}({\textstyle\RkOne}) \times_{\mathcal{H}(\RkOne)}
    \mathcal{H}(\ProdPts),
  \]
  and so if $\lambda_g$ is an equivalence, then $\lambda_f$ is an
  equivalence if and only if $\lambda_{g \circ f}$ is an equivalence.
\end{proof}

\begin{Lem} \label{Lem:ProperRepresentableLifting}
  If a map $f \colon \mathcal{F} \to \mathcal{G}$ of v-stacks is proper and
  representable by diamonds, then $f$ is proper*.
\end{Lem}

\begin{proof}
  This is an immediate consequence of \cite[Proposition~2.18]{ZhangThesis}. Note
  that partial properness is used to produce uniquely existing lifts along $\coprod_i
  \spa(C_i, \mathcal{O}_{C_i}) \to \coprod_i \spa(C_i, C_i^+)$.
\end{proof}

\begin{Lem} \label{Lem:BaseChangeLifting}
  Given a Cartesian square of presheaves of groupoids on $\perf$
  \[ \begin{tikzcd}
    \mathcal{F}^\prime \arrow{d}{f^\prime} \arrow{r}{g^\prime} &
    \mathcal{F} \arrow{d}{f} \\ \mathcal{G}^\prime \arrow{r}{g} & \mathcal{G},
  \end{tikzcd} \]
  \begin{enumerate}[(1)]
   \item if $f$ is proper*, then $f^\prime$ is proper*, and 
    \item if $f'$ is proper*, and for every product of geometric points $S$ in $\perf$ the map $\mathcal{G}'(\ProdPts) \to \mathcal{G}(\ProdPts)$ induced by $g$ is essentially surjective, then $f$ is proper*.
  \end{enumerate}
\end{Lem}

\begin{proof}
  We note that there is a natural commutative diagram
  \[ \begin{tikzcd}
    \mathcal{F}^\prime(\ProdPts) \arrow{r}{\lambda_{f^\prime}}
    \arrow{d}{g^\prime} & \mathcal{F}^\prime(\RkOne)
    \times_{\mathcal{G}^\prime(\RkOne)} \mathcal{G}^\prime(\ProdPts)
    \arrow{d}{g^\prime \times_g g} \arrow{r} & \mathcal{G}^\prime(S)
    \arrow{d}{g} \\ \mathcal{F}(\ProdPts) \arrow{r}{\lambda_f} &
    \mathcal{F}(\RkOne) \times_{\mathcal{G}(\RkOne)} \mathcal{G}(\ProdPts)
    \arrow{r} & \mathcal{G}(\ProdPts)
  \end{tikzcd} \]
  where both squares are Cartesian. It is then clear that $\lambda_f$ being an
  equivalence implies $\lambda_{f^\prime}$ being an equivalence. Conversely, if
  $\mathcal{G}^\prime(\ProdPts) \to \mathcal{G}(\ProdPts)$ is essentially
  surjective, then all vertical maps are, and hence $\lambda_{f^\prime}$ being
  an equivalence implies that $\lambda_f$ is an equivalence as well.
\end{proof}

\subsection{Some Bruhat--Tits theory } \label{Sec:BruhatTits} Let $G$ be a connected reductive group over $\qp$. We write $\Gamma_p$ for the absolute Galois group $\gal(\qpbar/\qp)$, and let $I_{p} \subset \Gamma_p$ be the inertia subgroup. Let $\pi_1(G)$ be the algebraic fundamental group of $G$, see \cite{Borovoi}. Recall from \cite[Section 7]{Kottwitz2}, that there is a functorial and surjective homomorphism 
\[
  \tilde{\kappa}_G \colon G(\qpbr) \to \pi_1(G)_{I_p}.
\]
The map $\tilde{\kappa}_G$ is called the Kottwitz map, an exposition of whose construction is given in \cite[Section 11.5]{KalethaPrasad}. Denote the composition of $\tilde{\kappa}_G$ with $\pi_1(G)_{I_p} \to \pi_1(G)_{\Gamma_p}$ by $\kappa_G$. We define $G(\qpbr)^0$ to be the kernel of $\tilde{\kappa}_G$ and $G(\qpbr)^1$ to be the inverse image under $\tilde{\kappa}_G$ of the torsion subgroup $\pi_1(G)_{I_p, \mathrm{tors}}$ of $\pi_1(G)_{I_p}$. 

\subsubsection{} \label{Sec:Parahorics} Let $\mathcal{B}(G,\qp)$ (resp. $\mathcal{B}(G,\qpbr)$) denote the (reduced) Bruhat--Tits building of $G$ (resp. of $G_{\qpbr})$; it is a contractible metric space with an action of $G(\qp)$ (resp. $G(\qpbr)$) by isometries, see \cite[Axiom 4.1.1, Corollary 4.2.9]{KalethaPrasad}. It also naturally has the structure of a polysimplicial complex (see \cite[Definition 1.5.1]{KalethaPrasad}) with facets denoted by $\mathcal{F} \subset \mathcal{B}(G,\qp)$ (resp. $\mathcal{F} \subset \mathcal{B}(G,\qpbr)$). Note that there is a $G(\qp)$-equivariant inclusion $\mathcal{B}(G,\qp) \subset \mathcal{B}(G,\qpbr)$ identifying $\mathcal{B}(G,\qp)$ with the fixed points of $\mathcal{B}(G,\qpbr)$ under the Frobenius $\sigma$, see \cite[Theorem 9.2.7]{KalethaPrasad}.

Given a subset $\Omega \subseteq \mathcal{B}(G,\qpbr)$ we consider the pointwise
stabilizers $G(\qpbr)^0_\Omega$ and $G(\qpbr)^1_\Omega$ of $\Omega$ inside of
$G(\qpbr)^0$ and $G(\qpbr)^1$, respectively. Subgroups of the form
$G(\qpbr)^0_{\mathcal{F}}$ for a facet $\mathcal{F}$ are called
\textit{parahoric subgroups}. Following
\cite[Section~2.2]{PappasRapoportRZSpaces}, we will define a
\textit{quasi-parahoric subgroup} $\breve{K} \subseteq G(\qpbr)$ to be any
subgroup for which there exists a facet $\mathcal{F}$ such that
\[
  G(\qpbr)^0_\mathcal{F} = G(\qpbr)^0 \cap \operatorname{Stab}_{\mathcal{F}}
  \subset \breve{K} \subset G(\qpbr)^1 \cap \operatorname{Stab}_{\mathcal{F}},
\]
where now
\[
  \operatorname{Stab}_{\mathcal{F}} =\{ g \in G(\qpbr) : g \mathcal{F} =
  \mathcal{F}\}
\]
is the stabilizer of $\mathcal{F}$ in $G(\qpbr)$ (rather than the pointwise
stabilizer). Note that the inclusion $G(\qpbr)^1_{\mathcal{F}} \subset
G(\qpbr)^1 \cap \operatorname{Stab}_{\mathcal{F}}$ is generally not an
equality, e.g., when $G = \mathrm{PGL}_2$ and $\mathscr{F}$ is an alcove.

\subsubsection{} For a quasi-parahoric subgroup $\breve{K}$ there is a unique
smooth affine group scheme $\mathcal{G}$ over $\zpbr$ together with an
isomorphism $\mathcal{G}_{\qpbr} \xrightarrow{\sim} G_{\qpbr}$ which identifies
$\mathcal{G}(\zpbr)$ with $\breve{K}$, called the \emph{quasi-parahoric group
scheme} associated to $\breve{K}$. When $\breve{K}$ is moreover stable under
$\sigma$, the group $\mathcal{G}$ descends canonically to a smooth affine group
scheme over $\zp$. For example, let $\Omega \subseteq \mathcal{B}(G,\qp)$ be a
nonempty subset contained in a facet, and view it as a subset of
$\mathcal{B}(G,\qpbr)$. Then the pointwise stabilizers $G(\qpbr)^0_\Omega$ and
$G(\qpbr)^1_\Omega$ are stable under $\sigma$, and define quasi-parahoric group
schemes over $\zp$.

\subsubsection{} \label{Sec:quasi-parahoric}
For any quasi-parahoric subgroup $\breve{K} \subseteq
G(\qpbr)$, there exists by definition a facet $\mathcal{F}$ in
$\mathcal{B}(G,\qpbr)$ with $G(\qpbr)^0 \cap \mathrm{Stab}_\mathcal{F} \subseteq
\breve{K} \subseteq G(\qpbr)^1 \cap \mathrm{Stab}_\mathcal{F}$. Intersecting
with $G(\qpbr)^0$, we observe that $G(\qpbr)_\mathcal{F}^0 = \breve{K} \cap
G(\qpbr)^0$, and hence the facet $\mathcal{F}$ is in fact uniquely determined by
$\breve{K}$, since the facet is determined by $G(\qpbr)_\mathcal{F}^0$, see \cite[Proposition~9.3.25]{KalethaPrasad}. The inclusion $G(\qpbr)^0_\mathcal{F} \hookrightarrow \breve{K}$
induces an open immersion $\mathcal{G}^\circ \to \mathcal{G}$ of smooth affine
group schemes over $\zpbr$, where $\mathcal{G}^\circ$ is the parahoric group
scheme corresponding to $G(\qpbr)^0_\mathcal{F}$. Moreover, the induced map on
the special fiber $\mathcal{G}^\circ_{\fpbar} \to \mathcal{G}_{\fpbar}$ is the
inclusion of the identity component, see \cite[Theorem~8.3.13]{KalethaPrasad}.
By \cite[Corollary~11.6.3]{KalethaPrasad}, the finite group
\[
  \pi_0(\mathcal{G}):=\pi_0(\mathcal{G}_{\fpbar})
\]
can be identified with the image of $\breve{K}$ in $\pi_1(G)_{I_p,\mathrm{tors}}$
under the Kottwitz map $\tilde{\kappa}_G$. When $\breve{K}$ is also stable under
$\sigma$, we obtain an inclusion $\mathcal{G}^\circ \to \mathcal{G}$ of smooth
affine group schemes over $\zp$. In this case, $\pi_0(\mathcal{G})$ is a finite
\'{e}tale group scheme over $\fp$.

\begin{Def} \label{Def:StabilizerParahoric}
  We say that a quasi-parahoric group scheme $\mathcal{G} / \zp$ is a
  \textit{stabilizer Bruhat--Tits group scheme} when $\mathcal{G}(\zpbr)
  \subseteq G(\qpbr)$ is of the form $G(\qpbr)^1_x$ for a point $x \in
  \mathcal{B}(G,\qp)$ (as opposed to $x \in \mathcal{B}(G,\qpbr)$). A
  \textit{stabilizer parahoric group scheme} is stabilizer Bruhat--Tits group
  scheme with connected special fiber.

  We also say that an open subgroup $K \subseteq G(\qp)$ is a \textit{stabilizer
  parahoric subgroup} when it is a parahoric subgroup and the corresponding
  smooth affine group scheme $\mathcal{G}/\zp$ is a stabilizer parahoric group
  scheme.\footnote{Stabilizer parahoric subgroups are also called
  \textit{connected parahorics} in the literature, e.g., in \cite{ZhouIsogeny}.
  We find the terminology ``stabilizer parahoric'' more descriptive, as
  parahoric group schemes are connected by construction.
  }
\end{Def}

\begin{Rem}
  A subgroup of $G(\qpbr)$ being a stabilizer of a point in
  $\mathcal{B}(G,\qp)$ is strictly stronger than being both $\sigma$-stable and
  a stabilizer of a point in $\mathcal{B}(G,\qpbr)$.

  Following \cite[Remark~2.3]{PappasRapoportCurves}, we consider the group $G =
  D^\times / \mathbb{G}_m$, where $D/\qp$ is
  the unique quaternion algebra. The building $\mathcal{B}(G,\qpbr) \cong
  \mathcal{B}(\mathrm{PGL}_2,\qpbr)$ is a tree, inside which
  $\mathcal{B}(G,\qp)$ is a midpoint of an edge $\mathcal{F}$, see
  \cite[Example~9.2.9]{KalethaPrasad}. Taking any point $x \in \mathcal{F}
  \setminus \mathcal{B}(G,\qp)$, we see that $G(\qpbr)_x^1 =
  G(\qpbr)_\mathcal{F}^1$ is $\sigma$-stable but not a stabilizer of a point in
  $\mathcal{B}(G,\qp)$.
\end{Rem}

The following result is well known, but we include the proof for the sake of completeness.

\begin{Lem} \label{Lem:QuasiToStabilizer}
  Let $\breve{K} \subseteq G(\qpbr)$ be a $\sigma$-stable quasi-parahoric
  subgroup. Then there exists a point $x \in \mathcal{B}(G,\qp)$ for which
  $\breve{K} \subseteq G(\qpbr)_x^1$ and $\breve{K} \cap G(\qpbr)^0 =
  G(\qpbr)_x^0$.
\end{Lem}

\begin{proof}
  Consider the unique facet $\mathcal{F}$ of $\mathcal{B}(G,\qpbr)$ for which
  $G(\qpbr)_\mathcal{F}^0 \subseteq \breve{K} \subseteq G(\qpbr)^1 \cap
  \operatorname{Stab}_\mathcal{F}$. Since $\breve{K}$ is $\sigma$-stable, so is
  $\mathcal{F}$. We now note that $\gal(\qpbr/\qp) \rtimes \breve{K}$ acts on
  $\mathcal{F}$ through affine-linear automorphisms. Thus if we take $x$ to be the
  center-of-mass of the vertices of $\mathcal{F}$, the point $x$ is fixed under
  the action of $\gal(\qpbr/\qp) \rtimes \breve{K}$. It is also contained in
  $\mathcal{F}$ because $\mathcal{F}$ is the interior of a convex polytope.

  Because $x$ is fixed under $\gal(\qpbr/\qp)$, we have $x \in
  \mathcal{B}(G,\qp)$. Because $x$ is fixed under the action of $\breve{K}$, we
  have $\breve{K} \subseteq G(\qpbr)_x^1$. Finally, we have
  \[
    \breve{K} \cap G(\qpbr)^0 = G(\qpbr)_\mathcal{F}^0 = G(\qpbr)_x^0
  \]
  because $x \in \mathcal{F}$, see \cite[Axiom~4.1.20(1)]{KalethaPrasad}.
\end{proof}

\begin{Cor} \label{Cor:QuasiToStabilizer}
  Let $\mathcal{G}/\zp$ be a quasi-parahoric group scheme for $G/\qp$. Then
  there exists a stabilizer Bruhat--Tits model $\mathcal{H}/\zp$ of $G$
  such that the identity map on $G$ extends to an open embedding $\mathcal{G}
  \hookrightarrow \mathcal{H}$.
\end{Cor}

\begin{proof}
  We take $\breve{K} = \mathcal{G}(\zpbr)$ in Lemma~\ref{Lem:QuasiToStabilizer}
  and let $\mathcal{H}$ be the quasi-parahoric group corresponding to
  $G(\qpbr)_x^1$. The first condition implies that there exists a map
  $\mathcal{G} \to \mathcal{H}$, see \cite[Corollary 2.10.10]{KalethaPrasad},
  which is an isomorphism over $\qp$. By the second condition, both connected
  components $\mathcal{G}^0$ and $\mathcal{H}^0$ are identified with the
  parahoric group scheme corresponding to $G(\qpbr)_x^0$ as in
  Section~\ref{Sec:quasi-parahoric}. By \cite[Corollary~11.6.3]{KalethaPrasad}, the map $\pi_0(\mathcal{G}_{\fpbar}) \to \pi_0(\mathcal{H}_{\fpbar})$ can be identified with the inclusion
  \begin{align}
      \operatorname{Im}\left(\breve{K} \to \pi_1(G)_{I_p,\mathrm{tors}}\right) \to \operatorname{Im}\left(G(\qpbr)_x^1 \to  \pi_1(G)_{I_p,\mathrm{tors}} \right).
  \end{align}
  It follows that $\mathcal{G}_{\fpbar} \to \mathcal{H}_{\fpbar}$ is an open
  embedding and in particular smooth of relative dimension zero. By the
  fiberwise criterion for smoothness, we deduce that $\mathcal{G} \to
  \mathcal{H}$ is smooth of relative dimension zero, hence \'etale. Since it is a monomorphism on both the special fiber and the generic fiber, it is universally injective by \cite[Tag 01S4]{stacks-project} and thus an open immersion by \cite[Tag 025G]{stacks-project}.
\end{proof}

\subsubsection{} \label{Sec:Alcove}
Fix a maximal split torus $S \subset G_{\qpbr}$ with centralizer $T$ and normalizer $N$. By definition, the Iwahori--Weyl group $\widetilde{W}$ associated with $S$ sits in a short exact sequence 
\begin{align}
    1 \to T(\qpbr)^0 \to N(\qpbr) \to\widetilde{W} \to 1,
\end{align}
see \cite[Definition 7]{HainesRapoportParahoric}. 
Let $G_\mathrm{sc}$ denote the simply connected cover of the derived group $G_\mathrm{der}$ of $G$, and let $\widetilde{W}_\mathrm{sc}$ be the Iwahori--Weyl group of $G_\mathrm{sc}$. There is a short exact sequence 
\begin{align}
    1 \to \widetilde{W}_\mathrm{sc} \to \widetilde{W} \to \pi_1(G)_{I_p}\to 0.
\end{align}
Any choice of an alcove\footnote{An alcove is a facet which is maximal for the inclusion relation between facets.} $\mathfrak{a}$ of $\mathcal{B}(G,\qpbr)$ in the apartment associated to $S$ determines a splitting of this short exact sequence.

For a $G(\qpbar)$-conjugacy class $\mu$ of cocharacters, we will use the notation $\admu \subset \widetilde{W}$ for the $\mu^{-1}$-admissible subset, defined as in \cite[Equation (3.4)]{RapoportReduction}. 

\subsubsection{} \label{Subsub:DeltaNotation} Assume $\calg$ is a quasi-parahoric group scheme over $\zp$ determined by a $\sigma$-stable quasi-parahoric subgroup $\Breve{K}\subset G(\qpbr)^1_\mathcal{F}$, such that $\mathcal{F}$ is $\sigma$-stable. As in \cite[Section 3]{PappasRapoportRZSpaces}, we let
$\Pi_{\mathcal{G}}$ be the kernel of $H^1(\zp, \mathcal{G}) \to H^1(\qp, G)$. By Lemma~3.1.1 of \textit{loc. cit.}, we may identify
\[
  \Pi_\mathcal{G} \cong \ker(\pi_0(\mathcal{G})_{\phi} \to \pi_1(G)_{\Gamma_p}),
\]
where $\phi \in \Gamma_p/I_p$ is the Frobenius.\footnote{From now on we will sometimes use $\phi$ instead of $\sigma$ for the Frobenius, to better match the conventions of \cite{PappasRapoportRZSpaces}.} Using this and the exact sequence
\[0\to \pi_1(G)^\phi_{I_p} \to \pi_1(G)_{I_p}\xrightarrow{1-\phi} \pi_1(G)_{I_p}\to \pi_1(G)_{\Gamma_p}\to 0,\] 
we may lift any $\delta\in \Pi_\mathcal{G}$ to an element $\dot{\delta}\in \pi_0(\calg)$, such that $\dot{\delta} =
(1 - \phi) \gamma$ for some $\gamma \in \pi_1(G)_{I_p}$. Choose a splitting of
$\widetilde{W} \to \pi_1(G)_{I_p}$ corresponding to a $\sigma$-stable alcove
$\mathfrak{a}$ of $\mathcal{B}(G,\qpbr)$ with $\mathcal{F} \subset \mathfrak{a}$
as in Section \ref{Sec:Alcove} above. By \cite[Lemma~4.3.1]{PappasRapoportRZSpaces}
there is a lift $\dot{\gamma}$ of $\gamma$ to $N(\qpbr)$ such that
$\dot{\delta} = \phi(\dot{\gamma})^{-1} \dot{\gamma} \in \mathcal{G}(\zpbr)$. We
then obtain a quasi-parahoric integral model $\calgdelta$ of $G$ such that
\[
  \calgdelta(\zpbr) = \dot{\gamma} \mathcal{G}(\zpbr) \dot{\gamma}^{-1} \quad \text{ and } \quad
  \calgdeltacirc(\zpbr) = \dot{\gamma} \mathcal{G}^\circ(\zpbr)
  \dot{\gamma}^{-1}.
\]
The $G(\qp)$-conjugacy class of $\calgdelta(\zpbr)$ does not depend on the
choice of $\dot{\gamma}$ or $\gamma$, see
\cite[Proposition~4.3.2]{PappasRapoportRZSpaces}, and hence the integral model
$\calgdelta$ only depends on $\delta \in \Pi_\mathcal{G}$ up to isomorphism.
However, we shall fix a choice of $\dot{\gamma}$ for each $\delta \in
\Pi_\mathcal{G}$, for later use in Section~\ref{Sec:Decomposition}.

\begin{Rem}
  The group $\calgdelta$ can be identified with the inner twist of $\calg$ by
  $\delta \in H^1(\zp, \calg)$, where the isomorphism $\calg \vert_{\qp} \cong
  G$ comes from the fact that the image of $\delta$ in $H^1(\qp, G)$ is
  trivial. Indeed, upon choosing $\dot{\gamma}$ and $\dot{\delta} =
  \phi(\dot{\gamma})^{-1} \dot{\gamma} \in \calgdelta(\zpbr)$ as above, the
  inner twist ${}^\delta \calg$ has the property that there is a
  $\phi$-equivariant isomorphism ${}^\delta \calg(\zpbr) \cong \calg(\zpbr)$
  where the $\phi$-action on $\calg(\zpbr)$ is $g \mapsto \dot{\delta}^{-1}
  \phi(g) \dot{\delta}$. Composing this with conjugation by $\dot{\gamma}$, we
  obtain a $\phi$-equivariant isomorphism ${}^\delta \calg(\zpbr) \cong
  \dot{\gamma} \calg(\zpbr) \dot{\gamma}^{-1}$ where the $\phi$-actions on both
  sides are natural.
 
\end{Rem}

\subsection{Finite \'etale covers of v-sheaves} \label{Sec:FiniteEtale} Let $\Lambda$ be a finite abelian group. In this section we will compare $\Lambda$-torsors over the v-sheaves associated to schemes with $\Lambda$-torsors over the corresponding schemes.

\begin{Prop}\label{Prop:FiniteEtaleCover}
Let $X$ be a flat normal scheme which is separated and of finite type over $\zp$, and let $f:Z_{\mathrm{rat}} \to X_{\qp}$ be a $\Lambda$-torsor. Suppose $\mathscr{Z}\to X^{\diamondsuit/}$ is an \'etale $\Lambda$-torsor whose generic fiber is $Z_\mathrm{rat}^\diamondsuit\to X_\qp^\diamondsuit$. Then the relative normalization $Z$ of $X$ in $Z_\mathrm{rat}\to X_\qp \to X$ is an \'etale $\Lambda$-torsor, and $Z^{\diamondsuit/}$ is isomorphic to $\mathscr{Z}$ over $X^{\diamondsuit/}$.
\end{Prop}

\begin{Rem}
For $X$ as in the statement of Proposition \ref{Prop:FiniteEtaleCover}, we expect that any finite \'etale cover of $X_{\qp}$ which extends to a finite \'etale cover of $X^{\diamondsuit/}$ comes from a finite \'etale cover of $X$. To prove this, it would suffice to prove an analogue of Lemma \ref{Lem:H1Henselianity} below for arbitrary finite \'etale covers. One would like to apply \cite[Theorem 4.27]{GleasonSpecialization} here, but we were unable to verify that if $\mathcal{F} \to X^{\diamondsuit/}$ is a finite \'etale cover, that then $\mathcal{F}$ must be a prekimberlite (in the sense of \cite[Definition 4.15]{GleasonSpecialization}). To be precise, we were unable to prove that $\mathcal{F}$ is v-specializing in the sense of \cite[Definition 4.6]{GleasonSpecialization}.
\end{Rem}

We first introduce some notation: Consider the closed and open subschemes of $X$ given by its special and generic fiber \[X_0\xrightarrow{i}X\xleftarrow{j}X_{\qp}.\]
Let $\widehat{X}$ be the completion of $X$ along $X_0$, and $\widehat{X}_\eta$ be its adic generic fiber. Its attached diamond $\widehat{X}_\eta^\diamondsuit$ is a quasicompact open sub-diamond of $X_\qp^\diamondsuit$. Similarly, we denote by $\widehat{Z}$ the $p$-adic completion of $Z$. 

\begin{Lem}\label{Lem:H1Henselianity}
    The natural map 
    \[H^1_\textup{\'et}(\widehat{X},\Lambda)\rightarrow H^1_\textup{\'et}(\widehat{X}^\diamondsuit,\Lambda),\]
    is an isomorphism.
\end{Lem}

\begin{proof}
We have a commutative diagram
    \[\begin{tikzcd}
H^1_\text{\'et}(X_0, \Lambda)\arrow[d, "\sim"] & H^1_\text{\'et}(\widehat{X}, \Lambda)\arrow[d]\arrow{l}{\sim}[swap]{i^\ast} \\ H^1_\text{\'et}(X_0^\diamond, \Lambda) & H^1_\text{\'et}(\widehat{X}^\diamondsuit, \Lambda) \arrow{l}[swap]{i^{\diamond,\ast}},
\end{tikzcd}\]
where we note that $\widehat{X}^\diamondsuit \cong X^\diamond$.
The left vertical arrow of the diagram is an isomorphism by \cite[Theorem 1.3]{KimVectorBundles}, and the top arrow is an isomorphism by \stacks{0DEG} and \stacks{0DEA}. This implies surjectivity of the bottom map $i^{\diamond,\ast}$. To prove the lemma, it suffices to prove that the bottom arrow is also injective. 

Without loss of generality, we may assume that $\widehat{X}$ is connected. This implies that $\widehat{X}^{\diamondsuit}$ is connected; indeed, this follows because the map
\begin{align}
    \Hom(\widehat{X}, \spf \zp \amalg \spf \zp) \to \Hom(\widehat{X}^{\diamondsuit}, \spd \zp \amalg \spd \zp)
\end{align}
is a bijection by full-faithfulness of $\widehat{X} \mapsto \widehat{X}^{\diamondsuit}$, see \cite[Theorem 2.16]{AGLR}. Suppose an \'etale $\Lambda$-torsor $f: \mathscr{Y}\to\widehat{X}^\diamondsuit$ splits over $X_0^\diamond$. The natural map from a connected component $\mathscr{G}$ of $\mathscr{Y}$ to $\widehat{X}^{\diamondsuit}$ is still finite \'etale. Thus its image on topological spaces is open and closed, and since $|\widehat{X}^{\diamondsuit}|$ is connected, we find that $|\mathscr{G}| \to |X^\diamond|$ is surjective.  Since $\mathscr{G} \to X^\diamond$ is quasicompact, it follows from \cite[Lemma 12.11]{EtCohDiam} that it is surjective as a map of v-sheaves. Thus the natural map $\mathscr{G} \to \widehat{X}^{\diamondsuit}$ is a finite \'etale cover of $\widehat{X}^\diamondsuit$.

For $\ell\neq p$, we have
\begin{align}
    H^0_\text{\'et}(\mathscr{Y},\mathbb{F}_\ell)&\cong H^0_\text{\'et}(\widehat{X}^\diamondsuit,f_\ast\mathbb{F}_\ell)
    \cong H^0_\text{\'et}(X_0^\diamond,i^{\diamond,\ast} f_\ast\mathbb{F}_\ell)\\
    &\cong H^0_\text{\'et}(\mathscr{Y}^\mathrm{red,\diamond},\mathbb{F}_\ell)
    \cong \mathbb{F}_\ell^{\oplus |\Lambda|}.
\end{align}
Here the second isomorphism follows from \cite[Remark V.4.3(ii)]{FarguesScholze} or \cite[Lemma~4.3]{GleasonLourencoLocalModel}, and the third isomorphism follows from proper base change \cite[Theorem 19.2]{EtCohDiam}. Hence $\mathscr{Y}$ has $n:=|\Lambda|$ connected components and the map $|\mathscr{Y}| \to |\widehat{X}^{\diamondsuit}|$ has fibers of size $n$. It follows that $|\mathscr{G}| \to |\widehat{X}^{\diamondsuit}|$ has fibers of size $1$ (since for each connected component the map on topological spaces is surjective), and thus $\mathscr{G} \to \widehat{X}^{\diamondsuit}$ is an isomorphism by \cite[Lemma 12.5]{EtCohDiam}. It is now clear that the action map $\widehat{X}^\diamondsuit \times \Lambda\to \mathscr{Y}$ over $\widehat{X}^{\diamondsuit}$ is an isomorphism; the injectivity of $i^{\diamond,\ast}$ follows. 
\end{proof}

\begin{proof}[Proof of Proposition \ref{Prop:FiniteEtaleCover}]
It follows from \cite[Theorem 1.3]{KimVectorBundles}, as explained in Lemma \ref{Lem:H1Henselianity}, that there exists an \'etale $\Lambda$-torsor $\mathfrak{Z}\to \widehat{X}$ whose special fiber identifies with $Z_0\to X_0$. It thus suffices to show that $\mathfrak{Z} \to \widehat{X}$ is isomorphic to $\widehat{Z}$ over $\widehat{X}$. \smallskip 

\noindent By \stacks{035L}, the relative normalization $Z$ of $X$ in $Z_\mathrm{rat} \to X$ is normal, since $Z_{\mathrm{rat}}$ is normal. We first prove that $\mathfrak{Z}$ and $\widehat{Z}$ are both $\eta$-normal in the sense of \cite[Definition A.1]{AchingerLaraYoucis}. To show this for $\widehat{Z}$, we use \cite[Lemma A.2]{AchingerLaraYoucis}, which implies that it is enough to check that the local rings of $\widehat{Z}$ at closed points are normal. Note that the closed points of $\widehat{Z}$ are the same as those for $Z$, and at such a point $z$ we have an isomorphism $\widehat{\mathcal{O}}_{\widehat{Z},z} \xrightarrow{\sim} \widehat{\mathcal{O}}_{Z,z}.$ Normality of $\widehat{\mathcal{O}}_{\widehat{Z},z}$ then follows from normality of
$\widehat{\mathcal{O}}_{Z,z}$ which in turn follows from the normality of $\mathcal{O}_{Z,z}$. Indeed, the normality of (quasi-excellent) Noetherian local rings is preserved under completion, see \stacks{0C23}. It then follows from faithful flatness of $\mathcal{O}_{\widehat{Z},z} \to \widehat{\mathcal{O}}_{\widehat{Z},z}$ along with \stacks{033G} that $\mathcal{O}_{\widehat{Z},z}$ is normal. The same proof shows that $\widehat{X}$ is $\eta$-normal, and then it follows from \cite[Corollary A.16]{AchingerLaraYoucis} that the same holds for $\mathfrak{Z}$. \smallskip

\noindent By \cite[Lemma 4.1]{AchingerLaraYoucis}, since both $\mathfrak Z$ and $\widehat{Z}$ are $\eta$-normal, to prove $\mathfrak{Z}$ is isomorphic to $\widehat{Z}$ over $\widehat{X}$, it suffices to show their rigid generic fibers are isomorphic as \'etale $\Lambda$-torsors over $\widehat{X}_\eta$. In turn, it suffices to show the two $\Lambda$-torsors are represented by the same class in $H^1_\text{\'et}(\widehat{X}_\eta,\Lambda)\cong H^1_\text{\'et}(\widehat{X}_\eta^\diamondsuit,\Lambda)$ (see \cite[Lemma 15.6]{EtCohDiam} for this isomorphism). But this follows from the commutative diagram below
\[\begin{tikzcd}
H^1_\text{\'et}(\widehat{X}, \Lambda)\arrow[r]\arrow[d,"\sim"] & H^1_\text{\'et}(\widehat{X}_\eta,\Lambda)\arrow[d,"\sim"] \\ H^1_\text{\'et}(\widehat{X}^\diamondsuit, \Lambda)\arrow[r] &H^1_\text{\'et}(\widehat{X}^\diamondsuit_\eta, \Lambda).
\end{tikzcd}\]
Indeed, going clockwise from $H^1_\text{\'et}(\widehat{X}, \Lambda)$ to $H^1_\text{\'et}(\widehat{X}^\diamondsuit_\eta, \Lambda)$ the class of $\mathfrak{Z}\to \widehat{X}$ is sent to that of $\mathfrak{Z}_\eta^\diamondsuit\to \widehat{X}^\diamondsuit_\eta$. On the other hand, by the proof of Lemma \ref{Lem:H1Henselianity}, going counterclockwise we get the class of $\mathscr{Z}\times_{X^{\diamondsuit /}}\widehat{X}^\diamondsuit_\eta\to \widehat{X}^\diamondsuit_\eta$. Hence we are done if we can show $\mathscr{Z}\times_{X^{\diamondsuit /}} \widehat{X}_\eta^\diamondsuit$ is isomorphic to $\widehat{Z}_\eta^\diamondsuit$. But since $Z \to X$ is integral, this follows from \cite[Proposition 1.9.6]{Huber96} and our assumption that the generic fiber of $\mathscr{Z} \to X^{\diamondsuit /}$ is given by $Z_\mathrm{rat}^\diamondsuit \to X_\qp^\diamondsuit$.
\end{proof}

\section{The moduli stack of quasi-parahoric shtukas} \label{Sec:Shtukas} 

The goal of this section is to study moduli stacks of quasi-parahoric shtukas, and to prove Corollary~\ref{Cor:QuasiParahoricShtukas}.
\subsection{Newton strata in the moduli stack of quasi-parahoric shtukas}
In what follows we let $\mathcal{G}$ be a quasi-parahoric group scheme over $\zp$ with generic fiber $G$, and we let $\mathcal{G}^{\circ} \subset \mathcal{G}$ be the corresponding parahoric group scheme. Let $\operatorname{Gr}_{\mathcal{G}}$ and $\operatorname{Gr}_{\calgcirc}$ over $\spd(\zp)$ be the Beilinson--Drinfeld affine Grassmannians of \cite[Definition 20.3.1]{ScholzeWeinsteinBerkeley}. The natural map $\operatorname{Gr}_{\calgcirc} \to \operatorname{Gr}_{\mathcal{G}}$ becomes an isomorphism after base changing to $\spd(\qp)$. We will call this common base change the $B^+_{\mathrm{dR}}$-affine Grassmannian, and we will denote it by $\mathrm{Gr}_{G} \to \spd(\qp)$, see \cite[Lecture XIX]{ScholzeWeinsteinBerkeley}. 

\subsubsection{} For a $G(\qpbar)$-conjugacy class of minuscule cocharacters $\mu$ of $G$ with reflex field $E$, we denote by $\mathrm{Gr}_{G,\mu} \subset \mathrm{Gr}_{G,E}$ the closed Schubert-cell determined by $\mu$, see \cite[Section 19.2]{ScholzeWeinsteinBerkeley}. We define the \emph{v-sheaf local model} $\locmodgmuv \subset \mathrm{Gr}_{\mathcal{G},\spd(\mathcal{O}_E)}$ to be the v-sheaf theoretic closure of $\mathrm{Gr}_{G,\mu}$ inside $\mathrm{Gr}_{\mathcal{G},\spd(\mathcal{O}_E)}$, and similarly we have $\locmodgcircmuv$. As shown in \cite[Proposition~21.4.3]{ScholzeWeinsteinBerkeley}, functoriality of local models applied to the map $\mathcal{G}^\circ \to \mathcal{G}$ induces an isomorphism
\begin{equation}\label{Eq:local-model-comparison}
    \locmodgcircmuv \xrightarrow{\sim} \locmodgmuv.
\end{equation}

\subsubsection{} A $\mathcal{G}$-shtuka over a perfectoid space $S$ with leg at an untilt $S^\sharp$ is defined to be a $\mathcal{G}$-torsor $\mathscr{P}$ over $S \bdtimes \spa(\zp)$, together with an isomorphism of $\mathcal{G}$-torsors\footnote{Here we consider $\mathrm{Frob}_S^{\ast} (\mathscr{P})$ as a $\mathcal{G}$-torsor via the isomorphism $\mathrm{Frob}_S^{\ast} (\mathcal{G}) \to \mathcal{G}$ coming from the fact that $\mathcal{G}$ is defined over $\zp$.}
\begin{align}
    \phi_{\mathscr{P}}:\restr{\mathrm{Frob}_S^{\ast} \mathscr{P}}{S \bdtimes
    \spa(\zp) \setminus S^{\sharp}} \to \restr{\mathscr{P}}{S \bdtimes \spa(\zp)
    \setminus S^{\sharp}},
\end{align}
that is meromorphic along $S^\sharp$ in the sense of \cite[Definition 5.3.5]{ScholzeWeinsteinBerkeley}. We will occasionally denote such a meromorphic map by
\begin{align*}
    \phi_\mathscr{P}: \frob_S^\ast \mathscr{P} \dashrightarrow \mathscr{P},
\end{align*}
when the choice of untilt $S^\sharp$ is clear. For $\mu$ as above, we say that a $\mathcal{G}$-shtuka $(\mathscr{P}, \phi_{\mathscr{P}})$ is bounded by $\mu$ if the relative position of $\frob_S^\ast \mathscr{P}$ and $\mathscr{P}$ at $S^\sharp$ is bounded by the v-sheaf local model $\locmodgmuv \subset \mathrm{Gr}_{\mathcal{G},\spd(\mathcal{O}_E)}$, see \cite[Section 2.3.4]{PappasRapoportShtukas}. We note that this is well-defined as the local model $\locmodgmuv$ is stable under the action of $L^+\mathcal{G}$ by \cite[Proposition~4.13]{AGLR}, where the argument works verbatim when $\mathcal{G}$ is quasi-parahoric.

For $S$ in $\perf$, denote by $\shtg(S)$ the groupoid of triples $(S^{\sharp}, \mathscr{P}, \phi_{\mathscr{P}})$, where
$S^{\sharp}$ is an untilt of $S$ and where $(\mathscr{P}, \phi_{\mathscr{P}})$ is a $\mathcal{G}$-shtuka over $S$ with leg at $S^\sharp$. By \cite[Proposition~2.1.2]{ScholzeWeinsteinBerkeley}, the assignment $S \mapsto \shtg(S)$ defines a v-stack $\shtg$ on $\perf$ (for this, use the fact that $S \bdtimes \spa(\zp)$ is sousperfectoid by the proof of \cite[Proposition~11.2.1]{ScholzeWeinsteinBerkeley}). For $\mu$ as above, we let $\shtgmu \subset \shtg \times_{\spd(\zp)} \spd(\mathcal{O}_E)$ be the closed substack whose $S$-points consists of $\mathcal{G}$-shtukas over $S$ with one leg at $S^\sharp$, which are bounded by $\mu$.\footnote{Our moduli stacks $\shtgmu$ should not be confused with the moduli spaces of shtukas $\operatorname{Sht}_{(\mathcal{G},b,\mu)}$ of \cite[Definition 23.1.1]{ScholzeWeinsteinBerkeley}, also denoted by $\operatorname{Sht}_{(G,b,\mu,\mathcal{G}(\zp))}$ in \cite[Section 3.4]{GleasonLimXu}. They should also not be confused with the moduli spaces of $p$-adic shtukas $\operatorname{Sht}^{\mathcal{G}_b}_{\zp}$ of \cite[Definition 2.21]{GleasonComponents}. These objects are moduli spaces of shtukas with a framing (towards $b$); the similarity in notation is unfortunate.}\footnote{The stack denoted by $\shtg$ in \cite[Definition 7.4 of version 1]{GleasonIvanov} corresponds in our notation to the stack $\shtg \times_{\spd(\zp)} \spd(\fp)$.}

\subsubsection{} Let $S = \spa(R,R^+) \to \spd(\zp)$ be an object in $\perf$ together with an untilt $S^\sharp$, and let $(\mathscr{P}, \phi_{\mathscr{P}})$ be a $\mathcal{G}$-shtuka over $S$ with one leg at $S^\sharp$. We can choose $r$ sufficiently large such that $\mathcal{Y}_{[r,\infty)}$ does not meet the divisor of $\mathcal{Y}_{[r,\infty)}$ defined by $S^\sharp$. The restriction of $(\mathscr{P},\phi_{\mathscr{P}})$ determines a $\phi=\operatorname{Frob}_{S}$-equivariant $\mathcal{G}$-torsor on $\mathcal{Y}_{[r,\infty)}$. By spreading out via the Frobenius (see \cite[Proposition 22.1.1]{ScholzeWeinsteinBerkeley}), the bundle $\mathscr{P}$ descends to a $G$-bundle $\mathcal{E}(\mathscr{P},\phi_{\mathscr{P}})$ on $\mathcal{X}_S$. In this way we obtain a morphism of v-stacks on $\perf$ 
\begin{align}\label{Eq:ShtukaMorphism}
    \shtg \to \bung, \quad (\mathscr{P},\phi_{\mathscr{P}}) \mapsto \mathcal{E}(\mathscr{P},\phi_{\mathscr{P}}).
\end{align}
and we will denote both this map and its restriction to $\shtgmu$ by $\mathrm{BL}^{\circ}$. 

Using $\mathrm{BL}^\circ$, we obtain locally closed substacks $\shtg^{[b]} \subset \shtg$ and $\shtgmu^{[b]}\subset \shtgmu$ defined as in \eqref{Eq:NewtonStratumNotation}. We will refer to these as the Newton strata corresponding to $[b]$ in $\shtg$ and $\shtgmu$, respectively.

\subsubsection{} For $\ell$ an
algebraically closed field in characteristic $p$ together with a fixed embedding
$e \colon k_E \hookrightarrow \ell$, write
\[
  \oeel{e} = \mathcal{O}_E \otimes_{W(k_E),e} W(\ell).
\]
We make the following definition, which is a slight generalization of
\cite[Definition 25.1.1]{ScholzeWeinsteinBerkeley}. 

\begin{Def} \label{Def:IntegralLocalShimuraVariety}
  Let $\ell$ be a perfect field of characteristic $p$ together with an embedding $e
  \colon k_E \hookrightarrow \ell$, and let $b \in G(W(\ell)[p^{-1}])$. The
  \emph{integral local Shimura variety}
  \[
    \mintgbmu{e} \to \spd
    \oeel{e}
  \]
  is the v-sheaf on $\perf$ that assigns to each perfectoid $S$ the set of
  isomorphism classes of tuples $(S^\sharp, \mathscr{P}, \phi_\mathscr{P},
  \iota_r)$, where $S^\sharp$ is an untilt of $S$ over $\oeel{e}$, where
  $(\mathscr{P}, \phi_\mathscr{P})$ is a $\mathcal{G}$-shtuka with one leg along
  $S^\sharp$ bounded by $\mu$, and $\iota_r$ is an isomorphism
  \[
    \iota_r \colon G \vert_{\mathcal{Y}_{[r,\infty)}(S)} \xrightarrow{\cong}
    \mathscr{P} \vert_{\mathcal{Y}_{[r,\infty)}(S)}
  \]
  for $r \gg 0$, which satisfies $\iota_r \circ \phi_\mathscr{P} = (b \times
  \mathrm{Frob}_S) \circ \iota_r$. Two tuples $(S^\sharp, \mathscr{P}, \phi_\mathscr{P},\iota_r)$, $(S^\sharp, \mathscr{P}', \phi_\mathscr{P}', \iota'_{r'})$ are isomorphic if there is an isomorphism of $\mathcal{G}$-shtukas $(\mathscr{P}, \phi_\mathscr{P}) \to (\mathscr{P}', \phi_\mathscr{P}')$ which is compatible with $\iota_r$ and $\iota'_{r'}$ after restricting to $\mathcal{Y}_{[R,\infty)}(S)$ for some $R \gg r,r'$.
\end{Def}
\begin{Lem} \label{Lem:LocalUniformisation}
    There is a Cartesian diagram
    \begin{equation}
        \begin{tikzcd}
            \mintgbmu{e} \arrow{r} \arrow{d} & \shtgmub \times_{\mathcal{O}_E} \oeel{e}  \arrow{d}{\mathrm{BL}^{\circ}} \\
            \spd(\ell) \arrow{r}{b} & \bun_{G}^{[b]},
        \end{tikzcd}
    \end{equation}
where $b:\spd(\ell) \to \bun_G^{[b]}$ is the map coming from \cite[Theorem 5.3]{AnschutzIsocrystal}. 
\end{Lem}
\begin{proof}
    This follows from unwinding the definition of the map $\mathrm{BL}^{\circ}:\shtgmu \to \bun_G$ and the definition of the sheaf $\mintgbmu{e}$.
\end{proof}
\begin{Rem} \label{Rem:BasePoint}
If 
\begin{align}
    b \in \bigcup_{w \in \admu} \mathcal{G}(W(\ell)) w \mathcal{G}(W(\ell)),
\end{align}
then $b$ defines an element $b \in \shtgmu(\spd(\ell))$ lifting $b \in  \bun_{G}(\spd(\ell))$, see \cite[Remark 4.2.3]{PappasRapoportShtukas}. The universal property of the fiber product diagram of Lemma \ref{Lem:LocalUniformisation} then gives us 
a tautological base point $x_0: \spd(\ell) \to \mintgbmu{e}$. 
\end{Rem}

\subsubsection{} We observe that it follows from the argument in \cite[Proposition 11.16]{ZhangThesis}\footnote{Note that although \cite[Proposition 11.16]{ZhangThesis} assumes $\calg$ being reductive, the reference to \cite{ScholzeWeinsteinBerkeley} cited in \textit{loc. cit.} only assumes that $\calg$ is smooth and has connected special fiber, so the argument works verbatim.} that there is an isomorphism (here the sheaf $\underline{\calgcirc(\zp)}$ is as in \cite[the discussion before Definition 10.12]{EtCohDiam})
\begin{align} \label{Eq:GenericFibreShtukas}
    c:\mathrm{Sht}_{\calgcirc,\mu,E} \to \left[\mathrm{Gr}_{G,\mu^{-1}} / \underline{\calgcirc(\zp)}\right].
\end{align}
We generally do \emph{not} have an isomorphism as in \eqref{Eq:GenericFibreShtukas} for $\mathcal{G}$-shtukas, as we will explain below in Corollary \ref{Cor:QuasiParahoricShtukasGeneric}. 
\begin{Lem} \label{Lem:NewtonOverHodge}
The map
\begin{align}
    \operatorname{BL}^{\circ}:\shtgcircmu \to \bung
\end{align}
factors through $\bungmu$
\end{Lem}
\begin{proof}
By definition, $\bungmu$ is the subfunctor of $\bung$ consisting of maps $X \to \bung$ for which $|X| \to |\bung| \cong B(G)$ factors over $\bgmu$, see Section \ref{Sec:BGMU}. It is therefore enough to show the factorization at the level of topological spaces. By Lemma \ref{Lem:LocalUniformisation} and the v-surjectivity of $b:\spd(\ovfp) \to \bun_{G}^{[b]}$, it suffices to show that $|\mintgcircbmu{e}|$ is empty unless $[b] \in \bgmu$. 

By \cite[Theorem 3.3.3]{PappasRapoportShtukas}, see \cite{GleasonComponents}, the reduction $\left(\mintgcircbmu{e}\right)^{\mathrm{red}}$ in the sense of \cite[Definition 3.12]{GleasonSpecialization} is isomorphic to the affine Deligne--Lusztig variety $X_{\calgcirc}(b, \mu^{-1})$ (see \cite[Definition 3.3.1]{PappasRapoportShtukas}). The space $X_{\calgcirc}(b, \mu^{-1})$ is empty unless $[b] \in \bgmu$ by \cite[Theorem A]{He2} and this along with \ \cite[Proposition 4.8.(4)]{GleasonSpecialization} implies that $|\mintgcircbmu{e}|$ is empty unless $[b] \in \bgmu$.
\end{proof}
Lemma \ref{Lem:NewtonOverHodge} will generally \emph{not} be true for $\shtgmu$, but we do have the following result: We recall from \cite[Theorem III.2.7]{FarguesScholze} that there is a locally constant map
\begin{align}
  \kappa_G:\lvert \bung \rvert \to \pi_1(G)_{\Gamma_p},
\end{align}
such that $\bungmu$ maps to $-\mu^{\natural}=\kappa_G(\mu^{-1})$. Here $\pi_1(G)$ is the algebraic fundamental group of $G$ and $\Gamma_p=\operatorname{Gal}(\qpbar/\qp)$. We let $\shtgmukappa \subset \shtgmu$ be the open and closed substack that is the preimage of $-\mu^\natural$ under $\kappa_G$.

\begin{Prop} \label{Prop:NewtonOverHodge}
  The map $\operatorname{BL}^{\circ}:\shtgmukappa \to \bung$ factors through $\bungmu$.
\end{Prop}

In the proof, we use the notation from Section~\ref{Subsub:DeltaNotation}.

\begin{proof}
  Since $\bungmu$ is an open substack of $\bung$, it is enough to check that for any map
  $\spa(C, C^+) \to \shtgmukappa$ with $C$ an algebraically closed
  perfectoid field, the induced map $\lvert \spa(C, C^+) \rvert \to B(G)$ on topological spaces has
  image contained in $\bgmu$. Using \cite[Proposition 2.1.1]{PappasRapoportShtukas}, we see that the restriction map
  \begin{align}
      \bun_G(\spa(C,C^+)) \to \bun_G(\spa(C, \mathcal{O}_C))
  \end{align}
  is an equivalence, and thus we may assume that  $C^+ = \mathcal{O}_C$. 
  Recall that $\shtgmu$ has a map to $\spd(\zp)$. We will verify the statement by dividing into the case when $C^\sharp$ has characteristic zero and when $C^\sharp = C$ has characteristic $p$. \smallskip

\noindent \textbf{Case 1:} First assume that $C^\sharp$ has characteristic zero, and let $(\mathscr{P},\phi_\mathscr{P})$ be a $\mathcal{G}$-shtuka with leg at $C^\sharp$ bounded by
  $\mu$. For sufficiently small $r$, the restriction
  $\mathscr{P} \vert_{\mathcal{Y}_{[0,r]}(C, \mathcal{O}_C)}$ defines a shtuka
  with no leg, and by \cite[Theorem~8.5.3]{KedlayaLiu}, this corresponds to an
  exact tensor functor $\mathsf{Rep}_{\zp}(\mathcal{G}) \to \mathsf{Mod}_{\zp}$
  which gives an element of $H_\mathrm{\acute{e}t}^1(\zp, \mathcal{G})$ by the
  Tannakian interpretation of torsors. Under the map (where $ B(G)_\mathrm{bsc}
  \subset B(G)$ denotes the subset of basic elements)
  \[
    H_\mathrm{\acute{e}t}^1(\zp, \mathcal{G}) \to H_\mathrm{\acute{e}t}^1(\qp,
    G) \hookrightarrow B(G)_\mathrm{bsc},
  \]
  this determines the isomorphism class $[b_0] \in B(G)_\mathrm{bsc}$ of the
  $C$-point of the $G$-bundle on $\mathcal{X}_{(C, \mathcal{O}_C)}$ coming from
  $\mathscr{P} \vert_{\mathcal{Y}_{(0,r]}(C,\mathcal{O}_C)}$, see \cite[Section~22.3]{ScholzeWeinsteinBerkeley}.
  This is because it corresponds to the exact tensor functor given by the
  composition $\mathsf{Rep}_{\zp}(\mathcal{G}) \to \mathsf{Mod}_{\zp} \to
  \mathsf{Mod}_{\qp} \to \mathsf{Vect}(\mathcal{X}_{(C, \mathcal{O}_C)})$, see
  the proof of \cite[Proposition~22.6.1]{ScholzeWeinsteinBerkeley}.
  The $G$-bundle $\mathcal{E}(\mathscr{P},\phi_\mathscr{P})$ comes from restricting to $\mathcal{Y}_{[r, \infty)}(C,\mathcal{O}_C)$ for $r \gg 0$, let us denote its isomorphism class by $[b] \in B(G)$. Then $\kappa([b_0]) - \kappa([b]) =
  -\mu^\natural$ as they are related by a modification bounded by $\mu$, see the proof of \cite[Proposition~25.3.2]{ScholzeWeinsteinBerkeley}. Since $(\mathscr{P},\phi_{\mathscr{P}})$ is assumed to be in $\shtgmukappa$, it follows
  that $\kappa([b_0]) = 0$. Moreover, since $[b_0]$ is basic, we see that $[b_0] = 0$.
  Therefore $[b] \in \bgmu$ by
  \cite[Proposition~9]{RapoportAccessiblePeriodDomain}. \smallskip

\noindent \textbf{Case 2:} Next, we consider the case where $C^\sharp = C$ has characteristic $p$. As before, let $(\mathscr{P}, \phi_\mathscr{P})$ be a $\mathcal{G}$-shtuka with leg at $C$ bounded by $\mu$. Using \cite[Proposition~2.1.3]{PappasRapoportShtukas} and \cite[Proposition~3.2.1]{PappasRapoportRZSpaces}, we may find a $\mathcal{G}$-torsor $\mathcal{P}$ on $\spec(W(\mathcal{O}_C))$ together with a
  meromorphic map $\phi_\mathcal{P} \colon \mathcal{P} \to \phi^\ast \mathcal{P}$
  that analytifies to $(\mathscr{P}, \phi_\mathscr{P})$. In particular, we may
  recover $[b] \in B(G)$ also by considering the isomorphism class of the
  $G$-isocrystal $(\mathcal{P}_{W(C)[p^{-1}]},
  \phi_{\mathcal{P}_{W(C)[p^{-1}]}})$.
  
  Fix a trivialization of $\mathcal{P} \cong \mathcal{G}_{W(\mathcal{O}_C)}$, so that the Frobenius $\phi_{\mathcal{P}_{W(C)[p^{-1}]}}$ gives us an element
  $b \in G(W(\mathcal{O}_C)[p^{-1}])$. The boundedness by $\mu$ condition now
  tells us that
  \[
    b \in \calgcirc(W(C)) \admu
    \mathcal{G}(W(C)).
  \]
Indeed, by definition the element $b$ lies in
\begin{align} \label{Eq:LocalModel}
     \locmodgmuv(C, \mathcal{O}_C) \subset \mathrm{Gr}_{\mathcal{G}}(C, \mathcal{O}_C) = G(W(C)[1/p])/\mathcal{G}(W(C)).
\end{align}
By \cite[Theorem~6.16]{AGLR}, we may identify 
\begin{align}
    \locmodgcircmuv(C, \mathcal{O}_C) \subset G(W(C)[1/p])/\calgcirc(W(C))
\end{align}
with 
\begin{align}
     \calgcirc(W(C)) \admu
    \calgcirc(W(C)) \subset G(W(C)[1/p])/\calgcirc(W(C)).
\end{align}
Since the natural map $\mathrm{Gr}_{\calgcirc} \to \mathrm{Gr}_{\mathcal{G}}$ induces an isomorphism 
\begin{align}
    \locmodgcircmuv \to \mathbb{M}_{\calg,\mu}^{\mathrm{v}},
\end{align}
we may identify \eqref{Eq:LocalModel} with
\begin{align}
    \calgcirc(W(C)) \admu
    \calg(W(C)) \subset  G(W(C)[1/p])/\mathcal{G}(W(C)).
\end{align}
We fix an embedding $\fpbar \hookrightarrow C$, so that we have a group homomorphism $G(\qpbr) \to G(W(C)[p^{-1}])$. We want to show that the class $[b] \in B(G)$, regarded as a $\phi$-conjugacy class in $G(W(C)[p^{-1}])$, is in $\bgmu$. Recall our assumption that
  $\kappa_G(b) = -\mu^\natural \in \pi_1(G)_{\Gamma_p}$. This means that
  \[
    \beta = -\tilde{\kappa}_G(b) - [\mu] \in \pi_0(\mathcal{G}) \subseteq
    \pi_1(G)_{I_p}
  \]
  can be written as $\beta = (1 - \phi) \gamma$ for some $\gamma \in
  \pi_1(G)_{I_p}$. 

  Choose a splitting of $\tilde{W} \to \pi_1(G)_{I_p}$ corresponding to a $\sigma$-stable alcove $\mathfrak{a}$ of $\mathcal{B}(G,\qpbr)$ with $\mathcal{F} \subset \mathfrak{a}$, as in Section \ref{Sec:Alcove}.\footnote{Here $\mathcal{F}$ is as in Section \ref{Subsub:DeltaNotation}.} By \cite[Lemma~4.3.1]{PappasRapoportRZSpaces}, there is a lift
  $\dot{\gamma}$ of $\gamma$ to $N(\qpbr)$ such that $\phi(\dot{\gamma})^{-1}
  \dot{\gamma} \in \mathcal{G}(\zpbr)$. We now have
  \begin{align*}
    b^\prime := \dot{\gamma} b \phi(\dot{\gamma})^{-1} = \dot{\gamma} (b
    \phi(\dot{\gamma})^{-1} \dot{\gamma}) \dot{\gamma}^{-1} &\in \dot{\gamma}
    \mathcal{G}(W(C)) \admu \mathcal{G}(W(C))
    \dot{\gamma}^{-1}. 
    \end{align*}
    As in the proof of
  \cite[Proposition~4.3.4]{PappasRapoportRZSpaces}, we may identify 
    \begin{align*}
    \dot{\gamma}\mathcal{G}(W(C)) \admu \mathcal{G}(W(C))
    \dot{\gamma}^{-1} &= \mathcal{G}_\delta(W(C))
    \admu \mathcal{G}_\delta(W(C)),
  \end{align*}
  where $\delta \in \Pi_G$ is the image of $\beta$ under $\pi_0(\mathcal{G}) \to
  \pi_0(\mathcal{G})_\phi$, so that $\mathcal{G}_\delta(W(C)) = \dot{\gamma}
  \mathcal{G}(W(C)) \dot{\gamma}^{-1}$. We then have
  \begin{align}
    \tilde{\kappa}_G(b^\prime) &= (1-\phi)\gamma + \tilde{\kappa}_G(b) \\
    &= (1 - \phi) \gamma - [\mu] - (1 - \phi)\gamma \\
    &=-[\mu] \in \pi_1(G)_{I_p}.
  \end{align}
  Moreover, as $b^\prime$ is a $\phi$-conjugate of $b$, it suffices to show that
  $[b^\prime] \in \bgmu$. By evaluating the isomorphism
  $\mathbb{M}_{\mathcal{G}_\delta^\circ,\mu}^\mathrm{v} \to
  \mathbb{M}_{\mathcal{G}_\delta,\mu}^\mathrm{v}$ on $(C,
  \mathcal{O}_C)$-points, we obtain the equality (as above)
  \[
    \mathcal{G}_\delta(W(C)) \admu
    \mathcal{G}_\delta(W(C)) = \mathcal{G}_\delta^\circ(W(C))
    \admu \mathcal{G}_\delta(W(C)).
  \]
  Thus we can write $b^\prime = hwg$ with $h \in
  \mathcal{G}_\delta^\circ(W(C))$, $w \in \admu$, and $g
  \in \mathcal{G}_\delta(W(C))$. By applying $\tilde{\kappa}_G$ on both sides,
  we obtain
  \[
    -[\mu] = \tilde{\kappa}_G(b^\prime) = \tilde{\kappa}_G(h) +
    \tilde{\kappa}_G(w) + \tilde{\kappa}_G(g) = -[\mu] + \tilde{\kappa}_G(g) \in
    \pi_1(G)_{I_p}
  \]
  This shows that $\tilde{\kappa}_G(g) = 0$, and hence
  \[
    g \in \mathcal{G}_\delta(W(C)) \cap \ker \tilde{\kappa}_G =
    \mathcal{G}_\delta^\circ(W(C)).
  \]
  It now follows that 
  \[
    b^\prime \in \mathcal{G}_\delta^\circ(W(C)) \admu
    \mathcal{G}_\delta^\circ(W(C))
  \]
  and thus by \cite[Theorem~A]{He2} that $[b] = [b^\prime] \in \bgmu$.
\end{proof}

\subsection{A group action on the moduli stack of parahoric shtukas}

Let $\mathcal{G}/\zp$ be a quasi-parahoric group scheme as before. The goal of
this section is to construct an action of $\pi_0(\mathcal{G})^\phi$ on
$\shtgcircmu$ together with a map $[\shtgcircmu / \pi_0(\mathcal{G})^\phi] \to
\shtgmu$.

\subsubsection{} Recall, e.g., from \cite[Section~4.4]{PappasRapoportRZSpaces},
that there is a short exact sequence
\[
  1 \to \mathcal{G}^\circ(\zp) \to \mathcal{G}(\zp) \to \pi_0(\mathcal{G})^\phi
  \to 1.
\]
For an element $g$ of $\pi_0(\mathcal{G})^\phi$ and a representation $(\Lambda,
\rho \colon \mathcal{G}^\circ \to \mathrm{GL}(\Lambda))$
in $\mathsf{Rep}_{\zp}(\mathcal{G}^\circ)$, we define the lattice
\[
  g \Lambda = \rho(\tilde{g}) \Lambda \subseteq \Lambda \otimes_{\zp} \qp,
\]
where $\tilde{g} \in \mathcal{G}(\zp) \subseteq G(\qp)$ is a lift of $g$. Note
that this does not depend on the choice of lift $\tilde{g}$ because $\Lambda$ is
stable under $\mathcal{G}^\circ(\zp)$.

\begin{Lem}
  The group homomorphism $\rho_{\qp} \colon G \to \mathrm{GL}_{\qp}(\Lambda
  \otimes_{\zp} \qp)$ extends (uniquely) to a homomorphism
  $g\rho:\mathcal{G}^\circ \to \mathrm{GL}_{\zp}(g\Lambda)$.
\end{Lem}

\begin{proof}
  Both $\mathcal{G}^\circ$ and $\mathrm{GL}_{\zp}(g\Lambda)$ are smooth
  affine integral models of their respective generic fibers. Therefore by \cite[Corollary~2.10.10]{KalethaPrasad}, it
  suffices to show that the image of $\mathcal{G}^\circ(\zpbr)$ is contained in
  $\mathrm{GL}_{\zpbr}(g\Lambda \otimes_{\zp} \zpbr)$. Since
  $\mathcal{G}^\circ(\zpbr) \subseteq \mathcal{G}(\zpbr)$ is normal, it is in
  particular stable under conjugation by $\tilde{g} \in \mathcal{G}(\mathbb{Z}_
  p)$ a lift of $g$. The claim now follows, as $\Lambda \otimes_{\zp} \zpbr$ is
  stable under the action of $\mathcal{G}^\circ(\zpbr)$.
\end{proof}

We see that $\rho \mapsto g \rho$ defines a $\pi_0(\mathcal{G})^\phi$-action on
the category $\mathsf{Rep}_{\zp}(\mathcal{G}^{\circ})$ by exact tensor
equivalences. Moreover, for any representation $(\Lambda, \rho)$ in the image of
the forgetful (exact tensor) functor $\mathsf{Rep}_{\zp}(\mathcal{G}) \to
\mathsf{Rep}_{\zp}(\mathcal{G}^\circ)$, we have $\Lambda = g\Lambda$ and hence
$\rho = g\rho$.

\subsubsection{} Using the Tannakian formalism, this defines a
$\pi_0(\mathcal{G})^\phi$-action on the groupoid of $\mathcal{G}^\circ$-torsors
on $\mathcal{Y}_{[0,\infty)}(S)$ for $S \in \perf$. More precisely, for each
$\mathcal{G}^\circ$-torsor $\mathscr{P}$ on $Y_{[0,\infty)}(S)$ and an element
$g \in \pi_0(\mathcal{G})^\phi$, there is a $\mathcal{G}^\circ$-torsor
$g\mathscr{P}$ characterized by the property that
\[
  \mathcal{G}^\circ \backslash ((g\mathscr{P}) \times \Lambda) =
  \mathcal{G}^\circ \backslash (\mathscr{P} \times g^{-1} \Lambda).
\]
By construction, if we push out to
$\mathcal{G}$, we see that there is a canonical isomorphism
\[
  \mathcal{G} \times^{\mathcal{G}^\circ} \mathscr{P} \cong \mathcal{G}
  \times^{\mathcal{G}^\circ} (g\mathscr{P}).
\]
We also see that there is a canonical meromorphic homomorphism
\[ \begin{tikzcd}
  \mathscr{P} \arrow[r, dashed, "g"] & g\mathscr{P}
\end{tikzcd} \]
that is an isomorphism away from $S = V(p) \hookrightarrow
\mathcal{Y}_{[0,\infty)}(S)$, coming from the fact that $\Lambda \otimes_{\zp}
\qp = g^{-1}\Lambda \otimes_{\zp} \qp$ by construction of $g^{-1}\Lambda$.

\begin{Rem} \label{Rem:ComparisonWithPR} The $\pi_0(\mathcal{G})^\phi$-action we construct is in fact equivalent to the
  one given in \cite[Section~4.4]{PappasRapoportRZSpaces}. If we choose a lift
  $\tilde{g} \in \mathcal{G}(\zp)$ and twist the $\mathcal{G}^\circ(\zp)$-torsor
  structure to get a new torsor $\mathscr{P}_{\tilde{g}}$, then $\tilde{g}^{-1}$ induces
  an isomorphism
  \begin{align*}
    \mathcal{G}^\circ \backslash (\mathscr{P}_{\tilde{g}} \times \Lambda) &= (\mathscr{P} \times
    \Lambda) / \left((\tilde{g}^{-1} h \tilde{g} x, h y) \sim (x, y)\right) \\
    &\xrightarrow{(\mathrm{id}, \tilde{g}^{-1})} (\mathscr{P} \times \tilde{g}^{-1}
    \Lambda) / \left((\tilde{g}^{-1} h \tilde{g} x, \tilde{g}^{-1} hy) \sim (x, \tilde{g}^{-1}
    y)\right) \\ &= (\mathscr{P} \times \tilde{g}^{-1} \Lambda) / \left((h^\prime x, h^\prime
    y^\prime) \sim (x, y^\prime)\right) = \mathcal{G}^\circ \backslash
    (\mathscr{P} \times \tilde{g}^{-1} \Lambda)
  \end{align*}
  by substituting $h^\prime = \tilde{g}^{-1} h \tilde{g}$ and $y^\prime = \tilde{g}^{-1}
  y$.
\end{Rem}

\subsubsection{} We now use this action to define an action of
$\pi_0(\mathcal{G})^\phi$ on $\shtgcircmu$. For $S \in \perf$ and a
$\mathcal{G}^\circ$-torsor $\mathscr{P}$ on $S \bdtimes \spa(\zp)$, first note
that $g(\frob_S^\ast \mathscr{P}) = \frob_S^\ast (g\mathscr{P})$, as both
correspond to the exact tensor functor
\[
  \mathsf{Rep}(\mathcal{G}^\circ) \xrightarrow{g^{-1}}
  \mathsf{Rep}(\mathcal{G}^\circ) \to \mathsf{Vect}(S \bdtimes \spa(\zp))
  \xrightarrow{\frob_S^\ast} \mathsf{Vect}(S \bdtimes \spa(\zp)).
\]
Given $(\mathscr{P}, \phi_\mathscr{P})$ a $\mathcal{G}^\circ$-shtuka, we now
define $\phi_{g\mathscr{P}}$ to be the meromorphic map
\begin{center}
    \begin{tikzcd}
        \frob_S^\ast (g\mathscr{P}) = g(\frob_S^\ast \mathscr{P}) 
            \arrow[r, dashed, "g^{-1}"]
        & \frob_S^\ast\mathscr{P} 
            \arrow[r, dashed, "\phi_\mathscr{P}"]
        & \mathscr{P}
            \arrow[r, dashed, "g"]
        & g\mathscr{P}.
    \end{tikzcd}
\end{center}

\begin{Prop} \label{Prop:ActionExist}
  Let $S \in \perf$ and let $(\mathscr{P}, \phi_\mathscr{P}) \in \shtgcircmu(S)$
  be a $\mathcal{G}^\circ$-shtuka. Then $(g\mathscr{P}, \phi_{g\mathscr{P}})$
  defines an object of $\shtgcircmu(S)$.
\end{Prop}

\begin{proof}
  We first check that $\phi_{g\mathscr{P}}$ only has poles at $R^\sharp$. This
  can be checked by considering the induced map on vector bundle shtukas for
  each $\Lambda \in \mathsf{Rep}_{\zp}(\mathcal{G}^\circ)$. We observe that the
  meromorphic map
  \begin{align*}
    \mathcal{G}^\circ \backslash (\frob_S^\ast \mathscr{P} \times g^{-1}
    \Lambda) &= \mathcal{G}^\circ \backslash (\frob_S^\ast (g\mathscr{P}) \times
    \Lambda) \dashrightarrow \mathcal{G}^\circ \backslash (\frob_S^\ast
    \mathscr{P} \times \Lambda) \\ &\dashrightarrow \mathcal{G}^\circ \backslash
    (\mathscr{P} \times \Lambda) \dashrightarrow \mathcal{G}^\circ \backslash
    (g\mathscr{P} \times \Lambda) = \mathcal{G}^\circ \backslash (\mathscr{P}
    \times g^{-1} \Lambda)
  \end{align*}
  is identified with $\phi_\mathscr{P} \times \mathrm{id}_{g^{-1}
  \Lambda}$ becasue the natural diagram obtained by pushing out
  $\phi_\mathscr{P} \colon \frob_S^\ast \mathscr{P} \dashrightarrow \mathscr{P}$
  along $g^{-1}\Lambda \dashrightarrow \Lambda$ commutes. The composition, which
  a priori is only an isomorphism away from $\spa(R^\sharp, R^{\sharp+}) \cup
  \spa(R, R^+)$, is therefore an isomorphism away from $\spa(R^\sharp,
  R^{\sharp+})$. Next, we need to show that the modification is bounded by
  $\mu$. Choosing a lift $\tilde{g} \in \mathcal{G}(\zp)$ of $g$ as in
  Remark~\ref{Rem:ComparisonWithPR}, this follows from the stability of
  $\mathbb{M}^{\mathrm{v}}_{\mathcal{G}^{\circ},\mu}$ under conjugation by
  $\tilde{g}$.
\end{proof}
We define the action of $\pi_0(\mathcal{G})^\phi$ on $\shtgcircmu$ by
\[
  g \colon (\mathscr{P}, \phi_\mathscr{P}) \mapsto (g\mathscr{P},
  \phi_{g\mathscr{P}}).
\]
Since we canonically have $\mathcal{G} \times^{\mathcal{G}^\circ} \mathscr{P} 
\cong \mathcal{G} \times^{\mathcal{G}^\circ} g\mathscr{P}$, the construction
$\mathscr{P} \mapsto \mathcal{G} \times^{\mathcal{G}^\circ} \mathscr{P}$
naturally induces a map
\[
  [\shtgcircmu / \pi_0(\mathcal{G})^\phi] \to \shtgmu.
\]

\subsection{Decomposition of the moduli stack of quasi-parahoric shtukas} \label{Sec:Decomposition} Our goal in this section is to prove Theorem \ref{Thm:IntroShtukas}. 
 We once again assume $\mathcal{G}/\zp$ is a quasi-parahoric group scheme with associated parahoric $\mathcal{G}^\circ$. As in \ref{Subsub:DeltaNotation}, we denote by $\Pi_\mathcal{G}$ the kernel of $H^1(\zp,\mathcal{G}) \to H^1(\qp,\mathcal{G})$. 
 
\subsubsection{} For each $\delta \in \Pi_{\mathcal{G}}$, fix a choice of
$\dot{\gamma} \in N(\qpbr)$ as in Section~\ref{Subsub:DeltaNotation}. We now construct maps
\[
  \shtgdeltamu \to \shtgmubr
\]
following \cite[Section~4.4]{PappasRapoportRZSpaces}. Suppose we are given a $\calgdelta$-shtuka
$(\mathscr{P}, \phi_\mathscr{P})$ over $\spa(R, R^+)$, where
the leg is over an untilt $\oebreve \to R^{\sharp+}$. Note that $R$ is
canonically an $\fpbar$-algebra, and hence $\spa(R, R^+) \bdtimes \spa(\zp)$
naturally lives over $\spa(\zpbr)$. There is an isomorphism of group schemes
\[
  \operatorname{Int} \dot{\gamma}^{-1} \colon \mathcal{G}_{\delta,\zpbr} \to
  \mathcal{G}_{\zpbr}; \quad g \mapsto \dot{\gamma}^{-1} g \dot{\gamma},
\]
and hence we can push out the $\mathcal{G}_\delta$-torsor $\mathscr{P}$ to a
$\mathcal{G}$-torsor
\[
  \mathscr{P}_{\dot{\gamma}} = \mathcal{G}_{\zpbr}
  \times^{\mathcal{G}_{\delta,\zpbr}} \mathscr{P}.
\]
This is equivalent to the description in \cite[Section 4.4]{PappasRapoportRZSpaces}, which is
phrased in terms of twisting the $\mathcal{G}$-action.

\subsubsection{} We now construct the Frobenius action on
$\mathscr{P}_{\dot{\gamma}}$. For $g \in \mathcal{G}_\delta(\zpbr)$ we have
\[
  \phi(\dot{\gamma}^{-1} g \dot{\gamma}) = \phi(\dot{\gamma}^{-1}) \phi(g)
  \phi(\dot{\gamma}) = \dot{\delta} (\dot{\gamma}^{-1} \phi(g) \dot{\gamma})
  \dot{\delta}^{-1},
\]
where $\dot{\delta} \coloneqq \phi(\dot{\gamma})^{-1} \dot{\gamma} \in
\mathcal{G}(\zpbr)$, and hence the diagram
\[ \begin{tikzcd}
  \frob_{\zpbr}^{\ast} \mathcal{G}_{\delta,\zpbr}
  \arrow{d}[']{\frob_{\zpbr}^{\ast} \operatorname{Int} \dot{\gamma}^{-1}}
  \arrow{rr}{\phi_{\calgdelta}} & & \mathcal{G}_{\delta,\zpbr}
  \arrow{d}{\operatorname{Int} \dot{\gamma}^{-1}} \\ \frob_{\zpbr}^{\ast}
  \mathcal{G}_{\zpbr} \arrow{r}{\phi_{\mathcal{G}}} & \mathcal{G}_{\zpbr} &
  \mathcal{G}_{\zpbr} \arrow{l}[']{\conj{\dot{\delta}}}
\end{tikzcd} \]
commutes. Writing $S = \spa(R, R^+)$ as usual, since $\frob_S^\ast \mathscr{P}_{\dot{\gamma}}$ is the pushforward of $\frob_S^\ast \mathscr{P}$ along $\frob_{\zpbr}^\ast \operatorname{Int}_{\dot{\gamma}^{-1}}$, it follows from the diagram above that $\frob_S^\ast \mathscr{P}_{\dot{\gamma}}$ is isomorphic to the $\mathcal{G}_{\zpbr}$-torsor 
\[
  \mathcal{G}_{\zpbr} \times^{\conj{\dot{\delta}},
  \mathcal{G}_{\zpbr}} (\mathcal{G}_{\zpbr} \times^{\mathcal{G}_{\delta,\zpbr}}
  \frob_S^\ast \mathscr{P}).
\]

Therefore, using $\phi_\mathscr{P}$, we may construct the meromorphic map
\begin{center}
    \begin{tikzcd} 
        \phi_{\mathscr{P}_{\dot{\gamma}}}\colon \frob_S^\ast \mathscr{P}_{\dot{\gamma}} = \mathcal{G}_{\zpbr} \times^{\conj{\dot{\delta}}, \mathcal{G}_{\zpbr}} (\mathcal{G}_{\zpbr} \times^{\mathcal{G}_{\delta,\zpbr}} \frob_S^\ast \mathscr{P})
           \arrow[r, dashed, "(\operatorname{Int} \dot{\delta}^{-1}{, \ }\phi_{\mathscr{P}})"]
        &[3em] \mathscr{P}_{\dot{\gamma}}.
    \end{tikzcd}
\end{center}

\begin{Prop} \label{Prop:ConjugationExist}
  For $S \in \perf$ and $(\mathscr{P}, \phi_\mathscr{P})$ an $S$-point of
 $\shtgdeltamu$, the induced shtuka $(\mathscr{P}_{\dot{\gamma}},
  \phi_{\mathscr{P}_{\dot{\gamma}}})$ defines an $S$-point of $\shtgmubr$.
\end{Prop}

\begin{proof}
  As in the proof of Proposition~\ref{Prop:ActionExist}, this follows from the
  fact that conjugation by $\dot{\gamma}^{-1}$ induces an isomorphism between
  the local models $\mathbb{M}_{\mathcal{G}_\delta, \mu
  , \oebreve}^\mathrm{v}$ and $\mathbb{M}_{\mathcal{G},  \mu
  , \oebreve}^\mathrm{v}$, together with the fact that $\mathbb{M}_{\mathcal{G},
   \mu , \oebreve}^\mathrm{v}$ is stable under conjugation by
 $\dot{\delta}^{-1}$.
\end{proof}

\subsubsection{} By combining Proposition~\ref{Prop:ActionExist},
Proposition~\ref{Prop:ConjugationExist}, and Lemma~\ref{Lem:NewtonOverHodge}, we
obtain a map
\[
  [\shtgdeltacircmu/\pi_0(\mathcal{G}_\delta)^\phi] \to \shtgmubrkappa
\]
for each $\delta \in \Pi_\mathcal{G}$. We now have the following key result.

\begin{Thm} \label{Thm:QuasiParahoricShtukas}
  The map
  \begin{align}
    \coprod_{\delta \in \Pi_\mathcal{G}}
    [\shtgdeltacircmu/\pi_0(\mathcal{G}_\delta)^\phi] \to \shtgmubrkappa
  \end{align}
  is an isomorphism.
\end{Thm}

The strategy of the proof is to reduce to the statement for rank-one geometric
points, and then verify the isomorphism on each Newton stratum using Proposition \ref{Prop:NewtonOverHodge} and \cite[Proposition~4.3.4]{PappasRapoportRZSpaces}.

\begin{Lem} \label{Lem:ProperStar}
  For every quasi-parahoric group $\mathcal{G} / \zp$ and geometric conjugacy
  class of a cocharacter $ \mu $ with reflex field $E$, the map
 $\shtgmu \to \spd(\mathcal{O}_E)$ is proper*. 
\end{Lem}

\begin{proof}
  As noted after \cite[Lemma~2.4.4]{PappasRapoportShtukas}, the exact tensor
  category of vector bundle shtukas on $\spa(R, R^+)$ agrees with that of
 $\spa(R, R^\circ)$, and therefore the map $\shtg \to \spd(\zp)$ has uniquely
  existing lifts along
  \[
    \textstyle \ProdRkOne = \spa((\prod_i C_i^+)[\varpi^{-1}], \prod_i
    \mathcal{O}_{C_i}) \to \spa((\prod_i C_i^+)[\varpi^{-1}], \prod_i C_i^+) =
    \ProdPts.
  \]
  On the other hand, $\locmodgmuv
  \hookrightarrow \operatorname{Gr}_{\mathcal{G},\spd(\mathcal{O}_E)}$ is a
  closed immersion, so the map $\shtgmu \to \spd(\mathcal{O}_E)$ also has
  uniquely existing lifts along $\ProdRkOne \to \ProdPts$ by
  Lemma~\ref{Lem:ProperRepresentableLifting}. 

  We now produce uniquely existing lifts along
  \[
    \textstyle\RkOne = \coprod_i \spa(C_i, \mathcal{O}_{C_i}) \to \spa((\prod_i
    \mathcal{O}_{C_i})[\varpi^{-1}], \prod_i \mathcal{O}_{C_i}) = \ProdRkOne,
  \]
  following the argument of \cite[Proposition~11.10]{ZhangThesis}. Using
  \cite[Proposition~9.5]{GleasonIvanov},
  \cite[Theorem~3.8]{KedlayaAlgebraization}, and
  \cite[Proposition~3.2.2]{PappasRapoportRZSpaces}, we see that an
 $\ProdRkOne$-point of $\shtg$ corresponds to a $\mathcal{G}$-torsor
 $\mathscr{P}$ on $\spec(W(\prod_i \mathcal{O}_{C_i}))$ together with a
  meromorphic map $\phi_\mathscr{P} \colon \mathrm{Frob}^\ast \mathscr{P}
  \dashrightarrow \mathscr{P}$, and similarly for each $s_i$-point. Using the Tannakian formalism and that $W(\mathcal{O}_{C_i})$ are local rings, we first
  observe that the groupoid of $\mathcal{G}$-torsors over $W(\prod_i
  \mathcal{O}_{C_i}) = \prod_i W(\mathcal{O}_{C_i})$ is canonically equivalent
  to the product of the groupoids of $\mathcal{G}$-torsors over
 $W(\mathcal{O}_{C_i})$. Next, to control the meromorphic Frobenius action, we
  use the fact that $\locmodgmuv \to \spd(\mathcal{O}_E)$ is proper*, which
  follows from it being proper and representable, together with
  Lemma~\ref{Lem:ProperRepresentableLifting}. By trivializing the
 $\mathcal{G}$-torsors, this implies that given a collection of 
 $\mathcal{G}$-torsors on each $W(\mathcal{O}_{C_i})$ with meromorphic
  Frobenius actions bounded by $\mu$, their product is a
 $\mathcal{G}$-torsor on $\prod_i W(\mathcal{O}_{C_i})$ with meromorphic
  Frobenius action again bounded by $\mu$.

\end{proof}

\begin{proof}[Proof of Theorem~\ref{Thm:QuasiParahoricShtukas}]
  We first check that the map is proper*. By
  Lemma~\ref{Lem:LiftsTwoOutOfThree}, it suffices to show that the structure maps
 $[\shtgdeltacircmu/\pi_0(\mathcal{G}_\delta)^\phi] \to \spd(\oebreve)$ and
 $\shtgmubrkappa \to \spd(\oebreve)$ are proper*. This follows by combining
  Lemma~\ref{Lem:BaseChangeLifting}, Lemma~\ref{Lem:LiftsTwoOutOfThree}, and
  Lemma~\ref{Lem:ProperStar}.

  At this point, it suffices to show that for every algebraically closed
  perfectoid field $C$ the map
  \[
    \coprod_{\delta \in \Pi_\mathcal{G}}
    [\shtgdeltacircmu/\pi_0(\mathcal{G}_\delta)^\phi](C) \to \shtgmubrkappa(C)
  \]
  is an equivalence of groupoids. We can verify this one Newton stratum at a time, and by Proposition \ref{Prop:NewtonOverHodge} and Lemma \ref{Lem:NewtonOverHodge}, we only need to work with Newton strata corresponding to elements in $\bgmu$. For $[b] \in \bgmu$ choose $b \in G(\qpbr)$ with $b \in [b]$. Then by Lemma~\ref{Lem:LocalUniformisation}, we may identify the restriction of the map in the statement of Theorem \ref{Thm:QuasiParahoricShtukas} to the Newton stratum corresponding to $[b]$, with the map
  \[
    \coprod_{\delta \in \Pi_{\mathcal{G}}}
    \left[\left(\mintgdeltacircmu/\pi_0(\mathcal{G}_{\delta})^{\phi}\right) /
    \tilde{G}_b\right] \to \left[\mintgmu/\tilde{G}_b\right].
  \]
  By construction, the induced map
  \[\coprod_{\delta \in \Pi_{\mathcal{G}}} \mintgdeltacircmu/\pi_0(\mathcal{G}_{\delta})^{\phi} \to \mintgmu\]
  agrees with
  the one constructed by Pappas and Rapoport in
  \cite[Equation (4.4.1)]{PappasRapoportRZSpaces}, which by
  Theorem 4.4.1 of \textit{loc. cit.} is an isomorphism. Thus the natural map in \ref{Thm:QuasiParahoricShtukas} is a bijection on
  rank one geometric points, and by Lemma~\ref{Lem:ProperStar} it is also a
  bijection on products of geometric points. Since both sides are v-stacks, while products of points form a basis of the v-topology by \cite[Remark 1.3]{GleasonSpecialization}, we are done.
\end{proof}

\subsubsection{} \label{subsub:DeltaOneStack} It follows from Theorem \ref{Thm:QuasiParahoricShtukas} that the image of the map $\shtgcircmu \to \shtgmu$ is an open and
closed substack; we will denote this image by $\shtgmuone$. 

\begin{Cor} \label{Cor:QuasiParahoricShtukas}
  The natural map $$\shtgcircmu \to \shtgmuone$$ is a torsor for the abelian group $\pi_0(\mathcal{G})^\phi$.
\end{Cor}

\begin{proof}
  By Theorem~\ref{Thm:QuasiParahoricShtukas}, the map is
  finite \'{e}tale upon base changing along $\spd(\oebreve) \to \spd(\mathcal{O}_E)$. Since
  $\spd(\oebreve) \to \spd(\mathcal{O}_E)$ is v-surjective, the map $\shtgcircmu \to
  \shtgmu$ is also finite \'{e}tale according to \cite[Corollary~9.11]{EtCohDiam}.
\end{proof}

\begin{Cor} \label{Cor:QuasiParahoricShtukasGeneric}
  There is a natural isomorphism
  \[
    \shtgmuone \times_{\spd(\mathcal{O}_E)} \spd(E) \simeq
    \left[\mathrm{Gr}_{G,\mu^{-1}} / \underline{\mathcal{G}(\zp)}\right].
  \]
\end{Cor}

\begin{proof}
  This is true for $\shtgcircmu$ by \cite[Proposition~11.16]{ZhangThesis}, and
  the result now follows from the short exact sequence
  \[
    1 \to \mathcal{G}^\circ(\zp) \to \mathcal{G}(\zp) \to
    \pi_0(\mathcal{G})^\phi \to 1.
  \]
  and Corollary \ref{Cor:QuasiParahoricShtukasGeneric}.
\end{proof}

\subsubsection{} Now let $\mathcal{H}$ be another quasi-parahoric model of $G$ such that $\calgcirc \subset \mathcal{H} \subset \mathcal{G}$. Then we have the following corollary.
\begin{Cor} \label{Cor:QuasiParahoricShtukasIII}
    The natural map $\shthmuone \to \shtgmuone$ is a torsor for the finite abelian group $\pi_0(\mathcal{G})^{\phi}/\pi_0(\mathcal{H})^{\phi}$.
\end{Cor}
\begin{proof}
Let $\pi_0(\mathcal{H}) \subset \pi_0(\mathcal{G})$ be the inclusion induced by $\mathcal{H} \subset \mathcal{G}$. Then applying the discussion in Section \ref{subsub:DeltaOneStack} to both $\mathcal{H}$ and $\mathcal{G}$, we can identify the map in concern with the natural map
    \begin{align}
        \left[\shtgcircmu / \pi_0(\mathcal{H})^{\phi} \right] \to  \left[\shtgcircmu / \pi_0(\mathcal{G})^{\phi} \right],
    \end{align}
    which is clearly a torsor for $\pi_0(\mathcal{G})^{\phi}/\pi_0(\mathcal{H})^{\phi}$.
\end{proof}
\begin{Rem} \label{Rem:QuasiParahoricRemark}
    The subgroup $\mathcal{H} \subset \mathcal{G}$ is \emph{not} determined by the subgroup $\mathcal{H}(\zp) \subset \mathcal{G}(\zp)$, because the latter only depends on $\pi_0(\mathcal{H})^{\phi}$ and not on $\pi_0(\mathcal{H})$ itself. Nevertheless, Corollary \ref{Cor:QuasiParahoricShtukasIII} tells us that the stack $\shthmuone$ only depends on $\mathcal{H}(\zp)$. 
    
    Indeed, if $\mathcal{H}_1$ and $\mathcal{H}_2$ are two quasi-parahoric models of $G$ such that $\mathcal{H}_1(\zp)=\mathcal{H}_2(\zp) \subset G(\qp)$, then $$\mathcal{H}_1^{\circ}(\zp)=\mathcal{H}_1(\zp) \cap G(\qp)^{0}=\mathcal{H}_2(\zp) \cap G(\qp)^0=\mathcal{H}_2^{\circ}(\zp).$$ Thus the identity components $\mathcal{H}_1^{\circ}$ and $\mathcal{H}_2^{\circ}$ are parahoric integral models of $G$ with the same $\zp$-points, and therefore they must be isomorphic. Since $\mathcal{H}_1(\zp)=\mathcal{H}_2(\zp)$ we moreover find that $\pi_0(\mathcal{H}_1)^{\phi}=\pi_0(\mathcal{H}_2)^{\phi}$ as subgroups of $\pi_1(G)_{I_p}^{\phi}$. Corollary \ref{Cor:QuasiParahoricShtukasIII} and its proof now tell us that there is an isomorphism
    \begin{align}
        \operatorname{Sht}_{\mathcal{H}_1, \mu, \delta=1} \simeq \operatorname{Sht}_{\mathcal{H}_2, \mu, \delta=1}.
    \end{align}
\end{Rem}

\section{Conjectural canonical integral models} \label{Sec:ConjecturalModels} Let $\gx$ be a Shimura datum with reflex field $\mathsf{E}$, let $p$ be a prime and write $G=\mathsf{G}_{\qp}$. Let $\mathcal{G}$ be a quasi-parahoric model of $G$ over $\zp$, and let $K_p = \mathcal{G}(\zp)$. Choose a prime $v$ of $\mathsf{E}$ above $p$, and let $E$ denote the completion of $\mathsf{E}$ at $v$. Let $\mu$ denote the $G(\qpbar)$-conjugacy class of cocharacters of $G$ corresponding to $X$ and $v$. We will write $\mathcal{O}_E$ for the ring of integers of $E$ and $k_E$ for its residue field. For $K^p \subset \gafp$ a sufficiently small compact open subgroup we write $K=K_pK^p$. Associated to $\gx$ and $K^p$ is the Shimura variety $\mathbf{Sh}_K\gx$, which we view as an $E$-scheme (i.e., we take the base change to $E$ of the canonical model over $\mathsf{E}$). 

We will often consider Shimura varieties with infinite level structures. In particular, we let
\begin{equation}\label{Eq:SVTower}
    \mathbf{Sh}_{K^p}\gx = \varprojlim_{K_p' \subset K_p} \mathbf{Sh}_{K_p'K^p}\gx
\end{equation}
as $K_p'\subset K_p$ varies over all compact open subgroups of the fixed $K_p$, and let
\begin{equation}
    \mathbf{Sh}_{K_p}\gx = \varprojlim_{K^p\subset \gafp} \mathbf{Sh}_{K_pK^p}\gx
\end{equation}
as $K^p$ varies over all sufficiently small compact open subgroups $K^p \subset \gafp$.

Let $\mathsf{Z}^\circ$ denote the connected component of the center of $\mathsf{G}$. We will assume that $\gx$ satisfies
\begin{equation}\label{Eq:SV5}
    \operatorname{rank}_\mathbb{Q}(\mathsf{Z}^\circ) = \operatorname{rank}_\mathbb{R}(\mathsf{Z}^\circ).
\end{equation}
This equality is equivalent to Milne's axiom SV5 \cite[p.63]{Milne} by \cite[Lemma 1.5.5]{KisinShinZhu}.
\begin{Rem}
    By \cite[Lemma 5.1.2.(i)]{KisinShinZhu}, the assumption \eqref{Eq:SV5} is satisfied whenever $\gx$ is of Hodge type, which will be the main case of interest to us.
\end{Rem}

\subsection{Canonical integral models, after Pappas--Rapoport}

\subsubsection{Shtukas} \label{Subsub:ShtukasGenericConstruction}
Each finite level Shimura variety $\mathbf{Sh}_{K_p'K^p}\gx$ is a smooth algebraic variety over $E$, and the transition maps in the tower \eqref{Eq:SVTower} are finite \'etale. We denote by $\mathbb{P}_K$ the pro-\'etale $\mathcal{G}(\zp)$-cover 
\begin{align}
    \mathbf{Sh}_{K^p}\gx \to \mathbf{Sh}_K\gx.
\end{align}
Let $\mu$ denote the $G(\qpbar)$-conjugacy class of cocharacters of $G$ coming from the Hodge cocharacter and the place $v$. There is a $G(\qp)$-equivariant Hodge--Tate period map $\mathbf{Sh}_{K^p}\gx^{\diamondsuit} \to \operatorname{Gr}_{G,\mu^{-1}}$, see \cite[Proposition 4.1.2]{PappasRapoportShtukas} or \cite[Corollary 4.1.5]{RodriguezCamargo}. Thus we have a map $\mathbf{Sh}_K\gx^{\diamondsuit} \to \left[\operatorname{Gr}_{G,\mu^{-1}}/\mathcal{G}(\zp)\right]$, which by Corollary \ref{Cor:QuasiParahoricShtukasGeneric} gives us a map
\begin{align}
    \mathbf{Sh}_K\gx^{\diamondsuit} \to \shtgmuone \subset \shtgmu. 
\end{align}
We denote the corresponding $\mathcal{G}$-shtuka by $\mathscr{P}_{K,E}$. Inspired by the axioms in \cite[Conjecture 4.2.2]{PappasRapoportShtukas}, we make the following definition.
\begin{Def}\label{Def:PRAxioms}
    Let $\{\scrs_K\gx\}_{K^p \subset \mathsf{G}(\afp)}$ be a system of normal schemes that are flat, separated and of finite-type over $\mathcal{O}_E$, with generic fiber $\mathbf{Sh}_K\gx$, and with $K^p$ varying over all sufficiently small compact open subgroups of $\mathsf{G}(\afp)$. We say the system $\{\scrs_K\gx\}_{K^p}$ is a \textit{canonical integral model for} $\{\mathbf{Sh}_{K}\gx\}_{K^p}$ if the following properties are satisfied:
    \begin{enumerate}[(i)]
        \item For every discrete valuation ring $R$ of characteristic $(0,p)$ over $\mathcal{O}_E$, 
        \begin{align}
            \mathbf{Sh}_{K_p}\gx(R[1/p]) = \left(\varprojlim_{K^p} \scrs_K\gx\right)(R).
        \end{align}
        \item For every $K^p \subset \mathsf{G}(\afp)$, $g \in \mathsf{G}(\afp)$, and ${K'}^p$ with $g{K'}^pg^{-1} \subset K^p$, there are finite \'etale morphisms $[g]: \mathscr{S}_{K'}\gx \to \mathscr{S}_K\gx$ extending the natural maps on the generic fiber.
        \item The $\mathcal{G}$-shtuka $\mathscr{P}_{K,E}$ on $\mathbf{Sh}_K\gx^\diamondsuit$ extends to a $\mathcal{G}$-shtuka $\mathscr{P}_K$ on $\scrs_K\gx^{\diamondsuit /}$ for every sufficiently small $K^p \subset \mathsf{G}(\afp)$. 
        \item Let $\ell$ be an algebraically closed field of characteristic $p$ together with an embedding $e \colon k_E \hookrightarrow \ell$. For $x \in \scrs_K\gx(\ell)$ with corresponding $b_x \in \shtgmu(\spd(\ell))$, let $x_0 \in \mintgbxmu{e}(\spd(\ell))$ be the base point as in Remark \ref{Rem:BasePoint}. Then there is an isomorphism of completions
        \begin{align}
            \Theta_x: \widehat{\mintgbxmu{e}}_{/x_0} \xrightarrow{\sim} (\widehat{{\scrs_K\gx}_{\oeel{e},/x}})^\diamondsuit,
        \end{align}
        under which the shtuka $\Theta_x^\ast(\mathscr{P}_K)$ agrees with the universal shtuka $\mathscr{P}^\mathrm{univ}$ on $\mintgbxmu{e}$ coming from the map $\mintgbxmu{e} \to \shtgmu$ of Lemma \ref{Lem:LocalUniformisation}. Here the left hand side is defined as in \cite[Definition 4.18]{GleasonSpecialization}, see the explanation in \cite[Section 3.3.1-2]{PappasRapoportRZSpaces}.
    \end{enumerate}
\end{Def}

\begin{Rem}
    The extension of the $\mathcal{G}$-shtuka in (iii) is necessarily unique up
    to unique isomorphism. As in the proof of
    \cite[Corollary~2.7.10]{PappasRapoportShtukas}, even for quasi-parahoric
    groups $\mathcal{G}$ we can use the Tannakian formalism to reduce to \cite[Theorem~2.7.7]{PappasRapoportShtukas}.
\end{Rem}

The following conjecture is an extension of \cite[Conjecture 4.2.2]{PappasRapoportShtukas} to the case of quasi-parahoric $\mathcal{G}$.

\begin{Conj} \label{Conj:PR}
  For every Shimura datum $\gx$ satisfying \eqref{Eq:SV5} and $\mathcal{G} /
  \zp$ a quasi-parahoric model of $G$, if we set $K_p = \mathcal{G}(\zp)$, then
  there exists a system of canonical integral models $\{\scrs_K\gx\}_{K^p}$ of
  $\{\mathbf{Sh}_{K}\gx\}_{K^p}$.
\end{Conj}

By \cite[Theorem 4.5.2]{PappasRapoportShtukas}, a system of canonical integral models exists in the case of a Hodge-type Shimura datum under the additional assumption that $\mathcal{G}$ is a stabilizer parahoric (see Definition~\ref{Def:StabilizerParahoric}). The conjecture is also known to hold if $\gx$ is of toral type (i.e., if $\mathsf{G} = \mathsf{T}$ is a torus) and $\mathcal{G}$ is parahoric, by \cite[Theorem A]{Daniels}.\footnote{In fact, an extension of Conjecture \ref{Conj:PR} is proven in \cite{Daniels} for $\gx$ of toral type which do not necessarily satisfy \eqref{Eq:SV5}. In this case one needs to work with a variant $\mathcal{G}^c$ of $\mathcal{G}$; see \cite[Section 4.2 and Section 4.3]{Daniels} for details.} We will show in Section \ref{Sec:HodgeType}, see Theorem \ref{Thm:Main}, that the conjecture holds for all Hodge type Shimura data $\gx$ and all quasi-parahoric models $\mathcal{G}$.

\begin{Rem} \label{Rem:ShtukaOne}
    The map $\scrs_K\gx^{\diamondsuit /} \to \shtgmu$ automatically factors through $\shtgmuone$ if it exists. Indeed, the inclusion $\shtgmuone \to \shtgmu$ is open and closed and the factorization property is true when restricted to $\mathbf{Sh}_{K}\gx^{\diamondsuit} \subset \scrs_K\gx^{\diamondsuit /}$ by construction. We now conclude using the fact that the inclusion $\mathbf{Sh}_{K}\gx^{\diamondsuit} \to \scrs_K\gx^{\diamondsuit /}$ induces a surjection on $\pi_0$. Indeed, taking $\pi_0$ of the pushout diagram
    \[
    \begin{tikzcd}
    (\scrs_K\gx^{\diamond})_E\arrow{r} \arrow{d}& \scrs_K\gx^{\diamond}\arrow{d}\\
    \mathbf{Sh}_{K}\gx^{\diamondsuit}\arrow{r} &\scrs_K\gx^{\diamondsuit/}
    \end{tikzcd}
    \]
    gives a pushout diagram of $\pi_0$'s. But we know that the top arrow is surjective on connected components by flatness of $\scrs_{K}\gx$ and \cite[Lemma 2.17]{AGLR}. This implies the desired surjectivity on $\pi_0$ for the bottom arrow. This also shows that the basepoint $x_0$ lies in the image of $\mintgcircbxmu{e} \to \mintgbxmu{e}$.
    
\end{Rem}
\begin{Rem} \label{Rem:Independence}
    Recall from Remark \ref{Rem:QuasiParahoricRemark} that quasi-parahoric models $\mathcal{G}$ are typically not determined by their set of $\zp$-points. Thus a priori it is possible that one could use different quasi-parahoric models $\mathcal{G}$ with the same $\zp$-points to give rise to different axioms for integral models of the Shimura variety of level $\mathcal{G}(\zp)$. However, by Remark \ref{Rem:QuasiParahoricRemark} and Remark \ref{Rem:ShtukaOne}, this does not happen.
\end{Rem}

\subsubsection{} Let $\iota: \gx \to \gxp$ be a closed embedding of Shimura data. Write $\mathsf{E}$, $\mathsf{E'}$ for the corresponding reflex fields. Choose a place $v$ of $\mathsf{E}$ above $p$ and let $v'$ be the induced place of $\mathsf{E}' \subset \mathsf{E}$; we let $E' \subset E$ denote the induced map on completions. Let $\mathcal{G}$ and $\mathcal{G}'$ be quasi-parahoric models of $G$ and $G'$ respectively; write $K_p=\mathcal{G}(\zp)$ and $U_p=\mathcal{G}'(\zp)$. We assume that $K_p=\iota^{-1}(U_p)\cap G(\qp)$. For every sufficiently small compact open subgroup $K^p \subset \gafp$ we choose $U^p \subset \gp(\afp)$ such that $\iota$ induces a closed immersion (see \cite[Lemma 2.1.2]{KisinModels})
\begin{align} \label{Eq:MorphismShimuraData}
    \mathbf{Sh}_{K}\gx \to \mathbf{Sh}_{U}\gxp \times_{\spec(E')} \spec(E),
\end{align}
where $U=U^pU_p$ and $K=K^pK_p$. We have the following version of \cite[Theorem 4.3.1, Theorem 4.5.2]{PappasRapoportShtukas}. 

\begin{Thm}[Pappas--Rapoport] \label{Thm:PRAxiomsFunctorial}
Let $\{\scrs_{U}\gxp\}_{U^p}$ be a canonical integral model of  $\{\mathbf{Sh}_{U}\gxp\}_{U^p}$. If $\mathcal{G}(\zpbr)= \iota^{-1}(\mathcal{G}'(\zpbr)) \cap G(\qpbr)$, then there is a canonical integral model $\{\scrs_K\gx\}_{K^p}$ of $\{\mathbf{Sh}_{K}\gx\}_{K^p}$ such that the morphism in \eqref{Eq:MorphismShimuraData} extends uniquely to a morphism
\begin{align}
    \iota:\scrs_{K}\gx \to \scrs_{U}\gxp \otimes_{\mathcal{O}_{E'}} \mathcal{O}_{E}
\end{align}
over $\spec(\mathcal{O}_E)$, such that the following diagram commutes
\begin{equation}
    \begin{tikzcd}
        \scrs_{K}\gx^{\diamondsuit/} \arrow{d}{\iota} \arrow{r}{\pi_{\mathrm{crys}, \mathcal{G}}} & \shtgmu \arrow{d} \\
        \scrs_{U}\gxp^{\diamondsuit/} \times_{\spd(\mathcal{O}_{E'})} \spd(\mathcal{O}_{E}) \arrow{r}\arrow{r}{\pi_{\mathrm{crys}, \mathcal{G}'}}&  \shtgmup \times_{\spd(\mathcal{O}_{E'})} \spd(\mathcal{O}_E).
    \end{tikzcd}
\end{equation}
\end{Thm}

\begin{proof}
This follows as in the proofs of \cite[Theorem 4.3.1, Theorem 4.5.2]{PappasRapoportShtukas}, with some small modifications as outlined below. \smallskip

We define $\mathscr{S}_{K}\gx$ to be the normalization of the Zariski closure of $\mathbf{Sh}_{K}\gx_{E}$ in $\mathscr{S}_{U}\gxp$ for all $K^p$. Axioms (i) and (ii) follow as in the proofs of \cite[Theorem 4.5.2]{PappasRapoportShtukas}. It remains to show that $\mathscr{S}_{K}\gx$ satisfies axioms (iii) and (iv). 

The assumption that $\mathcal{G}(\zpbr)= \iota^{-1}(\mathcal{G}'(\zpbr)) \cap G(\qpbr)$ implies that there is a natural map $\mathcal{G} \to \mathcal{G}'$ extending $G \to G'$ on the generic fiber, see \cite[Corollary 2.10.10]{KalethaPrasad}, which identifies $\mathcal{G}$ with the group smoothening of the Zariski closure $\overline{\mathcal{G}}$ of $G$ in $\mathcal{G}'$; this is explained in \cite[Section~3.6]{PappasRapoportShtukas}. 

Thus we obtain a commutative diagram
\begin{equation}
    \begin{tikzcd}
    \mathbf{Sh}_{K}\gx_{E}^{\diamondsuit} \arrow{d} \arrow{r} & \shtgmu \times_{\spd(\mathcal{O}_E)} \spd(E) \arrow{d} \\
    \mathbf{Sh}_{U}\gxp_{E}^{\diamondsuit} \arrow{r} & \operatorname{Sht}_{\mathcal{G}',\mu'} \times_{\spd(\mathcal{O}_E')} \spd(E).
    \end{tikzcd}
\end{equation}
By assumption the bottom horizontal arrow extends to a morphism $\mathscr{S}_{U^pU_p}\gxp^{\diamondsuit/} \to \operatorname{Sht}_{\mathcal{G}',\mu'} \times_{\spd(\mathcal{O}_E')} \spd \mathcal{O}_E$. We want to show that the top horizontal arrow extends (necessarily uniquely) to a morphism
\begin{align}
    \mathscr{S}_{K}\gx^{\diamondsuit/} \to \shtgmu
\end{align}
for all sufficiently small $K^p$. The existence of this extension of the $\mathcal{G}$-shtuka can be proved by following the argument in \cite[Section 4.6]{PappasRapoportShtukas}, taking into account the modifications to these arguments discussed in \cite[Section 4.8.1]{PappasRapoportShtukas}. Moreover we should take into account that \cite[Corollary 11.6]{AnschuetzExtension}, used in \cite[Lemma 4.6.6]{PappasRapoportShtukas}, has been extended to include quasi-parahoric group schemes, see \cite[Proposition 3.2.1, Proposition 3.2.2]{PappasRapoportRZSpaces}. A key point is that $\mathcal{G} \to \mathcal{G}'$ factors as $\mathcal{G} \to \overline{\mathcal{G}} \xhookrightarrow{} \mathcal{G}'$ where $\overline{\mathcal{G}} \xrightarrow{} \mathcal{G}'$ is a closed immersion and $\mathcal{G} \to \overline{\mathcal{G}}$ is a group smoothening, see \cite[Lemma 3.6.1]{PappasRapoportShtukas}. In particular, for a faithful representation $\mathcal{G}' \to \operatorname{GL}_{V,\zp}$, we may cut out both $\overline{\mathcal{G}} \subset \operatorname{GL}_{V,\zp}$ and $\mathcal{G}' \subset \operatorname{GL}_{V,\zp}$ by tensors by \cite[1.3.2]{KisinModels}. \smallskip 

We observe that the proof of \cite[Proposition 4.7.1]{PappasRapoportShtukas} goes through for quasi-parahoric $\mathcal{G}$ and with $k_E$ replaced by an arbitrary algebraically closed field $\ell$. The proof of Axiom (iv) in \cite[Section 4.7.1]{PappasRapoportShtukas} then applies to prove axiom (iv) for $\mathscr{S}_{K}\gx^{\diamondsuit/}$.
\end{proof}

\subsubsection{} We have the following version of
\cite[Theorem~4.2.4]{PappasRapoportShtukas}. Let $f \colon \gxg \to (\gp,
\mathsf{X}^\prime, \mathcal{G}^\prime)$ be a morphism of triples (meaning
Shimura data together with quasi-parahoric models) with induced inclusion
$\mathsf{E}' \subseteq \mathsf{E}$. Let $v \mid p$ be a place of $\mathsf{E}$
with induced place $v'$ of $\mathsf{E}'$, and let $E,E'$ be the respective
completions; write $K_p = \mathcal{G}(\zp)$ and $K_p' = \mathcal{G}'(\zp)$.
{
\def\gxgp{(\mathsf{G}', \mathsf{X}', \mathcal{G}')}
\def\ssgx{\mathscr{S}_K(\mathsf{G}, \mathsf{X})}
\def\ssgxp{\mathscr{S}_{K'}(\mathsf{G}', \mathsf{X}')}
\def\bsspp{\bar{\mathscr{S}}''}
\def\sspp{\mathscr{S}''}
\def\minttub{\widehat{\mathcal{M}^\mathrm{int}_{\mathcal{G},b_x,\mu,/x_0}}}
\def\minttubp{\widehat{\mathcal{M}^\mathrm{int}_{\mathcal{G}',b_{x'},\mu',/x_0'}}}

\begin{Prop} \label{Prop:PRAxiomsFunctorial}
  Assume there exist canonical integral models $\{\ssgx\}_{K^p}$ and
  $\{\ssgxp\}_{K^{p\prime}}$. Then for neat
  $K^p$ and $K^{p\prime}$ such that $f(K^p) \subseteq
  K^{p\prime}$, the map of Shimura varieties $\mathbf{Sh}_K\gx \to
  \mathbf{Sh}_K\gxp_{E}$ extends (necessarily uniquely) to a map of integral models
  \[
    \ssgx \to \ssgxp_{\mathcal{O}_{E}},
  \]
  and moreover there exists a (necessarily unique) $2$-commutative diagram
  \[ \begin{tikzcd}
    \ssgx^{\diamondsuit/} \arrow{r} \arrow{d}{\pi_{\mathrm{crys},\mathcal{G}}} &
    \ssgxp_{\mathcal{O}_{E}}^{\diamondsuit/} \arrow{d}{\pi_{\mathrm{crys},\mathcal{G}'}} \\
    \shtgmu \arrow{r} & \shtgmup \otimes_{\spd \mathcal{O}_{E'}} \spd \mathcal{O}_{E}
  \end{tikzcd} \]
  extending the natural one on the generic fiber.
\end{Prop}

\begin{proof}
The uniqueness of the morphism follows from the flatness and separatedness of the integral models, and the commutativity of the diagram follows from \cite[Corollary 2.7.10.]{PappasRapoportShtukas} and the existence of the analogous commutative diagram on the generic fiber. For the existence of the morphism, let us denote by $\bsspp$ the scheme-theoretic closure of the graph of $f$, 
  $\Gamma_f \subseteq \mathbf{Sh}_K\gx \times_{\spec E'} \mathbf{Sh}_{K'}\gxp$, inside $\ssgx \times_{\spec \mathcal{O}_E'} \ssgxp$. We define $\nu \colon \sspp \to \bsspp$ to be its normalization, so that we have maps
  \[
    \sspp \xrightarrow{\nu} \bsspp \hookrightarrow \ssgx \times_{\spec
    \mathcal{O}_{E'}} \ssgxp.
  \]
  We are going to show that the maps $\nu:\sspp \to \bsspp$ and $\bsspp \to \ssgx$ are isomorphisms, so that $\sspp \to \ssgx \times_{\spec
    \mathcal{O}_{E'}} \ssgxp$ is the graph of desired morphism. We note that the generic fiber of $\bsspp$ is isomorphic to $\mathbf{Sh}_K\gx$,
  which is already normal, and thus $\nu$ is an isomorphism over the generic
  fiber.\smallskip

  By definition of a canonical integral model, there exists a
  $\mathcal{G}$-shtuka $\mathscr{P}$ on $\ssgx^\diamond$ and also a
  $\mathcal{G}'$-shtuka $\mathscr{P}'$ on $\ssgxp^\diamond$. We consider their
  pullbacks along the morphisms
  \[
    \mathrm{pr}_1 \colon \bsspp \to \ssgx, \quad \mathrm{pr}_2 \colon \bsspp \to
    \ssgxp.
  \]
  We then obtain two $\mathcal{G}'$-shtukas
  \[
    \mathcal{G}' \times^{\mathcal{G}} (\nu^\ast \mathrm{pr}_1^\ast \mathscr{P}),
    \quad \nu^\ast \mathrm{pr}_2^\ast \mathscr{P}'
  \]
  on $\sspp^\diamond$, where the restriction of both $\mathcal{G}'$-shtukas to
  the generic fiber is naturally identified with the shtuka induced by the
  $K_p'$-local system on $\mathbf{Sh}_K\gx$. Using
  \cite[Corollary~2.7.10]{PappasRapoportShtukas}, we may extend the identification over the generic fiber uniquely to an
  isomorphism of $\mathcal{G}'$-shtukas
  \[
    \psi \colon \mathcal{G}' \times^{\mathcal{G}} (\nu^\ast \mathrm{pr}_1^\ast
    \mathscr{P}) \xrightarrow{\cong} \nu^\ast \mathrm{pr}_2^\ast \mathscr{P}'.
  \]

  Let $x'' \in \sspp(\fpbar)$ be an arbitrary point, where we implicitly choose
  an embedding $k_{E'} \to \fpbar$. Consider its images $\bar{x}'' = \nu(x'')
  \in \bsspp(\fpbar)$, $x = \mathrm{pr}_1(\bar{x}'') \in \ssgx(\fpbar)$, and
  $x' = \mathrm{pr}_2(\bar{x}'') \in \ssgxp(\fpbar)$. Denote by $\spf R_x$ and
  $\spf R_{x'}$ the formal completions of the closed points $x \in
  \ssgx_{\mathcal{O}_{\breve{E}}}$ and $x^\prime \in
  \ssgxp_{\mathcal{O}_{\breve{E}^\prime}}$ respectively, and similarly define $\spf R_{x''}$ and $\spf R_{\bar{x}''}$. By axiom (iv) of Definition \ref{Def:PRAxioms}, there exist
  framings of the shtukas $\mathscr{P} \vert_{\spd R_x}$ and $\mathscr{P}'
  \vert_{\spd R_{x'}}$ which induce isomorphisms
  \begin{equation}\label{Eq:framings}
    \spd R_x \xrightarrow{\cong} \widehat{\mintgbxmu{e}}_{/x_0}, \quad \spd R_{x'} \xrightarrow{\cong}
    \widehat{\mathcal{M}^{\mathrm{int}}_{\mathcal{G}',b_{x'},\mu',e'}}_{/x_0'}.
  \end{equation}
  Via the isomorphism $\psi$, we may identify the two basepoints and obtain a
  morphism
  \[
    g \colon \spd R_x \cong \widehat{\mintgbxmu{e}}_{/x_0} \to \widehat{\mathcal{M}^{\mathrm{int}}_{\mathcal{G}',b_{x'},\mu',e'}}_{/x_0'} \cong \spd R_{x'}
  \]
  by functoriality of integral local Shimura varieties. By \cite[Proposition~18.4.1]{ScholzeWeinsteinBerkeley}, this
  corresponds to a continuous ring homomorphism $R_{x'} \to R_x$\footnote{The normality
  of $R_x$ and $R_{x'}$ follow from normality of $\ssgx$ and $\ssgxp$ because
  base change along the ind-\'{e}tale maps $\mathcal{O}_E \to
  \mathcal{O}_E^\mathrm{unr}$ and $\mathcal{O}_{E'} \to
  \mathcal{O}_{E'}^\mathrm{unr}$ preserves normality, see \cite[Tag~033C, Tag~037D]{stacks-project}, and then normality passes further along formal completions by excellence, see \cite[Tag 0C23]{stacks-project}.}. \smallskip

  Note that the generic fiber $(\spf R_x)_\eta$ is naturally an open subset of
  the rigid analytic variety
  $\mathbf{Sh}_K\gx_{\breve{E}}^\mathrm{ad}$, and similarly for $(\spf
  R_{x'})_\eta$.

  \begin{Claim} \label{Claim:1}
   The diagram
  \[ \begin{tikzcd}[row sep=small]
    (\spf R_x)_\eta \arrow{r}{g} \arrow[hook]{d} & (\spf R_{x'})_\eta
    \arrow[hook]{d} \\ \mathbf{Sh}_K\gx_{\breve{E}}^\mathrm{ad} \arrow{r}{f} &
    \mathbf{Sh}_{K'}\gxp_{\breve{E}'}^\mathrm{ad}
  \end{tikzcd} \]
  commutes. 
  \end{Claim}

  \begin{proof}
  We first note that by \cite[Proposition~4.2.5]{PappasRapoportShtukas}, the
  isomorphisms \eqref{Eq:framings} may be chosen such that 
  \begin{itemize}
    \item[(1)] the framing of $\mathscr{P} \vert_{\spd R_x}$ pulled back to
      $\spd R_{x''}$ along $\mathrm{pr}_1 \circ \nu$ and pushed forward along
      $\mathcal{G} \to \mathcal{G}'$, and
    \item[(2)] the framing of $\mathscr{P}' \vert_{\spd R_{x'}}$ pulled back to
      $\spd R_{x''}$ along $\mathrm{pr}_2 \circ \nu$
  \end{itemize}
  agree under the identification of $\psi$. This shows that the diagram
  \[ \begin{tikzcd}
    \spd R_{x''} \arrow[equals]{r} \arrow{d}{\mathrm{pr}_1} & \spd R_{x''}
    \arrow{d}{\mathrm{pr}_2} \\ \spd R_x \cong \widehat{\mintgbxmu{e}}_{/x_0} \arrow{r}{g} & \widehat{\mathcal{M}^{\mathrm{int}}_{\mathcal{G}',b_{x'},\mu',e'}}_{/x_0'}
    \cong \spd R_{x'}
  \end{tikzcd} \]
  commutes. By \cite[Proposition~18.4.1]{ScholzeWeinsteinBerkeley}, the corresponding diagram of formal schemes also commutes, and this implies the commutativity of the following diagram:
  \begin{equation}
      \begin{tikzcd}
          (\spf R_{x''})_{\eta}  \arrow[equals]{r} \arrow{d} & (\spf R_{x''})_{\eta} \arrow{d} \\
           (\spf R_x)_\eta \arrow{r}{g} & (\spf R_{x'})_\eta.
      \end{tikzcd}
  \end{equation}
Observe that there is an open inclusion
\[
    (\spf R_{\bar{x}''})_\eta \subset ((\spf R_x)_\eta \times_{\spa \breve{E}'} (\spf
    R_{x'})_\eta) \cap \Gamma_f^\mathrm{ad}=:Y
  \]
and it follows similarly that $(\spf R_{x''})_{\eta}$ is open inside $Y$. It thus follows from the previous commutative diagram, that the diagram in the claim commutes when restricted to the open $(\spf R_{x''})_{\eta} \subset (\spf R_x)_\eta$. Since the locus where the two maps in the diagram in the claim agree is closed (the Shimura variety is separated) and contains the nonempty open $(\spf R_{x''})_{\eta} \subset (\spf R_x)_\eta$, we may conclude using the connectedness of $(\spf R_x)_\eta$ which follows from the normality of $R_x$, see \cite[Lemma~7.3.5]{deJongFormalRigid}.
  \end{proof}
Recall that we have a closed embedding
  \[
    \spf R_{\bar{x}''} \hookrightarrow \spf (R_x
    \operatorname{\widehat\otimes}_{\mathcal{O}_{\breve{E}'}} R_{x'}) = \spf R_x
    \times_{\spf \mathcal{O}_{\breve{E}'}} \spf R_{x'}
  \]
  of formal schemes, which on the generic fiber identifies with the graph of $g$. Now we have the following claim. 
  \begin{Claim}
  The closed sub-formal scheme
  $\spf R_{\bar{x}''}$ is equal to the graph of $g$, and $\spf R_{\bar{x}''}=\spf R_{x''}$.    
  \end{Claim}
  \begin{proof}
  The natural map
  \begin{align}
      \coprod_{x''}\spf R_{x''} \to \sspp \times_{\bsspp} \spf R_{\overline{x}''},
  \end{align}
  where the disjoint union is over all $x''$'s that map to $\overline{x}''$, is
  an isomorphism. In particular, the rigid fiber of the left hand side agrees
  with the rigid fiber of $\spf R_{\overline{x}''}$. Now both the graph of $g
  \colon \spf R_x \to \spf R_{x'}$ and $\coprod_{x''}\spf R_{x''}$ are affine
  flat normal formal schemes (by excellence of finite type schemes over
  $\mathcal{O}_{\breve{E}}$) mapping to $\spf R_x \times_{\spf
  \mathcal{O}_{\breve{E}'}} \spf R_{x'}$. Moreover, their generic fibers are
  the same Zariski closed subspace of the generic fiber of $\spf R_x
  \times_{\spf \mathcal{O}_{\breve{E}'}} \spf R_{x'}$ by Claim \ref{Claim:1}.
  The formal schemes are thus isomorphic (over $\spf R_x \times_{\spf
  \mathcal{O}_{\breve{E}'}} \spf R_{x'}$), because they can be recovered as
  the $\mathcal{O}^+$-sections of their (isomorphic) adic generic fibers, see
  \cite[Theorem 7.4.1]{deJongFormalRigid}. Since the graph of $g$ is closed
  inside of $\spf R_x \times_{\spf \mathcal{O}_{\breve{E}'}}\spf R_{x'}$, it
  follows that $\spf R_{x''} \to \spf R_x \times_{\spf
  \mathcal{O}_{\breve{E}'}} \spf R_{x'}$ is a closed immersion. Hence $\spf
  R_{x''} = \spf R_{\bar{x}''}$.
  \end{proof}
  Since $\nu:\sspp \to \bsspp$ is surjective, it follows that that complete
  local rings of $\bsspp$ are normal, which implies that $\bsspp$ is normal
  (because normality can be checked at completed local rings at closed points by
  \cite[Lemma~00MC, Lemma~033D]{stacks-project}) and thus that $\nu$ is an
  isomorphism. We have moreover shown that $\bsspp \to \ssgx$ induces
  isomorphisms $\spf R_{\bar{x}''} \xrightarrow{\sim} \spf R_{x}$ for all $x''$.
  Because $\mathscr{S}''(\fpbar) \to \bar{\mathscr{S}}''(\fpbar)$ is surjective,
  the birational map $\mathrm{pr}_1 \colon
  \bar{\mathscr{S}}''_{\mathcal{O}_{\breve{E}}} \to
  \ssgx_{\mathcal{O}_{\breve{E}}}$ induces isomorphisms on completed local
  rings. It is therefore a quasi-finite birational map between reduced separated
  Noetherian schemes, where the target is normal, and hence an open embedding by
  Zariski's main theorem.
  \smallskip

  To show that $\mathrm{pr}_1$ is an isomorphism, it now suffices to show that
  it is surjective on $\fpbar$-points. Given any $x \in
  \ssgx_{\mathcal{O}_{\breve{E}}}(\fpbar)$, we can first (by flatness) lift it
  to some $\mathcal{O}_F$-point $\tilde{x}$ for $F/\breve{E}$ a finite
  extension. We first note that the $K^p$-local system on the its generic
  point $\tilde{x}_\eta \in \ssgx(F)$ is trivial as it comes from a $K^p$ local
  system on $\spec \mathcal{O}_F$, and hence so is the
  $K^{p\prime}$-local system on $f(\tilde{x}_\eta)$. By the extension axiom~(i)
  of Definition~\ref{Def:PRAxioms}, this extends to an $\mathcal{O}_F$-point
  $\tilde{y}$ of $\ssgxp$. As $(\tilde{x}_\eta, \tilde{y}_\eta)$ is in the graph
  of $f$, its extension $(\tilde{x}, \tilde{y})$ is in the Zariski closure
  $\bar{\scrs}'' \cong \scrs$. Then the reduction $(x, y)$ is an $\fpbar$-point
  mapping under $\mathrm{pr}_1$ to $x$, completing the proof of surjectivity.
\end{proof}

The following corollary is a slight generalization of \cite[Theorem 4.2.4]{PappasRapoportShtukas} in the quasi-parahoric setting.

\begin{Cor} \label{Cor:PRUnique}
A canonical integral model $\{\scrs_K\gx\}_{K^p}$ of $\{\mathbf{Sh}_{K}\gx\}_{K^p}$ is unique up to unique isomorphism, if it exists.
\end{Cor}

\begin{proof}
Let $\{\ssgx\}_{K^p}$ and $\{\ssgx'\}_{K^{p}}$ be canonical integral models of $\{\mathbf{Sh}_{K}\gx\}_{K^p}$. Then by Proposition \ref{Prop:PRAxiomsFunctorial} applied to the identity map $f:\gxg \to \gxg$, there are unique maps
\[
  \ssgx \to \ssgx', \quad \ssgx' \to \ssgx
\]
extending the identity on the generic fiber. These are mutually inverse, because they are mutually inverse on the generic fiber and both $\ssgx$ and $\ssgx'$ are separated.
\end{proof}
}

\subsubsection{} The following theorem will be a crucial ingredient in the proof of Theorem \ref{Thm:IntroMain}. Let $\mathcal{G}$ be a quasi-parahoric model of $G$ and let $\mathcal{H} \subset \mathcal{G}$ be a quasi-parahoric subgroup (i.e. $\mathcal{H}^\circ=\calgcirc$). Let $K_p=\mathcal{G}(\zp)$ and $K_p'=\mathcal{H}(\zp)$. For $K^p \subset \gafp$ a sufficiently small compact open subgroup write $K=K^pK_p$ and $K'=K^pK_p'$. Assume for each $K^p$, we have a normal integral model $\scrs_{K}\gx$ of $\mathbf{Sh}_K\gx$ which is flat, separated and of finite-type over $\mathcal{O}_E$. We define
\begin{align}
    \scrs_{K'}\gx \to \scrs_{K}\gx
\end{align}
to be the relative normalization of $\mathscr{S}_K\gx$ in the composition
\[\mathbf{Sh}_{K^\prime}\gx \to \mathbf{Sh}_K\gx \to \mathscr{S}_K\gx.\]

\begin{Thm} \label{Thm:Devissage}
With the above construction, suppose $\{\scrs_K\gx\}_{K^p}$ is a canonical integral model for $\{\mathbf{Sh}_{K}\gx\}_{K^p}$. Then $\{\scrs_{K'}\gx\}_{K^p}$ is a canonical integral model for $\{\mathbf{Sh}_{K'}\gx\}_{K^p}$. Moreover, for each $K^p$ the map $\scrs_{K'}\gx \to \scrs_K\gx$ is finite \'{e}tale and the $2$-commutative diagram
\begin{equation}
    \begin{tikzcd}
        \scrs_{K'}\gx^{\diamondsuit/} \arrow{r} \arrow{d} & \shthmuone \arrow{d} \\
        \scrs_K\gx^{\diamondsuit/} \arrow{r} & \shtgmuone 
    \end{tikzcd}
\end{equation}
is $2$-Cartesian.
\end{Thm}

\begin{proof}
  We start by noting that axioms (i) and (ii) are a straightforward consequence
  of the corresponding axioms for $\{\scrs_K\gx\}_{K^p}$. Indeed, given a
  discrete valuation ring $R$ of mixed characteristic and a map $\spec R[1/p]
  \to \varprojlim_{K^p} \mathbf{Sh}_{K'}\gx$, we can extend its projection
  $\spec R[1/p] \to \varprojlim_{K^p} \mathbf{Sh}_{K}\gx$ to a map $\spec R \to
  \varprojlim_{K^p} \scrs_K\gx$. Since each $\scrs_{K'}\gx \to \scrs_K\gx$ is
  finite hence proper, we can use the valuative criterion to lift it to $\spec R
  \to \varprojlim_{K^p} \scrs_{K'}\gx$.

  For axiom (ii), we observe that for neat open compact subgroups $K^p_1
  \subseteq K^p_2$, the base change
  \begin{align*}
    \mathbf{Sh}_{K_p' K^p_1}\gx &= \mathbf{Sh}_{K_p' K^p_2}\gx
    \times_{\scrs_{K_p K^p_2}} \scrs_{K_p K^p_1} \\ &\to \scrs_{K_p' K^p_2}\gx
    \times_{\scrs_{K_p K^p_2}} \scrs_{K_p K^p_1} \to \scrs_{K_p K^p_1}
  \end{align*}
  is the composition of an open embedding followed by a finite map over a normal
  scheme $\scrs_{K_p K^p_1}$, and hence the relative normalization. This shows
  that $\scrs_{K_p' K^p_1} = \scrs_{K_p' K^p_2}\gx \times_{\scrs_{K_p K^p_2}}
  \scrs_{K_p K^p_1} \to \scrs_{K_p' K^p_2}$ is the base change of a finite
  \'{e}tale map, hence finite \'{e}tale.

  We next verify axiom (iii).
By Remark \ref{Rem:ShtukaOne}, the map $\scrs_K\gx^{\diamondsuit/} \to \shtgmu$ factors through $\shtgmuone$. The morphism
    \begin{align}
    \mathscr{Z}:=\scrs_{K}\gx^{\diamondsuit/} \times_{\shtgmuone} \shthmuone \to \scrs_K\gx^{\diamondsuit/}
    \end{align}
    is a torsor for the finite abelian group $\pi_0(\mathcal{G})^{\phi}/\pi_0(\mathcal{H})^{\phi}$ by Corollary \ref{Cor:QuasiParahoricShtukasIII}. If we base change to $\spd(E)$ and apply Corollary \ref{Cor:QuasiParahoricShtukasGeneric} twice, we obtain the Cartesian diagram
    \begin{equation}
    \begin{tikzcd}
                \mathscr{Z}_E \arrow{d} \arrow{r} & \left[ \operatorname{Gr}_{G, \mu^{-1}} / \underline{\mathcal{H}(\zp)} \right] \arrow{d} \\
        \mathbf{Sh}_K\gx^{\diamondsuit} \arrow{r} & \left[ \operatorname{Gr}_{G, \mu^{-1}} / \underline{\mathcal{G}(\zp)} \right].
    \end{tikzcd}
    \end{equation}
It follows from the construction of the bottom horizontal map, see Section \ref{Subsub:ShtukasGenericConstruction}, that this identifies $ \mathscr{Z}_E \to \mathbf{Sh}_K\gx^{\diamondsuit}$ with $\mathbf{Sh}_{K'}\gx^{\diamondsuit} \to \mathbf{Sh}_K\gx^{\diamondsuit}$. It now follows from Proposition \ref{Prop:FiniteEtaleCover} that $\scrs_{K'}\gx^{\diamondsuit/}$ is isomorphic to $\mathscr{Z}$. This shows that there is a morphism 
    \begin{align}
        \scrs_{K'}\gx^{\diamondsuit/} \to \shthmuone,
    \end{align}
    proving axiom (iii) and exhibits the Cartesian diagram in the statement of the theorem. Moreover, we see that $\scrs_{K'}\gx \to \scrs_{K}\gx$ is finite \'etale.

    For axiom (iv), we first observe that, by \cite[Proposition 4.2.1]{PappasRapoportRZSpaces} the natural map
    \begin{align}
        \widehat{\mintxgcircmu}_{/x_0} \to \widehat{\mintxgmu}_{/x_0} 
    \end{align}
    is an isomorphism, and the same is true for the natural map of formal completions of $\scrs_{K'}\gx$ and $\scrs_{K}\gx$, since $\scrs_{K'}\gx \to \scrs_{K}\gx$ is finite \'etale. Moreover, on both the integral local Shimura variety and the global integral model, the universal $\mathcal{G}$-shtuka pulls back to the pushforward along $\mathcal{G}^\circ \to \mathcal{G}$ of the universal $\mathcal{G}^\circ$-shtuka, by functoriality of the constructions. Thus axiom (iv) for $\{\scrs_{K'}\gx\}_{K^p}$ follows from axiom (iv) for $\{\scrs_{K}\gx\}_{K^p}$.
\end{proof}

\subsection{Integral models of Shimura varieties of Hodge type} \label{Sec:HodgeType}
For a symplectic space $(V, \psi)$ over $\mathbb{Q}$ we write $\gv=\mathrm{GSp}(V, \psi)$ for the group of symplectic similitudes of $(V,\psi)$ over $\mathbb{Q}$. It admits a Shimura datum $\mathcal{H}_V$ consisting of the union of the Siegel upper and lower half spaces. 
\begin{Lem} \label{Lem:SiegelAxioms}
    Conjecture \ref{Conj:PR} holds for any choice of parahoric $\mathcal{G}_V$ of $G_V$. 
\end{Lem}
\begin{proof}
    This is essentially a special case of \cite[Theorem 4.5.2]{PappasRapoportShtukas}. Our formulation of axiom (iv) is stronger, but the same proof works once we observe that the deformation theory for $p$-divisible groups as in the proof of \cite[Lemma 4.10.1]{PappasRapoportShtukas} works for arbitrary algebraically closed fields. 
\end{proof}

\subsubsection{Main results} \label{Sec:Main} Let $\gx$ be a Shimura datum of Hodge type with reflex field $\mathsf{E}$, let $p$ be a prime and write $G=\mathsf{G}_{\qp}$. Fix a place $v$ above $p$ of the reflex field $\mathsf{E}$, and let $E$ be the completion of $\mathsf{E}$ at $v$ with ring of integers $\mathcal{O}_E$ and residue field $k_E$. Let $\mathcal{H}$ be any quasi-parahoric integral model of $G$ and write $K_p'=\mathcal{H}(\zp)$. For any sufficiently small compact open subgroup $K^p \subset \gafp$ we will consider the Shimura variety $\mathbf{Sh}_{K'}\gx$ of level $K'=K^pK_p'$ as a scheme over $E$. The following is the main result of this paper and verifies Conjecture \ref{Conj:PR}.
\begin{Thm} \label{Thm:Main}
   There exists a canonical integral model $\{\scrs_{K'}\gx\}_{K^p}$ of $\{\mathbf{Sh}_{K'}\gx\}_{K^p}$. 
\end{Thm}
\begin{proof}
By Corollary~\ref{Cor:QuasiToStabilizer}, we may choose a stabilizer Bruhat--Tits group scheme $\mathcal{G}$ such that $\mathcal{H}$ is an open subgroup of $\mathcal{G}$; write $K_p=\mathcal{G}(\zp)$. It is explained in \cite[Section~1.3.2]{KMPS} that there exists a Hodge embedding $\iota:\gx \to \gvx$ and a $\zp$-lattice $V_{p} \subset V_{\qp}$ on which $\psi$ is $\zp$-valued, such that $\mathcal{G}(\zpbr)$ is the stabilizer in $G(\qpbr)$ of $V_p \otimes_{\zp} \zpbr$.
  In other words, if $\mathcal{G}_V$ is the parahoric integral model of $G_V$ over $\zp$ that is the stabilizer of $V_p$, then we have $\mathcal{G}(\zpbr)=G(\qpbr) \cap \iota^{-1}(\mathcal{G}_V(\zpbr))$. It now follows from Theorem~\ref{Thm:PRAxiomsFunctorial} and Lemma~\ref{Lem:SiegelAxioms} that there exists a canonical integral model $\{\scrs_{K}\gx\}_{K^p}$ of $\{\mathbf{Sh}_{K}\gx\}_{K^p}$, and the theorem is now a direct consequence of Theorem~\ref{Thm:Devissage}.
\end{proof}
\begin{Rem}
    It follows from the proof of Theorem \ref{Thm:Devissage} that the integral models of Theorem \ref{Thm:Main} are constructed as relative normalizations, as in \cite[Section 4.3]{KisinPappas} and \cite[Section 7.1.10]{KisinPappasZhou}. Thus our integral models agree with those constructed in \cite[Section 4.3]{KisinPappas} and \cite[Section 7.1.10]{KisinPappasZhou}.
\end{Rem}
\subsection{Local model diagrams and a conjecture of Kisin and Pappas} \label{Sec:LocalModel} Let the notation be as in Section \ref{Sec:ConjecturalModels}. In particular, $\mathcal{G}$ is a quasi-parahoric model of $G$. As in \cite[Section 4.9.1]{PappasRapoportShtukas}, we associate to $\mathcal{G}$ the v-sheaf $\mathcal{G}^\diamondsuit$. Explicitly, if $S = \spa(R,R^+)$ is in $\perf$, then $\mathcal{G}^\diamondsuit(S)$ consists of pairs $(S^\sharp, g)$, where $S^\sharp= \spa(R^\sharp, R^{\sharp +})$ is an untilt of $S$ and $g$ is an element of $\mathcal{G}(R^\sharp)$.

In \textit{loc. cit.}, Pappas and Rapoport show that for $S$ in $\perf$ and $f:S \to \shtgmu$ there is a $\mathcal{G}^{\diamondsuit}$-torsor $\tilde{S} \to S$ equipped with a $\mathcal{G}^{\diamondsuit}$-equivariant map $\tilde{S} \to \locmodgmuv$\footnote{Note that since $\mu$ is minuscule, the action of the positive loop group $\mathcal{L}^+\calg$ on $\locmodgmuv$ factors through $\calg_{}^\diamondsuit$. This defines the $\calg^\diamondsuit$-action on the v-sheaf local model.}. In other words, there is a morphism of stacks
\begin{align}
    \shtgmu \to \left[ \locmodgmuv / \mathcal{G}^{\diamondsuit} \right].
\end{align}
By construction, this morphism is functorial in $\mathcal{G}$ in the sense that, given a morphism $\alpha:\mathcal{G}_1 \to \mathcal{G}_2$ of quasi-parahoric group schemes, the diagram
\begin{equation}\label{Eq:local-model-commutativity}
    \begin{tikzcd}
        \operatorname{Sht}_{\mathcal{G}_1, \mu}
            \arrow[r] \arrow[d]
        & \left[ \mathbb{M}_{\mathcal{G}_1, \mu}^\mathrm{v} / \mathcal{G}_1^\diamondsuit  \right] 
            \arrow[d]
        \\ \operatorname{Sht}_{\mathcal{G}_2, \alpha \circ \mu}
            \arrow[r]
        & \left[ \mathbb{M}_{\mathcal{G}_2, \alpha \circ \mu}^\mathrm{v} / \mathcal{G}_2^\diamondsuit \right]
    \end{tikzcd}
\end{equation}
is 2-commutative. Here the vertical maps are obtained by functoriality of the constructions of $\shtg$ and v-sheaf local models.

\subsubsection{} \label{Subsub:SchemeTheoreticLocalModel}
By \cite[Theorem~1.11]{AGLR} and \cite[Corollary~1.4]{GleasonLourencoLocalModel}, there exists a unique (up to unique isomorphism) normal scheme $\locmodgmu$ that is flat and proper over $\mathcal{O}_{E}$ with reduced special fiber, whose associated v-sheaf is isomorphic to $\locmodgmuv$. A canonical integral model $\{\scrsg\}_{K^p}$ is said to have a \emph{scheme-theoretic local model diagram} if for all sufficiently small $K^p$ there is a smooth morphism of algebraic stacks
\begin{align}
    \pi_{\mathrm{dR}, \mathcal{G}} \colon \scrsg \to \left[\locmodgmu / \mathcal{G} \right],
\end{align}
whose generic fiber comes from the canonical model of the standard principal bundle (base-changed to $E$), see \cite[Theorem 4.1]{Milne2002}, together with a
2-commutative diagram
\begin{equation}
    \begin{tikzcd}
        \scrsg^{\diamondsuit/} \arrow{d}{\pi_{\mathrm{dR},\mathcal{G}}^{\diamondsuit/}} \arrow{r}{\pi_{\mathrm{crys}}} & \shtgmu \arrow{d} \\\left[\locmodgmuv / \mathcal{G}^{\diamondsuit/}\right] \arrow{r} & \left[\locmodgmuv / \mathcal{G}^{\diamondsuit}\right].
    \end{tikzcd}
\end{equation}
Now assume that $\gx$ is of Hodge type. Then by Theorem \ref{Thm:WhyDidWeCheckThis?}, the local model diagrams of \cite[Theorem 7.1.3]{KisinPappasZhou} give scheme-theoretic local model diagrams for $\gxg$, where $\mathcal{G}$ is a stabilizer Bruhat--Tits group scheme. We note that these results are stated under some additional assumptions on $\gxg$ and $p$ that we will make explicit in Section \ref{subsub:Assumptions}.

\subsubsection{} Let $\calgcirc \subset \calg$ be the relative identity component and write $K_p^{\circ}=\calgcirc(\zp)$ and $K_p=\calg(\zp)$. For $K^p \subset \gafp$ a sufficiently small compact open subgroup we write $K=K^pK_p$ and $K^{\circ}=K^pK_p^{\circ}$. Under the assumptions made in \cite[Theorem 4.2.7.]{KisinPappas}, Kisin and Pappas conjecture in \cite[Section 4.3.10]{KisinPappas}, that the composition
\begin{align}
    \scrs_{K^{\circ}}\gx \to \scrs_{K}\gx \to \left[\locmodgmu / \mathcal{G}\right]
\end{align}
factors through
\begin{align}
    \left[ \locmodgmu / \calgcirc \right] \to \left[\locmodgmu / \mathcal{G} \right].
\end{align}
The following proposition shows that such factorization exists, whenever a scheme-theoretic local model diagram exists for $\{\scrs_{K}\gx\}_{K^p}$. 

\begin{Prop} \label{Prop:KisinPappasConjecture}
    Suppose that $\{\scrs_{K}\gx\}_{K^p}$ admits a scheme-theoretic local model diagram $\pi_{\mathrm{dR}, \mathcal{G}}$. Then $\{\scrs_{K^{\circ}}\gx\}_{K^p}$ admits a scheme-theoretic local model diagram $\pi_{\mathrm{dR}, \calgcirc}$ such that for all (sufficiently small) $K^p$, the diagram
    \begin{equation}\label{Eq:LocalModelSquare}
    \begin{tikzcd}
     \scrs_{K^{\circ}}\gx \arrow{d} \arrow{r}{\pi_{\mathrm{dR}, \calgcirc}} & \left[\locmodgcircmu / \calgcirc\right] \arrow{d} \\
     \scrs_{K}\gx \arrow{r}{\pi_{\mathrm{dR}, \mathcal{G}}} & \left[\locmodgmu / \mathcal{G}\right].
    \end{tikzcd}
    \end{equation}
    commutes, where we identify $\locmodgcircmu=\locmodgmu$ via the isomorphism \eqref{Eq:local-model-comparison}.
\end{Prop}
To prove the proposition, we will need a lemma. As motivation, we recall from the proof of \cite[Proposition 3.2.1]{PappasRapoportRZSpaces} that there is a short exact sequence
    \begin{align} \label{Eq:SeS}
        1 \to \calgcirc \to \mathcal{G} \to i_{\ast} \pi_0(\mathcal{G}) \to 1
    \end{align}
    on the (big) \'etale site of $S=\scrs_{K^{\circ}}\gx$, where we view $\pi_0(\mathcal{G})$ as an \'{e}tale group scheme over $S_{k_E}:=\scrs_{K^{\circ}}\gx_{k_E}$, and $i$ is the closed immersion $i: S_{k_E}\hookrightarrow S$.
\begin{Lem} \label{Lem:ReductionOfStructures}
    Let $i:\spd(\fp) \to \spd(\zp)$ denote the inclusion. There is a diagram of short exact sequence of v-sheaves of groups over $\spd(\zp)$ 
    \begin{equation}
\begin{tikzcd}
     1 \arrow{r} & \mathcal{G}^{\circ, \diamondsuit/} \arrow{r} \arrow{d} & \mathcal{G}^{\diamondsuit/} \arrow{r} \arrow{d} & i_{\ast} \underline{\pi_0(\mathcal{G})} \arrow{r} \arrow[d, equals] & 1 \\
     1 \arrow{r} & \mathcal{G}^{\circ, \diamondsuit} \arrow{r} & \mathcal{G}^{\diamondsuit} \arrow{r} & i_{\ast} \underline{\pi_0(\mathcal{G})} \arrow{r} & 1 .
\end{tikzcd}
    \end{equation}
\end{Lem}

\begin{proof}
  Note that we can check exactness after base changing to $\spd(\zpbr)$. For
  surjectivity of $\mathcal{G}^{\diamondsuit/}_{\zpbr} \to (i_\ast
  \underline{\pi_0(\mathcal{G})})_{\zpbr}$, we observe that there is an open
  cover $\coprod_{g \in \pi_0(\mathcal{G})(\fpbar)} \spd(\zpbr) \to
  (i_\ast\underline{\pi_0(\mathcal{G})})_{\zpbr}$ and a section $\spec(\zpbr) \to
  \mathcal{G}_{\zpbr}$ for each $g \in \pi_0(\mathcal{G})(\fpbar)$. These induce
  sections $\spd(\zpbr) \to \mathcal{G}^{\diamondsuit/}_{\zpbr}$, and hence imply
  surjectivity of $\mathcal{G}^{\diamondsuit/}_{\zpbr} \to (i_\ast
  \underline{\pi_0(\mathcal{G})})_{\zpbr}$. The kernel of this map can be
  identified with $\mathcal{G}^{\circ,\diamondsuit/}_{\zpbr}$, because the zero
  section $\spd(\zpbr) \to (i_\ast\underline{\pi_0(\mathcal{G})})_{\zpbr}$ is an
  open embedding whose preimage in $\mathcal{G}^{\diamondsuit/}$ precisely recovers
  $\mathcal{G}^{\circ,\diamondsuit/}$. The proof of the exactness of the second row
  is identical.
\end{proof}

\begin{proof}[Proof of Proposition \ref{Prop:KisinPappasConjecture}]
The morphism $\pi_{\mathrm{dR}, \mathcal{G}}$ induces a $\mathcal{G}$-torsor $\mathcal{P}'$ over $\scrs_{K}\gx$; we will denote its pullback to $\scrs_{K^{\circ}}\gx$ by $\mathcal{P}$. From the short exact sequence \eqref{Eq:SeS}, we see that the pushout $\mathcal{P} \times^{\mathcal{G}} i_{\ast} \pi_0(\mathcal{G})$ is a torsor for the sheaf of abelian groups $i_{\ast} \pi_0(\mathcal{G})$. It suffices to construct a section of it over $S$. Indeed, given such a section, the pullback along this section of the natural map $\mathcal{P} \to \mathcal{P} \times^{\mathcal{G}} i_{\ast} \pi_0(\mathcal{G})$ gives a reduction of $\mathcal{P}$ to a $\calgcirc$-torsor. 

By the 2-commutativity of the diagram \eqref{Eq:local-model-commutativity} applied to $\mathcal{G}^\circ \to \mathcal{G}$, we have a reduction of $\mathcal{P}^{\diamondsuit}$ to a $\left(\calgcirc\right)^{\diamondsuit}$-torsor $\tilde{\mathcal{Q}} \subset \mathcal{P}^{\diamondsuit}$. This gives an $S^{\diamondsuit/}$-point of 
\begin{align}
    \mathcal{P}^{\diamondsuit} \times^{\mathcal{G}^{\diamondsuit}}  i_{\ast} \underline{\pi_0(\mathcal{G})} \cong \mathcal{P}^{\diamondsuit/} \times^{\mathcal{G}^{\diamondsuit/}} i_{\ast} \underline{\pi_0(\mathcal{G})} \cong \left(\mathcal{P} \times^{\mathcal{G}} i_{\ast} \pi_0(\mathcal{G}) \right)^{\diamondsuit/},
\end{align}
where we used Lemma \ref{Lem:ReductionOfStructures} for the first isomorphism. We want to show that this point is induced by an $S$-point of $\mathcal{P} \times^{\mathcal{G}} i_{\ast} \pi_0(\mathcal{G})$. \smallskip

We first observe that 
\begin{equation}\label{Eq:TorsorPushforward}
\mathcal{P} \times^{\mathcal{G}} i_{\ast} \pi_0(\mathcal{G}) = i_{\ast} \left(\mathcal{P}_{k_E} \times^{\mathcal{G}_{k_E}} \pi_0(\mathcal{G}) \right).
\end{equation}
From this it follows that
\begin{equation}\label{Eq:LozengeSlashSections}
    H^0\left(S^{\diamondsuit/}, \left(\mathcal{P} \times^{\mathcal{G}} i_{\ast} \pi_0(\mathcal{G}) \right)^{\diamondsuit/}\right) = H^0\left(S_{k_E}^{\diamond}, \left(\mathcal{P}_{k_E} \times^{\mathcal{G}_{k_E}} \pi_0(\mathcal{G}) \right)^{\diamond} \right).
\end{equation}
Let $S_{k_E}^{\mathrm{perf}}$ denote the perfection of $S_{k_E}$. It follows from the full-faithfulness of the functor $X \mapsto X^{\diamond}$ on perfect schemes, see \cite[Proposition 18.3.1]{ScholzeWeinsteinBerkeley}, that 
\begin{align}
    H^0\left(S_{k_E}^{\diamond}, \left(\mathcal{P}_{k_E} \times^{\mathcal{G}_{k_E}} \pi_0(\mathcal{G}) \right)^{\diamond} \right) &= H^0\left(S_{k_E}^{\mathrm{perf}}, \left(\mathcal{P}_{k_E} \times^{\mathcal{G}_{k_E}} \pi_0(\mathcal{G}) \right)^{\mathrm{perf}} \right).
\end{align}
By topological invariance of the \'etale site, the right hand side identifies with
$H^0 (S_{k_E}, \mathcal{P}_{k_E} \times^{\mathcal{G}_{k_E}} \pi_0(\mathcal{G}))$.
But from \eqref{Eq:TorsorPushforward} and the definition of $i_\ast$, we have
\begin{align}
H^0\left(S_{k_E}, \mathcal{P}_{k_E} \times^{\mathcal{G}_{k_E}} \pi_0(\mathcal{G})\right) = 
H^0\left(S, i_{\ast} \left(\mathcal{P}_{k_E} \times^{\mathcal{G}_{k_E}} \pi_0(\mathcal{G}) \right) \right). 
\end{align}
By combining these bijections, we obtain from the $(\mathcal{G}^\circ)^\diamondsuit$-torsor $\tilde{\mathcal{Q}}$ an $S$-point of $\mathcal{P} \times^{\mathcal{G}} i_{\ast}\pi_0(\mathcal{G}) $, i.e., a $\mathcal{G}^\circ$-torsor $\mathcal{Q}$. Clearly $\mathcal{Q}^{\diamondsuit} \cong \tilde{\mathcal{Q}}$, since both are the pullback of $\mathcal{P}^\diamondsuit$ along the same section of $\mathcal{P}^{\diamondsuit} \times^{\mathcal{G}^{\diamondsuit}}  i_{\ast} \underline{\pi_0(\mathcal{G})}$ over $S^\diamondsuit$.

Thus we obtain the desired morphism $\pi_{\mathrm{dR},\mathcal{G}^\circ}$, and it follows from the construction that \eqref{Eq:LocalModelSquare} commutes. That $\pi_{\mathrm{dR},\mathcal{G}^\circ}$ recovers the canonical model of the standard principal bundle on the generic fiber follows from the corresponding fact for $\pi_{\mathrm{dR},\mathcal{G}}$. Finally, it remains to show that $\pi_{\mathrm{dR}, \mathcal{G}^\circ}$ is smooth. This can be checked after pullback to the smooth cover $\locmodgmu \to \left[\locmodgmu / \calgcirc\right]$, but here the map is given by 
\[\mathcal{Q}\hookrightarrow\mathcal{P}\xrightarrow{\pi_{\mathrm{dR},\mathcal{G}}|_{\mathcal{P}}} \locmodgmu.\]
While the first map is an open immersion and the second map is smooth by assumption, the composition is smooth. This concludes the proof that $\pi_{\mathrm{dR},\mathcal{G}}$ is a scheme-theoretic local model diagram for $\scrs_{K^\circ}\gx$.
\end{proof}

\subsubsection{} \label{subsub:KisinPappasZhou} We now combine Proposition \ref{Prop:KisinPappasConjecture} with Theorem \ref{Thm:WhyDidWeCheckThis?} to deduce the existence of scheme-theoretic local model diagrams for Shimura varieties of Hodge type at parahoric level. 
\begin{Thm} \label{Thm:SchemeTheoreticLocalModelExists}
     If $(\mathsf{G}, \mathsf{X}, \calgcirc)$ is a triple of Hodge type with $\calgcirc$ a parahoric which is the identity component of a stabilizer parahoric $\mathcal{G}$ with the property that $(\mathsf{G}, \mathsf{X}, \calg)$ satisfies assumptions (A),(B),(C) of Section \ref{subsub:Assumptions}, then $\scrs_{K^{\circ}}\gx$ admits a scheme-theoretic local model diagram. 
\end{Thm}
\begin{proof}
Let $\mathcal{G}$ be as in the statement of the theorem. Under our assumptions Theorem \ref{Thm:WhyDidWeCheckThis?} applies to produce a scheme-theoretic local model diagram for $\scrs_{K}\gx$. The result is now a consequence of Proposition \ref{Prop:KisinPappasConjecture}.
\end{proof}
The following two remarks show that there are many cases in which the assumptions of Theorem \ref{Thm:SchemeTheoreticLocalModelExists} are satisfied.
\begin{Rem} \label{Rem:AssumptionsI}
Let $(\mathsf{G}_2, \mathsf{X}_2)$ be a Shimura datum of abelian type with reflex field $\mathsf{E}_{2}$, let $v_{2}$ be a finite place of $\mathsf{E}_2$ above a rational prime $p$ and let $\calgcirc_2$ be a parahoric. If $p>2$, then by \cite[Proposition 7.2.18]{KisinPappasZhou} there is a Hodge type Shimura datum $\gx$ together with a central isogeny $\mathsf{G}^{\mathrm{der}} \to \g_2^{\mathrm{der}}$ inducing an isomorphism of adjoint Shimura data $\gx^{\mathrm{ad}}\to (\mathsf{G}_2, \mathsf{X}_2)^{\mathrm{ad}}$, and there is moreover a parahoric $\calgcirc$ of $G$ associated to $\calgcirc_2$ for which the assumptions of Theorem \ref{Thm:SchemeTheoreticLocalModelExists} hold. Using this, we expect it is possible to use Theorem \ref{Thm:SchemeTheoreticLocalModelExists} to construct scheme-theoretic local model diagrams for all $(\mathsf{G}_2, \mathsf{X}_2, \calgcirc_{2}, v_2)$ of abelian type, as long as $p>2$.
\end{Rem}

\begin{Rem} \label{Rem:AssumptionsII}
Let $(\mathsf{G}, \mathsf{X}, \calgcirc)$ be a triple of Hodge type. If $p$ is coprime to $2 \cdot \pi_1(\g^{\mathrm{der}})$, if $G$ splits over a tamely ramified extension, and if $G$ is non-exceptional in the sense of \cite[Section 6.1]{KisinPappasZhou}, then there is a quasi-parahoric model $\mathcal{G}$ such that: The identity component of $\mathcal{G}$ is $\calgcirc$, the triple $(\mathsf{G}, \mathsf{X}, \calg)$ satisfies assumptions (A),(B),(C) of Section \ref{subsub:Assumptions}. Indeed, take $\mathcal{G}$ to be the stabilizer quasi-parahoric associated to a generic point in the facet corresponding to $\calgcirc$ and apply \cite[Theorem 6.1.9]{KisinPappasZhou} and the fact that groups splitting over tamely ramified extensions are automatically $R$-smooth.
\end{Rem}

\subsection{Rapoport--Zink uniformization} Let the notation be as in Section \ref{Sec:LocalModel}. In particular, $\mathcal{G}$ is a stabilizer Bruhat--Tits model of $G$ over $\zp$ with $K_p=\mathcal{G}(\zp)$. Denote by $k_E$ the residue field of $\mathcal{O}_E$ as before. For $\ell$ an algebraically closed field in characteristic $p$ together with a fixed embedding
$e \colon k_E \hookrightarrow \ell$, write
\[
  \oeel{e} = \mathcal{O}_E \otimes_{W(k_E),e} W(\ell)
\]
as before. Then for $b \in G(W(\ell)[1/p])$ we have the v-sheaves $\mintgbmu{e}, \mintgcircbmu{e}$ and we will also consider the image $\mintgbmuone{e}$ of $\mintgcircbmu{e} \to\mintgbmu{e}$.

Let $x \in \scrs_{K_p}\gx(\ell)$, then its image under
$\pi_\mathrm{crys}$ defines a $\spd(\ell)$-point $b_x$ of $\shtgmuone$. Let $e:k_E \to \ell$ be the map corresponding to $x: \spec(\ell) \to \scrs_{K_p}\gx_{k_E} \to \spec(k_E)$, then attached to $x$ is a base point
  \[
    x_0 \colon \spd(\ell) \to \mintgbxmu{e},
  \]
  given by the $\spd(\ell)$-point of $\shtgmu$ corresponding to $\pi_{\mathrm{crys}}(x)$, see Remark \ref{Rem:BasePoint}. In fact, since $\pi_{\mathrm{crys}}(x) \in \shtgmuone$, our base point lies in $\mintgbxmuone{e}(\spd(\ell))$.

\begin{Thm} \label{Thm:Uniformization}
    If $\gx$ is of Hodge type, then there exists a uniformization map
\[
  \Theta_{\mathcal{G},x} \colon \mintgbxmuone{e} \to
  \scrs_{K_p}\gx_\oeel{e}^\diamond
\]
sending the base point $x_0$ to $x$, which restricts to an isomorphism
\[
  \Theta_{\mathcal{G},x} \colon \widehat{\mintgbxmuone{e}}_{/x_0}
  \xrightarrow{\cong} (\widehat{\scrs_{K_p}\gx_\oeel{e}}_{/x})^\diamondsuit.
\]
Moreover the composition of $\Theta_{\mathcal{G},x}$ with $\pi_{\mathrm{crys}}$ 
\begin{align}
    \mintgbxmuone{e} \to
  \scrs_{K_p}\gx_\oeel{e}^\diamond \to \shtgmu \times_{\spd(\CO_E)} \spd(\oeel{e})
\end{align}
is 2-isomorphic to the natural map of Lemma \ref{Lem:LocalUniformisation}.
\end{Thm}

\begin{Rem} \label{Rem:HamacherKimRemark}
    Under certain additional hypotheses on $\gx$, it is conjectured in \cite[Axiom A]{Hamacher-Kim} that (for $\ell=\ovfp)$ there should be a uniformization map $\mintgbxmu{e}(\ovfp) \to \scrs_{K_p}\gx(\ovfp)$. If $\Pi_{\mathcal{G}} \not=1$, then such a map cannot upgrade to a uniformization map as in Theorem \ref{Thm:Uniformization}. Indeed, the natural map $ \scrs_{K_p}\gx^{\diamond} \to \shtgmu$ factors through $\shtgmuone$ by Remark \ref{Rem:ShtukaOne}, while the natural map $\mintgbxmu{e} \to \shtgmu$ does not.
\end{Rem}

\begin{proof}
The proof of \cite[Corollary 6.3]{GleasonLimXu} goes through, with the following modification. In the notation of \cite[Section 3.4]{GleasonLimXu}, we have an isomorphism, where the right hand side is the local Shimura variety of level $\mathcal{G}(\zp)$ over $\spd(\ebreve)$ associated to $(G, b_x, \mu, \mathcal{G}(\zp))$,
\begin{align}
  \mintgbxmuone{e} &\times_{\spd(\oeel{e})} \spd(\oeel{e}[1/p]) \\ &\simeq \operatorname{Sht}_{G,b_x,\mu, \mathcal{G}(\zp)} \times_{\spd(\ebreve)} \spd(\oeel{e}[1/p])
\end{align}
see \cite[Theorem 4.5.1]{PappasRapoportRZSpaces}. This means that we can follow the construction in \cite[Corollary 3.11]{GleasonLimXu} to construct a map
\begin{align}
    G(\qp)/\mathcal{G}(\zp) \to \pi_0 \left(\mintgbxmuone{e}\right).
\end{align}
To show that this map is surjective, we use the commutative diagram
\begin{equation}
    \begin{tikzcd}
        G(\qp)/\calgcirc(\zp) \arrow{r} \arrow{d} & \pi_0\left(\mintgcircbxmu{e}\right) \arrow{d} \\
        G(\qp)/\mathcal{G}(\zp) \arrow{r} & \pi_0\left(\mintgbxmuone{e}\right),
        \end{tikzcd}
\end{equation}
and the surjectivity of $\mintgcircbxmu{e} \to \mintgbxmuone{e}$ and $G(\qp)/\calgcirc(\zp) \to \pi_0\left(\mintgcircbxmu{e}\right)$, see \cite[Corollary 3.11]{GleasonLimXu} for the latter. With this in mind, the rest of the proof of \cite[Corollary 6.3]{GleasonLimXu} goes through. 
\end{proof}

\begin{Cor} \label{Cor:Uniformization}
For $z \in \scrs_{K_p^{\circ}}\gx(\ell)$ with image $x \in  \scrs_{K_p}\gx(\ell)$, there is a uniformization map 
    \begin{align}
        \Theta_{\calgcirc, z}:\mintgcircbxmu{e} \to \scrs_{K_p^{\circ}}\gx_{\oeel{e}}^{\diamond}
    \end{align}
sending the base point $z_0$ to $z$, that restricts to an isomorphism
\[
  \Theta_{\calgcirc,x} \colon \widehat{\mintgcircbxmu{e}}_{/z_0}
  \xrightarrow{\cong} (\widehat{\scrs_{K_p^{\circ}}\gx_\oeel{e}}_{/z})^\diamondsuit.
\]
\end{Cor}
\begin{proof}
If we define $Y$ (and $\Theta_{\calgcirc, z}$) as the fiber product
\begin{equation} 
    \begin{tikzcd}
        Y \arrow{d} \arrow{r}{\Theta_{\calgcirc, z}} & \scrs_{K_p^{\circ}}\gx_{\oeel{e}}^{\diamond}  \arrow{d} \\
        \mintgbxmuone{e} \arrow{r}{\Theta_{\mathcal{G}, x}} & \scrs_{K_p}\gx_{\oeel{e}}^\diamond,
    \end{tikzcd}
\end{equation}
then by concatenating fiber product squares (see Theorem \ref{Thm:Devissage}) we get a fiber product diagram
\begin{equation}
    \begin{tikzcd}
        Y \arrow{r} \arrow{d} & \shtgcircmu \arrow{d} \\
        \mintgbxmuone{e} \arrow{r} & \shtgmuone.
    \end{tikzcd}
\end{equation}
It follows from the proof of Theorem \ref{Thm:QuasiParahoricShtukas} and Lemma \ref{Lem:LocalUniformisation} that $Y \to \mintgbxmuone{e}$ is isomorphic to $\mintgcircbxmu{e} \to\mintgbxmuone{e}$. To finish the proof, we need to show that the base point $z_0 \in \mintgcircbxmu{e}$ is mapped to $z$. Since $x_0 \in \mintgbxmuone{e}$ is mapped to $x$ under $\Theta_{\mathcal{G}, x}$, the identity $\Theta_{\calgcirc, z}(z_0)=z$ follows from the fact that $z_0$ is the unique lift of $x_0$ whose associated $\mathcal{G}^{\circ}$-shtuka is given by $\spd \ell \xrightarrow{z} \scrs_{K_p^{\circ}}\gx_{\oeel{e}}^{\diamond} \to \shtgcircmu$; this follows from the Cartesian diagram of Theorem \ref{Thm:Devissage}.
\end{proof}

\appendix 

{\section{On scheme-theoretic local model diagrams} \label{Appendix:A}
\def\lozenge{\diamondsuit}

In this appendix we flesh out the remark in \cite[Section 4.9.2]{PappasRapoportShtukas} that the local model diagrams of \cite[Theorem 7.1.3]{KisinPappasZhou} give scheme-theoretic local model diagrams for integral models of Hodge-type Shimura varieties. 

\subsection{Some rational \texorpdfstring{$p$}{p}-adic Hodge theory} \label{sub:rationalpadicHodge} Let $X$ be a smooth rigid space over a finite extension $E$ of $\qp$ with pro-\'etale site $X_{\text{pro\'et}}$ as in \cite[Definition 3.9]{ScholzePAdicHodge}. We will consider the period sheaves $\mathbb{B}_\mathrm{dR}^+, \mathbb{B}_\mathrm{dR}$ and $\mathcal{O}\mathbb{B}_\mathrm{dR}$, see \cite[Definition 6.1, Definition 6.8]{ScholzePAdicHodge}. Let $\mathbb{L}$ be a de Rham $\zp$-local system of rank $n$ on $X_{\text{pro\'et}}$. Associated to $\mathbb{L}$ is a filtered vector bundle with integrable connection $D_{\mathrm{dR}}(\mathbb{L})=(\mathcal{E}, \operatorname{Fil}^{\bullet}, \nabla)$ on $X_\text{\'et}$ satisfying Griffiths transversality, see \cite[Theorem 3.9]{LiuZhu}. We have two $\mathbb{B}_\mathrm{dR}^+$-lattices on $X_{\text{pro\'et}}$:
\begin{align}
\mathbb{M}:=\mathbb{L}\otimes_{\underline{\zp}}\mathbb{B}_\mathrm{dR}^+, \quad \text{ and } \quad \mathbb{M}_0:=(D_\mathrm{dR}(\mathbb{L})\otimes_{\CO_X}\CO\mathbb{B}_\mathrm{dR}^+)^{\nabla=0}.
\end{align}
Here for the construction of $\mathbb{M}_0$, we take flat sections for the induced connection $\nabla=\nabla_{D_\mathrm{dR}(\mathbb{L})}\otimes \mathrm{id}+\mathrm{id}\otimes\nabla_{\CO\mathbb{B}_\mathrm{dR}^+}$. Also, by \cite[Definition 2.6.4, Proposition 2.6.3, Proposition 2.5.1]{PappasRapoportShtukas}, there is an induced shtuka $\mathscr{V}_{\mathbb{L}}$ of rank $n$ on $X^{\lozenge}$. 

Let $S=\spa(R,R^+)$ be an affinoid perfectoid space of characteristic $p$ together with a map $f:S \to X^{\lozenge}$ corresponding to an untilt $S^{\sharp}$ and a map $f:S^{\sharp} \to X$. Note that by construction of $\mathscr{V}_{\mathbb{L}}$, the completion of $\mathscr{V}_{\mathbb{L}}$ (resp. $\mathrm{Frob}_S^\ast\mathscr{V}_{\mathbb{L}}$) along $S^\sharp$ is canonically identified with $f^\ast\mathbb{M}$ (resp. $f^\ast\mathbb{M}_0$), where these pullbacks are defined as in the proof of \cite[Proposition 2.6.3]{PappasRapoportShtukas}. We equip $\restr{\operatorname{Frob}_S^{\ast}\mathscr{V}_{\mathbb{L}}}{S^{\sharp}}$ with a decreasing filtration such that (the Tate twist can be ignored)
\[\mathrm{Fil}^{-i}(\restr{\operatorname{Frob}_S^{\ast}\mathscr{V}_{\mathbb{L}}}{S^{\sharp}})\coloneqq \mathbb{M}\cap \mathrm{Fil}^{i}(\mathbb{M}_0)/\mathbb{M}\cap\mathrm{Fil}^{i+1}(\mathbb{M}_0)(-i).\]
\begin{Lem} \label{Lem:DdR}
    There is a natural isomorphism between filtered vector bundles $D_{\mathrm{dR}}(\mathbb{L})_{S^{\sharp}}$ and $\restr{\operatorname{Frob}_S^{\ast}\mathscr{V}_{\mathbb{L}}}{S^{\sharp}}$. 
\end{Lem}
\begin{proof}
The underlying vector bundle of $D_\mathrm{dR}(\mathbb{L})|_{X_\text{pro\'et}}$ can be recovered from $\mathbb{M}_0$ by taking 0th graded piece, see the discussion after the proof of \cite[Lemma 7.7]{ScholzePAdicHodge}. Its filtration can be recovered from the relative position of $\mathbb{M}$ and $\mathbb{M}_0$, as explained in \cite[Proposition 7.9]{ScholzePAdicHodge}. Since $\restr{\operatorname{Frob}_S^{\ast}\mathscr{V}_{\mathbb{L}}}{S^{\sharp}}=\mathrm{gr}^0(\mathbb{M}_0)$, by comparing with the formula in \cite[Proposition 7.9]{ScholzePAdicHodge}, we see that it agrees with $D_\mathrm{dR}(\mathbb{L})|_{S^\sharp}$ as filtered vector bundles. This identification is moreover natural in $S$. 
\end{proof}

Note that Lemma \ref{Lem:DdR} gives us a $2$-commutative diagram of tensor functors (discarding the connection on $D_\mathrm{dR}(-)$)
\begin{equation}
    \begin{tikzcd}
    \{\text{de Rham $\zp$-local systems on } X_\text{pro\'et} \} \arrow{r}{\mathrm{PR}} \arrow{d}{D_{\mathrm{dR}}} &  \{\text{Shtukas on } X^{\lozenge} \} \arrow{d} \\
     \{\text{Filtered vector bundles on } X_\text{\'et}\} \arrow{r} & \{\text{Filtered vector bundles on } X^{\lozenge}\}.
    \end{tikzcd}
\end{equation}    
Here $\mathrm{PR}$ is the (exact) tensor functor of \cite[Definition 2.6.4]{PappasRapoportShtukas}, and the right vertical arrow takes a shtuka $\mathscr{V}$ on $X^{\lozenge}$ to $\restr{\operatorname{Frob}_S^{\ast}\mathscr{V}_{\mathbb{L}}}{S^{\sharp}}$ as in Lemma \ref{Lem:DdR}.

\subsubsection{} \label{subsub:PidR} Now suppose that $X=Z^{\mathrm{an}}$ for a smooth $E$-scheme $Z$, and that there is an abelian scheme $\pi: A \to Z$ of relative dimension $g$ such that $\mathbb{L}\coloneqq R^1\pi_{\ast,\text{pro\'et}}\underline{\zp}$. Then as explained in \cite[Example 2.6.2]{PappasRapoportShtukas},
\begin{align}
D_{\mathrm{dR}}(\mathbb{L}) \simeq (\mathcal{H}^1_{\mathrm{dR}}(A/X), \operatorname{Fil}^{\bullet}_{\operatorname{Hdg}}),
\end{align}
where $\mathcal{H}^1_{\mathrm{dR}}(A/X)$ denotes the first relative de Rham cohomology of $\pi$, equipped with its Hodge filtration $ \operatorname{Fil}^{\bullet}_{\operatorname{Hdg}}$. To be precise, there is a natural surjective map of vector bundles $\mathcal{H}^1_{\mathrm{dR}}(A/X) \to \operatorname{Lie} (A^\vee)$ with kernel $ \operatorname{Fil}^{1}_{\operatorname{Hdg}}$. Note that it follows from \cite[Proposition 2.2.3]{CaraianiScholzeCompact} that $\mathbb{M}_0 \subset \mathbb{M}$. We recall the element $\xi \in \mathbb{B}_\mathrm{dR}^+$ generating $\ker \theta$, see \cite[Section 6]{ScholzePAdicHodge}.
\begin{Lem}\label{Lem:DdRFiltration}
    We can identify $\xi\mathbb{M} \subset \mathbb{M}_0$ with the kernel of the map
    \begin{align}
        \mathbb{M}_0 \to \mathcal{H}^1_{\mathrm{dR}}(A/X) \to \operatorname{Lie} (A^\vee).
    \end{align}
\end{Lem}
\begin{proof}
Note that the Hodge filtration on $\mathcal{H}^1_{\mathrm{dR}}(A/X)$ only has two jumps. The lemma now follows from the explicit formula of the filtration above. 
\end{proof}

Let $\Lambda=\zp^{\oplus 2g}$ and let $P_{\Lambda}=\underline{\operatorname{Isom}}_X(\Lambda\otimes_{\zp} \CO_{X}, \mathcal{H}^1_{\mathrm{dR}}(A/X))$ be the frame bundle of $\mathcal{H}^1_{\mathrm{dR}}(A/X)$. Let $\operatorname{Gr}_{g, \Lambda}$ be the Grassmannian of rank $g$ quotients of $\Lambda$ considered as scheme over $\zp$\footnote{This is a partial flag variety for $\mathrm{GL}(\Lambda)$, not to be confused with the affine version appearing later. It equivalently parametrizes rank $g$ subspaces of $\Lambda^\vee\simeq \Lambda$. Hence the map below could also be interpreted as sending to $\operatorname{Lie} (A)\hookrightarrow \mathcal{H}_1^\mathrm{dR}(A/X)$. We will often not distinguish these two interpretations below, but we caution the reader that it is useful to sometimes switch between the two descriptions of $P_\Lambda$ to compare with various literature.}. Then there is a map of adic spaces 
\begin{align*}
    \pi_{\mathrm{dR}}:P_{\Lambda} \to \operatorname{Gr}_{g, \Lambda,E}^{\mathrm{an}}
\end{align*}
defined using the natural quotient map $\Lambda\otimes_{\zp} \CO_{X}\cong\mathcal{H}^1_{\mathrm{dR}}(A/X) \to \operatorname{Lie}(A^\vee)$. 

\subsubsection{} \label{subsub:BB} We adopt the notation from \ref{subsub:PidR} above. Let $\mathscr{V}$ be the vector bundle shtuka induced from $\mathbb{L}$. Then by Lemma~\ref{Lem:DdRFiltration}, $\mathscr{V}^\vee$ is minuscule of height $2g$ and dimension $g$ in the sense of \cite[Definition 2.2.2]{PappasRapoportShtukas} (note the inclusion relation there, which is why we switched to $\mathscr{V}^\vee$). By \cite[Lemma 2.4.4]{PappasRapoportShtukas}, we can think of it as a ${\operatorname{GL}(\Lambda)}$-shtuka bounded by the cocharacter $\mu_g=(1^{(g)}, 0^{(g)})$. Since ${\operatorname{GL}(\Lambda)}$ is a reductive group, we may identify the local model $\mathbb{M}_{{\operatorname{GL}(\Lambda)}, \mu_g}$ with the flag variety $\operatorname{Gr}_{g, \Lambda}$ of $g$-dimensional quotients of $\Lambda$ as $\CO_E$-schemes (see \cite[Example 4.12]{AGLR} and \cite[Proposition 19.4.2]{ScholzeWeinsteinBerkeley}). We consider the diamond associated to the local model $\mathbb{M}_{{\operatorname{GL}(\Lambda)}, \mu_g}^{\lozenge}$ as a closed subfunctor of the 
Beilinson--Drinfeld affine Grassmannian $\operatorname{Gr}_{{\operatorname{GL}(\Lambda)}}$ for ${\operatorname{GL}(\Lambda)}$.

On the generic fiber (base changed to $E$), the isomorphism $\mathbb{M}_{{\operatorname{GL}(\Lambda)}, \mu_g, E}^{\lozenge}\xrightarrow{\sim} \operatorname{Gr}_{g, \Lambda,E}^\lozenge$ is induced by the Bia\l{}ynicki-Birula map, see \cite[Proposition 19.4.2]{ScholzeWeinsteinBerkeley}.

\subsubsection{}\label{subsub:GLTorsors} Let the notation be as in Section~\ref{subsub:PidR}. We can define a (left) $\mathrm{GL}(\Lambda)^\lozenge$-torsor of trivializations $\mathcal{P}_\Lambda$ over $X^\lozenge$ via 
\[(S\to X^\lozenge)\mapsto \mathrm{Isom}_{\CO_{S^\sharp}}(\Lambda\otimes_{\zp} \CO_{S^\sharp},\restr{\frob_S^{\ast }\mathscr{V}}{S^{\sharp}}).\] 
Lemma \ref{Lem:DdR} shows that there is a natural isomorphism of $\mathrm{GL}(\Lambda)^\lozenge$-torsors over $X^\lozenge$ 
\begin{align}
   P^\lozenge_{\Lambda} \xrightarrow{\sim} \mathcal{P}_\Lambda.
\end{align}
Applying the construction of \cite[Section 4.9.1]{PappasRapoportShtukas}, we get a diagram
\[X^\lozenge\leftarrow \mathcal{P}_\Lambda\rightarrow \mathbb{M}_{{\operatorname{GL}(\Lambda)}, \mu_g,E}^{\lozenge}.\]
The right arrow is $\mathrm{GL}(\Lambda)^\lozenge$-equivariant, and following the notation in \textit{loc. cit.}, it is constructed by (locally on $S$) lifting a section of $\mathcal{P}_\Lambda$ to an isomorphism over $\widehat{S^\sharp}\coloneqq\spec(\hat{\CO}_{S\dot{\times}\spa \zp, S^\sharp})$, and then send it to the triple 
\[(S^\sharp, \mathscr{V}, \alpha:   \Lambda\otimes_{\zp} \CO_{\widehat{S^\sharp}\backslash S^\sharp}\simeq \frob^\ast\mathscr{V}|_{\widehat{S^\sharp}\backslash S^\sharp}\xrightarrow[\sim]{\phi_{\mathscr{V}}}\mathscr{V}|_{\widehat{S^\sharp}\backslash S^\sharp})\in \operatorname{Gr}_{\mathrm{GL}(\Lambda),E}(S).\]
Its image lies in the minuscule Schubert cell $\mathbb{M}_{{\operatorname{GL}(\Lambda)}, \mu_g, E}^{\lozenge}= \operatorname{Gr}_{\mathrm{GL}(\Lambda),\mu_g, E}$. By Lemma~\ref{Lem:DdR}, we have the following compatibility.

\begin{Prop} \label{Prop:EqualityFiltrations}
The diagram below commutes, where the vertical isomorphisms are the ones from Sections~\ref{subsub:GLTorsors} and \ref{subsub:BB}.
\[\begin{tikzcd}
    X^\lozenge \arrow[d,equal] & \mathcal{P}_\Lambda \arrow{l} \arrow[d,"\sim"]\arrow{r}& \mathbb{M}_{{\operatorname{GL}(\Lambda)}, \mu_g, E}^{\lozenge}\arrow{d}{\sim}\\
    X^\lozenge & P_\Lambda^\lozenge \arrow{r}\arrow{l}& \operatorname{Gr}_{g,\Lambda,E}^\lozenge
\end{tikzcd}\]
\end{Prop}

\subsubsection{} \label{subsub:Shimura} Now let $\gx$ be a Shimura datum with Hodge cocharacter $\mu$ and reflex field $\mathsf{E}$ satisfying \eqref{Eq:SV5}. Let $v |p$ be a place of $\mathsf{E}$, $E\coloneqq \mathsf{E}_v$, and $\mathcal{G}$ be a parahoric model of $G$ over $\qp$. Let $K_p=\mathcal{G}(\zp)$ and $K=K^pK_p$ for $K^p \subset \gafp$ a neat compact open subgroup. We now specialize the previous section to the situation that $X=\mathbf{Sh}_K\gx^{\mathrm{an}}$. Then there is a pro-\'etale $\mathcal{G}(\zp)$-torsor $\mathbb{P} \to X$ which is de Rham in the sense of \cite[Definition 2.6.5]{PappasRapoportShtukas}, see \cite[Section 4.1]{PappasRapoportShtukas}. This gives us an exact tensor functor
\begin{align}
     \mathbb{L}_p:\operatorname{Rep}_{\zp} \mathcal{G} &\to \{\text{de Rham $\zp$-local systems on } X_\text{pro\'et} \} \\
     W &\mapsto \mathbb{P} \times^{\underline{\calg(\zp)}} \underline{W}. 
\end{align}

The composition $D_{\mathrm{dR}} \circ \mathbb{L}_p$ defines a $G^\text{an}$-torsor $P$ on $X_\text{\'et}$ via the Tannakian formalism. It thus follows from Lemma \ref{Lem:DdR} that the ${G}^{\lozenge}$-torsor $P^\lozenge$ on $X^{\lozenge}$ is naturally isomorphic to the $\mathcal{G}^{\lozenge}$-torsor $\mathcal{P}_{\mathrm{PR}}$ induced by the $\mathcal{G}$-shtuka over $X^{\lozenge}$ coming from $\mathbb{L}_p$. We note that since the filtered vector bundles in the essential image of $D_\mathrm{dR}\circ \mathbb{L}_p$ are equipped with a decreasing filtration of type $\mu$, the torsor $P$ has a canonical reduction of structure group to the parabolic $P^\mathrm{std}_\mu$ attached to $\mu$\footnote{We follow the convention in \cite[Section 2.1]{CaraianiScholzeCompact} for parabolics attached to cocharacters, i.e.
\[ P_\mu^\mathrm{std} = \lbrace g \in G : \lim_{t \to \infty} \mu(t) g \mu(t)^{-1} \text{ exists} \rbrace.\] 
We alert the readers that this is called $P_\mu$ in \cite[Definition~19.4.1]{ScholzeWeinsteinBerkeley}, but opposite to $P_\mu$ in \cite{CaraianiScholzeCompact}.}, see \cite[Remark 4.1(i)]{LiuZhu}. Therefore it admits a map $P\rightarrow \mathcal{F}\ell_{G,\mu}\coloneqq(G/P^\mathrm{std}_\mu)_E^\mathrm{an}$. On the other hand, similar to what we have explained in Section~\ref{subsub:GLTorsors}, the construction in \cite[Section 4.9.1]{PappasRapoportShtukas} gives a map $\mathcal{P}_\mathrm{PR}\to \mathrm{Gr}_{G,E}$ with image in the Schubert cell $\mathrm{Gr}_{G,\mu, E}=\mathbb{M}_{\calg,\mu,E}^\lozenge$. Moreover, the following diagram is commutative,
\[\begin{tikzcd}
    X^\lozenge \arrow[d,equal] & \mathcal{P}_\mathrm{PR} \arrow{l} \arrow[d,"\sim"]\arrow{r}& \mathbb{M}_{\calg,\mu, E}^{\lozenge}\arrow{d}{\sim}\\
    X^\lozenge & P^\lozenge \arrow{r}\arrow{l}& \mathcal{F}\ell_{G,\mu},
\end{tikzcd}\]
where the rightmost vertical arrow is induced by the Bia\l{}ynicki-Birula map. Note that there is also an exact tensor functor (the canonical construction)
\begin{align}
\mathcal{L}:\operatorname{Rep}_{\zp} \mathcal{G} \to \{\text{Filtered vector bundles on } X\},
\end{align}
see \cite[Proposition 5.2.10]{DiaoLanLiuZhu}. It follows from \cite[Theorem 5.3.1]{DiaoLanLiuZhu} that there is a natural isomorphism of tensor functors
 \begin{align}
     \mathcal{L} \xrightarrow{\sim} D_{\mathrm{dR}} \circ \mathbb{L}_p.
 \end{align}\
Thus the $\mathcal{G}^{\lozenge}$-torsor on $X^{\lozenge}$ corresponding to $\mathcal{L}$ via the Tannakian formalism, is naturally isomorphic to the $\mathcal{G}^{\lozenge}$-torsor $\mathcal{P}_{\mathrm{PR}}$ induced by the $\mathcal{G}$-shtuka over $X^{\lozenge}$ coming from $\mathbb{L}_p$. This isomorphism is moreover compatible with filtrations and thus with the map to $\mathbb{M}_{\calg,\mu, E}^{\lozenge}$.

\subsection{Some integral \texorpdfstring{$p$}{p}-adic Hodge theory}
Let $S^\sharp=\spa(R^{\sharp}, R^{\sharp+})$ be an untilt of an affinoid perfectoid space $S= \spa(R,R^+)$ in characteristic $p$ and let $\mathfrak A$ be the completion of an abelian scheme over $R^{\sharp+}$ with associated $p$-divisible group $Y = \mathfrak{A}[p^\infty]$. By \cite[Theorem 17.5.2]{ScholzeWeinsteinBerkeley}, we can associate to $Y$ a finite free $W(R^+)$-module $M(Y)$ equipped with an isomorphism 
\begin{align}
    \phi_{M}:\phi^{\ast} M(Y)[1/\phi(\xi)] \xrightarrow{\sim} M(Y)[1/\phi(\xi)]
\end{align}
such that
\begin{align}
    M(Y) \subset \phi_{M}\left(\phi^\ast M(Y)\right) \subset \tfrac{1}{\phi(\xi)} M(Y).
\end{align}
Here $\xi$ is a generator of the kernel of $W(R^+) \to R^{\sharp+}$. Let $M(Y)^{\ast}$ denote the $W(R^+)$-linear dual of $M(Y)$, which we will equip with the isomorphism $\phi_{M^{\ast}}$ given by the inverse of the $W(R^+)$-linear dual of $\phi_{M}$. This is the (contravariant) prismatic Dieudonn\'e module of $Y$ and it satisfies
\begin{align}
    \xi M(Y)^{\ast} \subset \phi_{M^{\ast}}(\phi^{\ast} M(Y)^{\ast})\subset M(Y)^{\ast}.
\end{align}
By restriction along $S\dot{\times}\spa \zp\to \spec W(R^+)$, it gives rise to a minuscule vector bundle shtuka with one leg at $S^\sharp$.

\begin{Lem} \label{Lem:IntegralDeRhamComparison}
There is a canonical isomorphism 
\begin{align}
    M(Y)^{\ast} \otimes_{W(R^+)} R^{\sharp+} \xrightarrow{\sim} H^1_{\mathrm{dR}}(A/R^{\sharp+})
\end{align}
compatible with base change.
\end{Lem}
\begin{proof}
We may identify $M(Y)^{\ast}$ with the $\phi$-pullback of the relative prismatic cohomology of $\mathfrak{A}$ using \cite[Corollary 4.63, Proposition 4.49]{AnschuetzLeBras}.\footnote{In \cite[Proposition 4.49]{AnschuetzLeBras}, the Frobenius twist is hidden in the notation $\tilde{\xi}=\phi(\xi)$.} The comparison isomorphism now follows from \cite[Theorem 1.8.(3)]{BhattScholzePrisms}.
\end{proof}
\subsubsection{} \label{subsub:Compatibility} For a characteristic zero untilt $R^{\sharp}$, we want to compare the isomorphism of Lemma \ref{Lem:IntegralDeRhamComparison} with the isomorphism of Lemma \ref{Lem:DdR}. We will do this under the assumption that $\mathfrak{A}$ is the pullback of a formal abelian scheme $f:\mathfrak{B} \to \mathfrak{X}$ over a smooth formal scheme $\mathfrak{X}/\CO_K$ for some discrete valued field $K/\qp$ (which will be the case in our situation since the Siegel modular variety is smooth). Denote the special fiber of $\mathfrak{X}$ by $\mathfrak{X}_s$ and the rigid generic fiber of $\mathfrak{X}$ by $X$, and similarly for $\mathfrak{B}$. 

There is an $F$-isocrystal $\mathcal{E}$ on $\mathfrak{X}_s$ obtained by the contravariant Dieudonn\'e crystal of $\mathfrak{B}_s$. It is associated with the vector bundle with flat connection $(E,\nabla):=(R^1f_{\mathrm{dR},\ast}\CO_{B},\nabla_\mathrm{GM})$ ($\nabla_\mathrm{GM}$ denotes the Gauss-Manin connection) on $X$, in the sense of \cite[Proposition 2.17]{GuoReinecke}. The proof of Proposition 2.36.(i) in \textit{loc. cit.} shows that $E\otimes_{\CO_{X}}\CO\mathbb{B}^+_\mathrm{dR}$ equipped with the product connection is isomorphic to $(\mathbb{B}^+_\mathrm{dR}(\mathcal{E})\otimes_{\mathbb{B}^+_\mathrm{dR}}\CO\mathbb{B}^+_\mathrm{dR}, \mathrm{id}\otimes \nabla_{\CO\mathbb{B}^+_\mathrm{dR}})$. In particular, we have a natural identification of their horizontal sections 
\[\mathbb{B}^+_\mathrm{dR}(\mathcal{E})=(\mathbb{B}^+_\mathrm{dR}(\mathcal{E})\otimes_{\mathbb{B}^+_\mathrm{dR}}\CO\mathbb{B}^+_\mathrm{dR})^{\nabla=0}\cong (E\otimes_{\CO_{X}}\CO\mathbb{B}^+_\mathrm{dR})^{\nabla=0}=: \mathbb{M}_0.\]

On the other hand, under the prismatic--crystalline comparison \cite[Theorem 1.8(1)]{BhattScholzePrisms}, one computes that
\begin{align} \label{Eq:ComparisonIsomorphisms}
\mathbb{B}^+_\mathrm{dR}(\mathcal{E})(S^\sharp) &= \mathbb{A}_\mathrm{crys}(\mathcal{E})\otimes_{\mathbb{A}_\mathrm{crys}}\mathbb{B}^+_\mathrm{dR}(S^\sharp) \\&= R^1f_\mathrm{crys,\ast}\CO \otimes_{\mathbb{A}_\mathrm{crys}}\mathbb{B}^+_\mathrm{dR}(S^\sharp)\\ 
&\xrightarrow{\sim} \phi^\ast H^1_{\prism}(\mathfrak{A}/\prism_{R^{\sharp+}})\otimes_{W(R^+)} \mathbb{B}^+_\mathrm{dR}(S^\sharp)\\
&=M(Y)^\ast \otimes_{W(R^+)} B^+_\mathrm{dR}(R^\sharp),
\end{align}
where $\prism_{R^{\sharp+}}$ denotes the perfect prism $(W(R^+), \mathrm{ker}\theta=(\xi))$.\footnote{Note that the definition of the de Rham period sheaves in \cite{GuoReinecke} differs from ours by a Frobenius twist, see Definition 2.3, Warning 2.4 in \textit{loc. cit.}, but their arguments work verbatim.} Therefore we have a commutative diagram (cf. \cite[Lemma 2.18]{ImaiKatoYoucis})
\begin{equation}
    \begin{tikzcd}
         M(Y)^{\ast} \arrow{r} \arrow{d} &  M(Y)^{\ast} \otimes_{W(R^+)} B_{\mathrm{dR}}^{+}(R^{\sharp}) \arrow{d}{\sim} \\
         H^1_{\mathrm{dR}}(A/R^{\sharp}) & \arrow{l} \mathbb{M}_0(S^\sharp),
    \end{tikzcd}
\end{equation}
where the left vertical map is the one in Lemma \ref{Lem:IntegralDeRhamComparison} (composed with inverting $p$); the bottom horizontal map is as in Lemma \ref{Lem:DdRFiltration}; and the right vertical map comes from Equation \eqref{Eq:ComparisonIsomorphisms}.

\subsubsection{} \label{subsub:Compatibility2} The discussion above implies an integral version of the result in Section~\ref{subsub:GLTorsors}. Namely, if we assume $A\to X$ is the rigid generic fiber of a family of formal abelian schemes $\mathfrak{A}\to\mathfrak{X}$, for some smooth formal scheme $\mathfrak{X}$ over $\mathrm{Spf}\CO_E$, then we can construct a vector bundle shtuka $\mathscr{V}$ over $\mathfrak{X}^{\lozenge}$ by descending the relative (contravariant) prismatic Dieudonn\'e crystal of $\mathfrak{A}$ over integral perfectoids. Attached to it, there is a $\mathrm{GL}(\Lambda)^\lozenge$-torsor over $\mathfrak{X}^\lozenge$ of trivializations
\[\mathcal{P}_\Lambda: (S\to \mathfrak{X}^{\lozenge})\mapsto \mathrm{Isom}_{\CO_{S^\sharp}}(\Lambda\otimes_{\zp} \CO_{S^\sharp},\restr{\frob_S^{\ast }\mathscr{V}}{S^{\sharp}}).\] 
On the other hand, one can also consider the frame bundle $P_{\Lambda}$ of $\mathcal{H}^1_{\mathrm{dR}}(\mathfrak{A}/\mathfrak{X})$. Lemma \ref{Lem:IntegralDeRhamComparison} implies that the two $\mathrm{GL}(\Lambda)^\lozenge$-torsors  $P^\lozenge_{\Lambda}\times^{\mathrm{GL}_\Lambda^\diamond} \mathrm{GL}_\Lambda^\lozenge$ and $ \mathcal{P}_\Lambda$ are canonically isomorphic.\footnote{Note that if we apply big diamond ($\lozenge$) functor to $P_\Lambda$, we obtain a $\mathrm{GL}_\Lambda^\diamond$-torsor, because $P_\Lambda$ is a $p$-adic formal scheme, as opposed to a scheme.}

Repeating the procedure in Section~\ref{subsub:GLTorsors} we get a commutative diagram 
\[\begin{tikzcd}
    \mathfrak{X}^{\lozenge} \arrow[d,equal] & \mathcal{P}_\Lambda \arrow{l} \arrow[d,"\sim"]\arrow{r}& \mathbb{M}_{{\operatorname{GL}(\Lambda)}, \mu_g}^{\lozenge}\arrow{d}{\sim}\\
    \mathfrak{X}^{\lozenge} & P^\lozenge_{\Lambda}\times^{\mathrm{GL}_\Lambda^\diamond} \mathrm{GL}_\Lambda^\lozenge \arrow{r}\arrow{l}& \operatorname{Gr}_{g,\Lambda}^\lozenge.
\end{tikzcd}\]
By the discussion in Section \ref{subsub:Compatibility}, it is compatible with the one in Proposition~\ref{Prop:EqualityFiltrations} when passing to the generic fiber.

\subsection{Shimura varieties of Hodge type} We follow the notation in the proof of Theorem \ref{Thm:Main}. In particular, we have $\gxg$ with $\mathcal{G}$ being a stabilizer Bruhat--Tits group scheme and a Hodge embedding $\iota:\gx \to \gvx$. We may moreover take a $\zp$-lattice $\Lambda \subset V_{\qp}$ on which $\psi$ is $\zp$-valued, such that $\mathcal{G}(\zpbr)$ is the stabilizer in $G(\qpbr)$ of $\Lambda \otimes_{\zp} \zpbr$. The Hodge cocharacter is denoted by $\mu$.
\subsubsection{} \label{subsub:Assumptions} We now make the following assumptions, cf. those stated in \cite[Section 7.1.2]{KisinPappasZhou}.
\begin{enumerate}[label=(\Alph*)]
    \item The group scheme $\mathcal{G}$ is the stabilizer of a point $x \in \mathcal{B}(G,\qp)$ which is generic in its facet. 

    \item The group $G$ is $R$-smooth in the sense of \cite[Definition 2.10]{DanielsYoucis}, and $p$ is coprime to $2 \cdot \pi_1(G^{\mathrm{der}})$.
 
    \item The local Hodge embedding $\iota:\mathcal{G} \to \operatorname{GL}(\Lambda)$ is very good in the sense of \cite[Definition 5.2.5]{KisinPappasZhou}.
\end{enumerate}
These assumptions are often satisfied, see \cite[Section 6, 7.2]{KisinPappasZhou}. A very good Hodge embedding is in particular good, which means that the natural maps
\begin{align}
    \mathcal{G} \to {\operatorname{GL}(\Lambda)} \ \text{ and } \
    \locmodgmu \to \operatorname{Gr}_{g, \Lambda}
\end{align}
are closed immersions.\footnote{The local models used by \cite{KisinPappasZhou} agree with ours because theirs also satisfy the Scholze--Weinstein conjecture (which means that they have the correct associated v-sheaf), see \cite[Lemma 3.4.1]{KisinPappasZhou}.} 

\subsubsection{} Let $P_{\Lambda^\vee} \to \scrs_{K}\gx$ be the left ${\operatorname{GL}(\Lambda^\vee)}$-torsor parametrizing trivialisations $\Lambda^{\vee} \to \mathcal{H}^1_{\operatorname{dR}}(A)$\footnote{Equivalently, $P_{\Lambda^\vee}$ is the right $\operatorname{GL}(\Lambda)$-torsor of trivializations $\mathcal{H}_1^\mathrm{dR}(A)\to \Lambda$. Hence over $\mathbb{C}$, this is same as $\mathcal{G}_\mathrm{dR}$ in \cite[Page 19]{CaraianiScholzeCompact}, or Milne's principal bundle.}, where $A$ is the universal abelian variety (up to prime-to-$p$ isogeny) coming from $\iota$. Then there is a morphism $P_{\Lambda^\vee} \to \operatorname{Gr}_{g, \Lambda^\vee, \mathcal{O}_E}$ as in Section \ref{subsub:PidR}. By the proof of \cite[Theorem 7.1.3]{KisinPappasZhou} (which uses (A), (B), and (C) above), there is a left $\mathcal{G}$-torsor $P_{} \to \scrs_{K}\gx$ together with a $\mathcal{G}$-equivariant map $P_{} \to P_{\Lambda^\vee}$ such that the composition
\begin{align}
    P_{} \to P_{\Lambda^\vee} \to \operatorname{Gr}_{g, \Lambda^\vee,\CO_E}
\end{align}
factors through $ \locmodgmu$ via a smooth map. Altogether we have a diagram
\[\scrs_K\gx\leftarrow P_{}\rightarrow\locmodgmu,\]
where the left arrow is a $\calg$-torsor and the right arrow is smooth and $\calg$-equivariant. It moreover follows from the construction that its generic fiber comes from the canonical model of the standard principal bundle, see the discussion in the proof of \cite[Lemma 2.3.5]{CaraianiScholzeCompact}. Let us write 
\begin{align}
    \pi_{\mathrm{dR}, \mathcal{G}}\colon \scrs_K\gx \to \left[ \locmodgmu / \mathcal{G} \right]
\end{align}
for the induced smooth morphism of algebraic stacks.
\begin{Thm} \label{Thm:WhyDidWeCheckThis?}
If Assumptions (A),(B),(C) hold, then the morphism $\pi_{\mathrm{dR}, \mathcal{G}}$ is a scheme-theoretic local model diagram.
\end{Thm}

For the proof of Theorem \ref{Thm:WhyDidWeCheckThis?}, we will need the following two lemmas.

\begin{Lem} \label{Lem:RestrictToQp}
    Let $\mathscr{Y}$ be a v-sheaf which is separated over $\spd(\zp)$. For any normal scheme $X$ which is flat, separated and of finite-type over $\zp$, the natural restriction map
    \begin{align}
        \operatorname{Hom}_{\spd(\zp)}(X^{\lozenge/}, \mathscr{Y}) \to  \operatorname{Hom}_{\spd(\qp)}(X_{\qp}^{\lozenge}, \mathscr{Y}_\qp)
    \end{align}
    is injective.
\end{Lem}
\begin{proof}
    This follows from the density of $|X_{\qp}^{\lozenge}| \subset |X^{\lozenge/}|$, which in turn follows from the density of $|(X^{\diamond})_{\qp}| \subset | X^{\diamond}|$ (see \cite[Lemma 2.17]{AGLR}). 
\end{proof}

\begin{Lem}\label{Lem:Separated}
    The quotient v-sheaf $\mathrm{GL}(\Lambda)^\lozenge / \mathcal{G}^\lozenge$ is separated over $\spd(\zp)$. 
\end{Lem}
\begin{proof}
By \cite[Lemma 6.17]{PhungHodosSantos}, there is a finite free representation $W$ of $\operatorname{GL}(\Lambda)$ and a free rank-one saturated $\zp$-submodule $L \subset W$ for which $\mathcal{G}$ is the scheme-theoretic stabilizer of $L$ inside of $\mathrm{GL}(\Lambda)$. It follows that there is a morphism of v-sheaves
    \begin{equation}\label{Eq:Separated}
        \mathrm{GL}(\Lambda)^\lozenge / \mathcal{G}^\lozenge \to \mathbb{P}(W)^\lozenge,
    \end{equation}
    defined at the level of presheaves by $[g] \mapsto g\cdot L$. We claim this is a monomorphism. Indeed, suppose $S$ is a perfectoid space in characteristic $p$, and that $a,b \colon S \to \mathrm{GL}(\Lambda)^\lozenge / \mathcal{G}^\lozenge$ are two morphisms which agree after the composition to $\mathbb{P}(W)^\lozenge$. After replacing $S$ by a v-cover, we may assume $a$ and $b$ factor through morphisms $\tilde{a}$, $\tilde{b}\colon S \to \mathrm{GL}(\Lambda)^\lozenge$. Since $a$ and $b$ agree after the composition to $\mathbb{P}(W)^\lozenge$, and $\mathcal{G}$ is the stabilizer of $L$, it follows that $\tilde{a}\cdot\tilde{b}^{-1}$ factors through $\mathcal{G}^\lozenge$; thus $a = b$.

    Since \eqref{Eq:Separated} is a monomorphism, its diagonal is an isomorphism, and therefore \eqref{Eq:Separated} is separated. Now $\mathbb{P}(W)$ is proper over $\spec(\zp)$, so $\mathbb{P}(W)^\lozenge = \mathbb{P}(W)^\diamond$, and hence $\mathbb{P}(W)^\lozenge \to \spd(\zp)$ is separated by \cite[Proposition 4.17]{GleasonSpecialization}. The result follows.
\end{proof}

\begin{proof}[Proof of Theorem \ref{Thm:WhyDidWeCheckThis?}]
We start by identifying
\begin{align}
    P_{}^{\lozenge/} \times^{\mathcal{G}^{\lozenge/}} \mathcal{G}^{\lozenge} 
\end{align}
with the $\mathcal{G}^{\lozenge}$-torsor coming from the map $\scrs_{K}\gx^{\lozenge/} \to \shtgmu$. We first check that this holds after 
composing with $\shtgmu \to \operatorname{Sht}_{{\operatorname{GL}(\Lambda)}, \mu_g}$ and pushing out via
\begin{align}
    \mathcal{G}^{\lozenge} \to {\operatorname{GL}(\Lambda)}^{\lozenge}.
\end{align}
The latter result is true over $\mathbf{Sh}_K\gx^{\lozenge}$ by Proposition~\ref{Prop:EqualityFiltrations}. Moreover, by Sections \ref{subsub:Compatibility} and \ref{subsub:Compatibility2}, it is true over $\scrs_{K}\gx^{\diamond}$\footnote{The discussion there assumes that our (formal) abelian scheme comes via pullback from a (formal) abelian scheme over a smooth (formal) scheme. This assumptions holds here since $\Lambda$ is self dual and thus the integral model of the Shimura variety for $\gvx$ is smooth.}, such that induced isomorphisms agree on $\left(\scrs_{K}\gx^{\diamond}\right)_E$, so they glue to an isomorphism over $\scrs_{K}\gx^{\lozenge/}$. 

Next, we check that the induced $\mathcal{G}^{\lozenge}$-torsors agree: after trivializing the induced ${\operatorname{GL}(\Lambda)}$-torsor, which we may do Zariski locally on $\scrs_{K}\gx$, we are trying to show the equality of two morphisms $\scrs_{K}\gx^{\lozenge/} \to {\operatorname{GL}(\Lambda)}^{\lozenge}/\mathcal{G}^{\lozenge}$. By Lemma \ref{Lem:Separated} and Lemma \ref{Lem:RestrictToQp}, it suffices to check this after base change to the generic fiber, where the result follows from the discussion in Section \ref{subsub:Shimura}.

Finally, we check that, under this identification, the maps to the local model agree. But $\mathbb{M}_{\mathcal{G},\mu}$ is projective and hence proper over $\spec(\mathcal{O}_E)$, so $\mathbb{M}_{\mathcal{G},\mu}^\lozenge \to \spd(\mathcal{O}_E)$ is separated by \cite[Proposition 4.17]{GleasonSpecialization}. By another application of Lemma \ref{Lem:RestrictToQp}, it suffices to check the morphisms agree after base change to the generic fiber, where the result follows from the discussion in Section~\ref{subsub:Shimura}. Together with the results implied by \cite[Theorem 7.1.3]{KisinPappasZhou} discussed above, this concludes the proof that $\pi_{\mathrm{dR}, \mathcal{G}}$ is a scheme-theoretic local model diagram.
\end{proof}
}

\renewcommand{\VAN}[3]{#3}
\bibliographystyle{amsalpha}
\bibliography{references}

\end{document}